\documentclass{article}
\pdfoutput=1
\usepackage[utf8]{inputenc}
\usepackage{fixltx2e}
\usepackage{stmaryrd}   
\usepackage{microtype}
\usepackage[]{amsmath}
\usepackage{amsthm}
\usepackage{latexsym}
\usepackage{amssymb}
\usepackage{mathrsfs}
\usepackage[american]{babel}
\usepackage{textcomp}
\usepackage[all,ps,tips,tpic]{xy}
\usepackage{upgreek}
\usepackage{url}
\usepackage{booktabs}
\usepackage{enumitem}
\usepackage[final,pdftex,colorlinks=false,pdfborder={0 0 0}]{hyperref}
\usepackage{tikz}
\usetikzlibrary{matrix,arrows,calc,positioning,patterns}
\input{ullmann-controlled-simplicial-rings-intervals-tikz.texinc}
\usepackage[mathscr]{euscript}
\usepackage{exscale}

\usepackage[draft]{fixme}

\hypersetup{pdfauthor={Mark Ullmann},pdftitle={Controlled Algebra for
Simplicial Rings and Algebraic K-theory},pdfkeywords={simplicial rings,
algebraic K-theory, controlled algebra, Farrell-Jones conjecture}}

\usepackage{remreset}
\makeatletter \@removefromreset{equation}{section} \makeatother

\DeclareMathOperator{\HOM}{HOM}
\newcommand{\inc}{\mathrm{inc}}
\newcommand{\sset}{\mathit{sSet}}

\newcommand{\Ab}{\mathit{Ab}}
\newcommand{\Nat}{\mathbb{N}}
\newcommand{\Int}{\mathbb{Z}}
\newcommand{\Real}{\mathbb{R}}

\newtheorem{lem}[subsubsection]{Lemma}
\newtheorem{definition}[subsubsection]{Definition}
\newtheorem{lemdef}[subsubsection]{Lemma/Definition}

\newtheorem{cor}[subsubsection]{Corollary}
\newtheorem{thm}[subsubsection]{Theorem}
\newtheorem{prop}[subsubsection]{Proposition}

\newtheorem*{thm*}{Theorem}
\newtheorem*{lem*}{Lemma}

\makeatletter
\newtheorem*{rep@theorem}{\rep@title}
\newcommand{\newreptheorem}[2]{%
\newenvironment{rep#1}[1]{%
 \def\rep@title{#2 \ref{##1}}%
 \begin{rep@theorem}}%
 {\end{rep@theorem}}}
\makeatother

\newtheorem{mainthm}{Theorem}
\newreptheorem{mainthm}{Theorem}
\newreptheorem{lem}{Lemma}
\newreptheorem{thm}{Theorem}

\theoremstyle{remark}
\newtheorem{rem}[subsubsection]{Remark}
\newtheorem{example}[subsubsection]{Example}

\newcommand{\id}{\mathrm{id}}

\DeclareMathOperator{\pr}{pr}

\DeclareMathOperator{\Hom}{Hom}

\DeclareMathOperator*{\colim}{colim}

\DeclareMathOperator{\supp}{supp}
\DeclareMathOperator{\Tel}{Tel}
\DeclareMathOperator{\Tr}{Tr}
\DeclareMathOperator{\Imag}{Im}

\DeclareMathOperator{\res}{res}

\newcommand{\sh}{\mathrm{sh}}

\newcommand{\op}{\mathrm{op}}

\newcommand{\Aut}{\mathrm{Aut}}
\newcommand{\bdot}{\textrm{\Large\bf.}}
\newcommand{\pt}{\mathrm{pt}}
\newcommand{\Ar}{\mathrm{Ar}}
\newcommand{\sk}{\mathrm{sk}}

\newdir{ >}{{}*!/-5pt/@{>}}
\newdir{ (}{{}*!/-5pt/@^{(}}
\SelectTips{cm}{10}

\author{Mark Ullmann}
\title{Controlled Algebra for Simplicial Rings and Algebraic K-theory}

\makeatletter
\renewcommand\subsubsection{\@startsection{subsubsection}{3}{\z@}%
                                    {3.25ex \@plus1ex \@minus.2ex}%
                                    {-1em}%
                                    {\normalfont\normalsize\bfseries}}
\makeatother

\begin{document}
\maketitle
\begin{abstract}
We develop a version of controlled algebra for simplicial rings. This
generalizes the methods which lead to successful proofs of the algebraic K-
theory isomorphism conjecture (Farrell-Jones Conjecture) for a large class of
groups.  This is the first step to prove the algebraic K-theory isomorphism
conjecture for simplicial rings.  We construct a category of controlled
simplicial modules, show that it has the structure of a Waldhausen category
and discuss its algebraic K-theory.

We lay emphasis on detailed proofs.  Highlights include the discussion of a
simplicial cylinder functor, the gluing lemma, a simplicial mapping telescope
to split coherent homotopy idempotents, and a direct proof that a weak
equivalence of simplicial rings induces an equivalence on their algebraic
K-theory.  Because we need a certain cofinality theorem for algebraic
K-theory, we provide a proof and show that a certain assumption, sometimes
omitted in the literature, is necessary.  Last, we remark how our setup
relates to ring spectra.
\end{abstract}
\pdfbookmark[1]{\contentsname}{Contents}
\tableofcontents
\section{Introduction}
Controlled algebra is a powerful tool to prove statements about the algebraic
K-theory of a ring $R$.  While early on it was used in
\cite{Pedersen-Weibel;nonconnective} to construct a nonconnective delooping
of $K(R)$---a space such that $\pi_i(K(R)) = K_{i-1}(R)$---it is a crucial
ingredient in recent progress of the Farrell-Jones Conjecture
(cf.~Section~\ref{subsec:farrell-jones}).  Our
aim here is to construct for a simplicial ring $R$, and a ``control
space'' $X$, a category of ``controlled simplicial $R$-modules over a $X$''.
It should be regarded as a generalization of controlled algebra from rings to
simplicial rings.

The category of
``controlled simplicial modules'' supports a homotopy theory which is formally
very similar to the homotopy theory of CW-complexes.  In particular, we have a
``cylinder object'' which yields a notion of homotopy and therefore the
category has homotopy equivalences.  Waldhausen nicely
summarized a minimal set of axioms to do homotopy theory, 
in~\cite{Waldhausen;spaces;;;;1985} he called it
a ``category with cofibrations and weak equivalences'', later authors used
the term \emph{Waldhausen category}.  He did this to define algebraic
K-theory of such a category.  Our category satisfies Waldhausen's axioms,
which is our main result: 
\begin{mainthm}\label{mainthm:c-waldh-cat}
  Let $X$ be a control space and $R$ a simplicial ring.  The category of
  controlled simplicial modules over $X$, $\mathcal{C}(X;R)$, together with
  the homotopy equivalences and a suitable class of cofibrations is a
  ``category with cofibrations and weak equivalences'' in the sense of
  Waldhausen (\cite{Waldhausen;spaces;;;;1985}).  Therefore, Waldhausen's
  algebraic K-theory of $\mathcal{C}(X;R)$ is defined.

  The category has a cylinder functor and it satisfies Waldhausen's cylinder
  axiom, his saturation axiom and his extension axiom.
\end{mainthm}

In fact, for $G$ a group (for us all groups are discrete), there is a
$G$-equivariant version of
$\mathcal{C}(X;R)$ and of this theorem, which is crucial for applications to
the Farrell-Jones Conjecture.
The $G$-equivariant version of the theorem
is not more difficult to prove than its non-equivariant counterpart.  It is
stated in Section~\ref{subsec:waldhausen-category} as
Theorem~\ref{mainthm:c-g-waldh-cat} and
Section~\ref{sec:proofs-ii:waldhausen-category} is devoted completely to its
proof.

It is well-known that if a category has infinite coproducts,
its algebraic K-theory vanishes.  As $\mathcal{C}^G(X;R)$ has this problem,
we restrict to a full subcategory of \emph{bounded locally finite objects},
abbreviated \emph{bl-finite} objects, $\mathcal{C}^G_f(X;R)$.
It behaves from the homotopy theoretic point of view like finite CW-complexes.
From the algebraic point of view it corresponds to finitely generated free
modules.  Corresponding to projective modules we define the full subcategory
of \emph{homotopy bl-finitely dominated objects}
$\mathcal{C}^G_{hfd}(X;R)$ of $\mathcal{C}^G(X;R)$.  An object $X$ is
\emph{homotopy
bl-finitely dominated} if there is a bl-finite object $A$ and maps $r \colon A
\rightarrow X$, $i\colon X \rightarrow A$ such that $r \circ i \simeq \id_X$.

\begin{mainthm}\label{mainthm:c-g-finite}
  Both $\mathcal{C}^G_f(X;R)$ and $\mathcal{C}^G_{hfd}(X;R)$ are Waldhausen
  categories, with the inherited structure from $\mathcal{C}^G(X;R)$.
  They still have a cylinder functor and satisfy the saturation, extension and
  cylinder axiom.

  Furthermore, the inclusion $\mathcal{C}^G_f(X;R) \rightarrow
  \mathcal{C}^G_{hfd}(X;R)$ induces an isomorphism
  \begin{equation*}
    K_i(\mathcal{C}^G_f(X;R)) \rightarrow K_i(\mathcal{C}^G_{hfd}(X;R))
  \end{equation*}
  for $i\geq 1$ and an injection for $i=0$.
\end{mainthm}

The category $\mathcal{C}^G_{hfd}(X;R)$ is an analogue to the idempotent
completion of an additive category
(cf.~\cite[Section~5]{Cardenas-Pedersen;Karoubi}).  For an additive category
$\mathcal{A}$, the idempotent completion $\mathcal{A}^{\wedge}$ is the
universal additive category with additive functor $\mathcal{A} \rightarrow
\mathcal{A}^{\wedge}$ such that every idempotent splits.
We show in Corollary~\ref{cor:hpty-idem-split} in the appendix
that idempotents and ``coherent'' homotopy idempotents split up to homotopy in
$\mathcal{C}^G_{hfd}(X;R)$.

Both $\mathcal{C}^G_f(X;R)$ and $\mathcal{C}^G_{hfd}(X;R)$ are
basic ingredients needed to attack the Farrell-Jones Conjecture for simplicial rings.

\subsection{Results of independent interest}
In this article we need to discuss several topics which might be of interest.
Here is a guide for these topics.

\subsubsection{Controlled algebra for discrete rings}
We explain in Subsection~\ref{subsec:controlled-algebra-discrete} that the
constructions here specialize to a construction of a category of controlled
modules over a discrete ring.  Readers who are interested in controlled
algebra for rings can read Sections~\ref{subsec:control-spaces},
\ref{subsec:controlled-simplical-modules} and the relevant part of
\ref{subsec:G-equivariance}, as well
as~\ref{subsec:controlled-algebra-discrete}.  This gives in very few pages a
construction of a category of controlled modules.  We think our category is
technically nicer than the model described in~\cite{Bartels-Farrell-Jones-Reich;;;},
because it is, e.g., functorial in the control space and has an obvious
forgetful functor to free modules.  Otherwise the categories are
interchangeable.

\subsubsection{Establishing a Waldhausen structure and the gluing lemma}
A basic result in the homotopy theory of topological spaces is the gluing
lemma: assume that $D_i$ is pushout of $C_i \leftarrow A_i \rightarrowtail
B_i$ for $i=0,1$, where $\rightarrowtail$ denotes a cofibration.  Assume we
have maps $\varphi_A \colon A_0 \rightarrow A_1$ etc., which form a map of pushout
diagrams.  If $\varphi_A$, $\varphi_B$, $\varphi_C$ are homotopy equivalences,
then $\varphi_D$ is one.  This is not obvious, as the homotopy inverse of
$\varphi_D$ is not induced by the homotopy inverses of the other maps.

Waldhausen made the gluing lemma into one of the axioms of a Waldhausen
category.  Proofs that a given category satisfies
Waldhausen's axioms are usually omitted in the literature.
Section~\ref{sec:proofs-ii:waldhausen-category} contains a detailed proof that
our category $\mathcal{C}^G(X;R)$ satisfies Waldhausen's axioms.  Because in
$\mathcal{C}^G(X;R)$ the weak equivalences are homotopy equivalences, which
one can define once one has a cylinder functor, the proofs should generalize
to related settings, e.g., to Waldhausen's category of retractive
spaces over $X$.

\subsubsection{A Cofinality Theorem for algebraic K-theory}
Let $\mathcal{B}$ be the category of finitely generated projective modules over a
discrete ring $R$ and $\mathcal{A}$ the subcategory of free modules.
It is well-known \cite[Section~2]{Staffeldt;on-fundamental} that the algebraic K-theory of $\mathcal{B}$ differs from
that of $\mathcal{A}$ only in degree $0$.  A way to describe this is to say
that
\begin{equation*}
  K(\mathcal{A}) \rightarrow K(\mathcal{B}) \rightarrow
  \text{``}K_0(\mathcal{B}) / K_0(\mathcal{A})\text{''}
\end{equation*}
is a homotopy fiber sequence of connective spectra, where the last term is the
Eilenberg-MacLane spectrum of the group $K_0(\mathcal{B}) / K_0(\mathcal{A})$
in degree 0.  There are statements in the literature providing such a homotopy
fiber sequence when $\mathcal{A}$ and $\mathcal{B}$ satisfy a list of
conditions, e.g.~in~\cite{weibel;k-book,Thomason-Trobaugh;higher-algebraic}.
We show that these miss an essential assumption and provide a counterexample
to Corollary V.2.3.1 of Weibel's K-book~\cite{weibel;k-book}.  We also give a
proof of such a cofinality theorem, in
Subsection~\ref{subsec:cofinality-theorem}.  Note that the above example of
free and projective modules is just an illustration.  To apply the theorem
to finite and projective modules we would need to replace them by suitable
categories of chain complexes first, as they do not have mapping cylinders.

\subsubsection{A simplicial mapping telescope}
In topological spaces one can form a mapping telescope of a sequence $A_0
\rightarrow A_1 \rightarrow A_2 \dots$ by gluing together the mapping
cylinder of the individual maps.  It can be used to show that a space which is
dominated by a CW-complex is homotopy equivalent to a CW-complex, see
e.g. Hatcher~\cite[Proposition~A.11]{Hatcher;algebraic-topology}.  We need an
analogue of a mapping cylinder
in our category $\mathcal{C}^G(X;R)$.  Because it is a simplicial category,
and the homotopies are simplicial, more care is required.  We construct
a simplicial mapping telescope in
Appendix~\ref{sec:appendix-simplicial-mapping-telescope}.  For this we define
an analogue of Moore homotopies and provide the necessary tools to deal with
them.  Our results are summarized as Theorem~\ref{thm:mapping-telescopes}.  We
also define what we call a coherent homotopy idempotent in
Definition~\ref{def:coherent-homotopy-idempotent} and use the mapping
telescope to show these split up to homotopy in $\mathcal{C}^G(X;R)$ as
Corollary~\ref{cor:hpty-idem-split}.
We only use a few formal properties of $\mathcal{C}^G(X;R)$ to derive that
result.  We expect this construction to work in other settings to split
idempotents. But because we have no further examples of such categories
we refrain from providing an axiomatic framework in which the theorem would
hold.

\subsubsection{Weak equivalences of simplicial rings and algebraic K-theory}
A map $f\colon R \rightarrow S$ of simplicial rings is a weak equivalence if
it is one on
the geometric realization of the underlying simplicial sets.  Such a map
induces an equivalence $K(R) \rightarrow K(S)$ on algebraic K-theory.  Usually
this is proved by using a plus-construction description of $K(R)$,
e.g., in~\cite[Proposition~1.1]{Waldhausen;topological-spaces;;;1978}.
In~\ref{subsec:change-of-rings},
we provide a proof which only uses Waldhausen's Approximation Theorem.  The
proof shows that $f$ induces a weak equivalence on the algebraic K-theory of
the categories of controlled modules, for which we do not have a
plus-construction description.  Note, however,
that~\cite[Proposition~1.1]{Waldhausen;topological-spaces;;;1978} provides the
stronger statement that an $n$-connected map induces an $n+1$-connected map
on K-theory.  We currently have no analogue of this for controlled
modules over simplicial rings.

\subsection{The idea of control}
Let us now sketch the construction of $\mathcal{C}^G(X;R)$.  For
simplification we
assume that $G$ is the trivial group and $X$ arises from a metric space
$(X,d)$, for example from $\mathbb{R}^n$ with the euclidean metric.  The complete
and precise definitions can be found in
Section~\ref{sec:simplicial-modules-control}.

As a simplicial $R$-module $M$ is \emph{generated by a set} $\diamond_R M = \{
e_i\}_{i\in I} \subseteq \coprod_n M_n$ if every $R$-submodule $M' \subseteq
M$ which contains $\{ e_i \}_{i\in I}$ is equal to $M$.  The idea is now to
label each of the chosen generators $e_i$ of $M$ by an element $\kappa(e_i)$
of $X$ and require that maps respect the labeling ``up to an
$\alpha > 0$''.  More precisely, a \emph{controlled simplicial $R$-module over
$X$} is a simplicial $R$-module $M$, a set of generators
$\diamond_R M$ of $M$ and a map $\kappa^M\colon\diamond_R M \rightarrow
X$.  A morphism $f\colon (M, \diamond_R M , \kappa^M)
\rightarrow(N, \diamond_R N , \kappa^N)$ of controlled simplicial $R$-modules
is a map $f\colon M \rightarrow N$ of simplicial $R$-modules such that there
is an $\alpha \in \mathbb{R}_{> 0}$ such that for each $e \in \diamond_R M$ we
have that $f(e)\subseteq N$ is contained in an $R$-submodule generated by
elements $e' \in \diamond_R N$ with $d(\kappa^N(e'), \kappa^M(e)) \leq
\alpha$.

There are two problems with the objects here:  first, we want to have the
generators satisfy as few relations as possible.  This is the case for
\emph{cellular $R$-modules}, when $\diamond_R M$ is a set cells of $M$.  We
define cellular $R$-modules in
Section~\ref{subsec:basic-definitions}.  Second, the boundary maps in $M$
should behave well with respect to the labels in the control space.  A quick
way to obtain well-behaved boundary maps is to require that $\id_M \colon (M,
\diamond_R M, \kappa^M)
\rightarrow(M, \diamond_R M, \kappa^M)$ is controlled. This is a
condition on $(M,\diamond_R M, \kappa^M)$.  Therefore, we restrict to cellular modules
which are controlled.  When $X$ is a metric space, the category 
$\mathcal{C}(X;R)$ is the category of controlled cellular $R$-modules and
controlled maps.  The more general case, when $G$ is not trivial and $X$ a
\emph{control space}, is carefully
introduced in Section~\ref{sec:simplicial-modules-control}.

The category $\mathcal{C}(X;R)$ relates to the categories of controlled
modules of
Pedersen-Weibel~\cite{Pedersen-Weibel;nonconnective} and
Bartels-Farrell-Jones-Reich~\cite{Bartels-Farrell-Jones-Reich;;;} like how
chain complexes of free modules
relate to projective modules, or like how CW-complexes relate to projective
$\mathbb{Z}$-modules.  For $M$ a simplicial $R$-module and
$\mathbb{Z}[\Delta^1]$ the free simplicial abelian group on the $1$-simplex
define $M[\Delta^1] = M \otimes_\mathbb{Z} \mathbb{Z}[\Delta^1]$.  If 
$(M, \diamond_R M, \kappa^M)$ is a controlled simplicial $R$-module,
$M[\Delta^1]$ is also one, canonically.  This is the cylinder which yields the
homotopy theory in $\mathcal{C}(X;R)$.

\subsection{Structure of this article}
The proof of Theorems~\ref{mainthm:c-waldh-cat} and~\ref{mainthm:c-g-finite}
are quite involved as we need to
develop the entirety of the homotopy theory in $\mathcal{C}^G(X;R)$.  
Therefore, we split this
article into two main parts.  The first,
Sections~\ref{sec:simplicial-modules-control}
to~\ref{sec:algebraic-k-theory-control} provides only definitions without any
proofs, such that we can state our main theorems as soon as possible. 

In Section~\ref{sec:simplicial-modules-control},  we concisely review
simplicial rings and simplicial modules, as well as the idea of control.
We define the category $\mathcal{C}^G(X;R)$.
Section~\ref{sec:waldhausen-categories} defines the Waldhausen
structure on $\mathcal{C}^G(X;R)$.  We state the $G$-equivariant version of
Theorem~\ref{mainthm:c-waldh-cat} as Theorem~\ref{mainthm:c-g-waldh-cat} in
Section~\ref{subsec:waldhausen-category}.  Then we introduce the finiteness
conditions of bl-finite, homotopy bl-finite, and homotopy bl-finitely dominated
modules.  Each of these gives us a full subcategory of $\mathcal{C}^G(X;R)$.
We show (Thms.~\ref{thm:c-g-finite-waldhausen},
\ref{thm:c-g-homotopy-finite-waldhausen},
\ref{thm:c-g-hfd-waldhausen}) that the full subcategories of these are
naturally Waldhausen categories.  In
Section~\ref{sec:algebraic-k-theory-control}, we state that the category of
bl-finite and homotopy bl-finite modules have the same algebraic K-theory,
while the
algebraic K-theory of the homotopy bl-finitely dominated ones differ only at
$K_0$ (Thm.~\ref{thm:k-theory-comparison-f-hf-hfd}). 
This settles Theorem~\ref{mainthm:c-g-finite}.
If $R \rightarrow S$ is a weak equivalence of
simplicial rings we show (Thm.~\ref{thm:change-of-rings}) that the categories of
controlled modules $\mathcal{C}^G(X;R)$ and $\mathcal{C}^G(X;S)$ have
equivalent algebraic K-theory.

The second part, Sections~\ref{sec:proofs-i:control}
to~\ref{sec:proofs-iv:connective-alg-k-theory} and the appendix provide the
proofs and all intermediate definitions and theorems we need.  There does not
seem to exist in the literature an established way to verify the axioms of a
Waldhausen category apart from trivial cases, although proofs are surely
well-known.  Therefore, we provide a reasonable level of detail.  Most of the
proofs are rather formal once we establish the Relative Horn-Filling
Lemma~\ref{lem:horn-filling-relative}.

We state some initial results on simplicial modules and controlled maps
between them in Section~\ref{sec:proofs-i:control}.  The results provided
should be enough to make it possible for the experienced reader to verify all
statements we made in Section~\ref{sec:simplicial-modules-control}.

Section~\ref{sec:proofs-ii:waldhausen-category} verifies the axioms of a
Waldhausen category for $\mathcal{C}^G(X;R)$ for the cofibrations and weak
equivalences we defined in Section~\ref{subsec:waldhausen-category}.  The key
ingredient is the Relative Horn-Filling Lemma~\ref{lem:horn-filling-relative}.
Further important results are the establishing of a cylinder
functor~\ref{subsec:cylinder-functor}, the gluing
Lemma~\ref{subsec:gluing-lemma} and the extension axiom for the homotopy
equivalences~\ref{subsec:extension-axiom}.  

Section~\ref{sec:proofs-iii:finiteness-conditions} discusses the
different finiteness conditions.  This proves the results of
Section~\ref{subsec:finiteness-conditions}, i.e., it establishes that each of
the full subcategories of bl-finite (Thm.~\ref{thm:c-g-finite-waldhausen}),
homotopy bl-finite (Thm.~\ref{thm:c-g-homotopy-finite-waldhausen}) and homotopy
bl-finitely dominated modules (Thm.~\ref{thm:c-g-hfd-waldhausen}) are again
Waldhausen categories and satisfy all the extra axioms we listed.

In Section~\ref{sec:proofs-iv:connective-alg-k-theory}, we switch to algebraic
K-theory and prove
comparison theorems of the algebraic K-theory of the aforementioned
categories.  As an important part we prove a cofinality
Theorem~\ref{thm:thomason-cofinality} for algebraic K-theory.  It is stated as an exercise
in~\cite{Thomason-Trobaugh;higher-algebraic} and in Corollary V.2.3.1
in~\cite{weibel;k-book}, but we show that a crucial assumption is missing
there and prove the correct statement.  Last, we give a direct proof
(Thm.~\ref{thm:change-of-rings}) that a
weak equivalence of simplicial rings gives an equivalence on algebraic
K-theory of controlled modules.

Section~\ref{sec:miscellaneous} gives some applications. We elaborate
on the relation of this work to the Farrell-Jones Conjecture and to controlled
algebra for discrete rings
(Sections~\ref{subsec:controlled-algebra-discrete},
\ref{subsec:farrell-jones}).  We give a construction of a nonconnective
delooping of the algebraic K-theory of a simplicial ring
(Section~\ref{subsec:non-connective-algebraic-k-theory}) and discuss the case
of ring spectra (Section~\ref{subsec:ring-spectra}).
Appendix~\ref{sec:appendix-simplicial-mapping-telescope} constructs a
simplicial mapping telescope  and proves 
Theorem~\ref{thm:mapping-telescopes} about them, which is used in
Corollary~\ref{cor:hpty-idem-split} to analyze
idempotents and coherent homotopy idempotents in $\mathcal{C}^G(X;R)$.

\subsection{Previous results}
Pedersen and Weibel \cite{Pedersen-Weibel;nonconnective} first used
controlled modules to construct a nonconnective delooping of the algebraic
K-theory space of a discrete ring. 
Vogell \cite{Vogell;bounded} used the idea of control to construct a
homotopically flavoured category
which provides a delooping of Waldhausen's algebraic K-theory of spaces
$A(X)$.  Unfortunately, Vogell does not
provide any details on
why his category is a Waldhausen category.  Later
Weiss~\cite{weiss;excision} gave a quick construction of a category similar to
Vogell's one, but he also does not give a proof of the Waldhausen structure.
Weiss' definitions inspired the definitions we use here.  

With regard to discrete rings controlled algebra was developed
with the applications to the Farrell-Jones Conjecture in mind.  A
fundamental result is that of
Cardenas-Pedersen~\cite{Cardenas-Pedersen;Karoubi}, who construct a
highly useful fiber sequence of algebraic K-theory spaces, arising solely from
control spaces.  The most recent incarnation of controlled algebra is
described in \cite{Bartels-Farrell-Jones-Reich;;;} which describes the
category which is used in the most recent approaches to the Farrell-Jones
Conjecture.
Pedersen \cite{Pedersen;controlled-algebraic;;;2000}
gives a nice survey of controlled algebra at the time of its writing.

\subsection{Acknowledgements}
This work was partially supported by a PhD-grant of the Graduiertenkolleg
1150 ``Homotopy and Cohomology'' and the SFB 647: ``Space - Time - Matter.
Analytic and Geometric Structures''.  The author wants to thank Charles
Weibel for answering a question which lead to the
counterexample~\ref{subsubsec:saturated-counterexample}.
The author also wants to thank Rosona Eldred and the referee for
spotting several misprints in an earlier version of this article and many
remarks that helped to improve the exposition.

\subsection{Conventions}
We sometimes use the property that for a diagram in a category
\begin{equation*}
  \xymatrix{
    \bdot \ar[r]\ar[d]           
         & \bdot \ar[r]\ar[d] \ar@{}[dl]|{\mathrm{I}} 
         & \bdot \ar[d] \ar@{}[dl]|{\mathrm{II}}\\
    \bdot \ar[r]                   & \bdot \ar[r]        & \bdot}
\end{equation*}
the whole diagram I + II is a pushout if I and II are pushouts and II is
a pushout if I and I + II are pushouts. The dual version is proved
in~\cite[I.2.5.9]{borceux;handbook-categorical}, the third possible
implication does not hold in general.

Equations and diagrams are numbered continuously throughout the article.  If
we refer to them by number we preface the number of the section or Theorem
in which they occur, so
\ref{subsubsec:adjunction}.\eqref{eq:hom-bijection-all} is equation
\eqref{eq:hom-bijection-all} in Section~\ref{subsubsec:adjunction}.

For us, a \emph{Waldhausen category} is what
Waldhausen in~\cite{Waldhausen;spaces;;;;1985} calls a ``category with
cofibrations and weak equivalences''.  This means, a Waldhausen category comes
with two subcategories, the cofibrations and the weak equivalences which
satisfy Waldhausen's axioms from~\cite{Waldhausen;spaces;;;;1985}.

The set of natural numbers $\Nat$ contains zero. All rings have a unit.

\section{Simplicial modules and control}
\label{sec:simplicial-modules-control}

\subsection{Basic definitions}
\label{subsec:basic-definitions}
We assume familiarity with the theory of simplicial sets.  A good reference
is~\cite{goerss-jardine;simplicial-homotopy-theory}.

\subsubsection{Simplicial modules}
We recall the definition of simplicial modules. $\Delta$ is
always the category of finite ordered sets $[n] = \{0,\dots,n\}$ and
monotone maps.  A simplicial
abelian group is a functor $\Delta^\op \rightarrow \Ab$. Similarly, a
simplicial ring is a functor $\Delta^\op \rightarrow \mathcal{R}ings$.  There
are obvious generalizations of the notions of left and right modules and tensor
products.

For a simplicial set $A$ let $\Int[A]$ be the free
simplicial abelian group on $A$.  For $M$ a simplicial left $R$-module, define
$M[A]$ as the simplicial left $R$-module $M \otimes_\Int \Int[A]$.  For $M, N$
simplicial left $R$-modules define $\HOM_R(M,N)$ as the simplicial abelian
group $[n] \mapsto \Hom_R(M[\Delta^n], N)$.

\subsubsection{Cellular modules}
We call the simplicial left $R$-module $R[\Delta^n]$ an \emph{$n$-cell} and
$R[\partial \Delta^n]$ the \emph{boundary of an $n$-cell}.  We say $M$
\emph{arises from
$M'$} by attaching an $n$-cell if $M$ is isomorphic to the pushout $M'
\cup_{R[\partial \Delta^n]} R[\Delta^n]$.  Like a CW-complex in topological
spaces, a  cellular $R$-module \emph{relative to a submodule $A$} is a module $M$
together with a filtration of $R$-submodules $M^i, i \geq -1$ with $M^{-1} = A$ and $\bigcup M^i
= M$ such that $M^i$ arises from $M^{i-1}$ by attaching $i$-cells.  We call
the map $A \rightarrow M$ a \emph{cellular inclusion}.  The composition of two
cellular inclusions is again a cellular inclusion.  We prove this as
Lemma~\ref{lem:composition-of-cellular-inclusions}.
A simplicial left
$R$-module is called cellular if $* \rightarrow M$ is a cellular inclusion.
In the category of simplicial $R$-modules we use
$\rightarrowtail$ to denote cellular inclusions.

We will always remember the attaching maps of the cells to $M$ and
call this a \emph{cellular structure} on $M$.  This
gives and can be reconstructed from an element $e_n \in M_n$ for each $n$-cell
of $M$, where $M_n$ denotes the set of $n$-simplices of $M$.  This gives a set
$\diamond_R M \subseteq \bigcup_n M_n$ to which we refer as \emph{the cells of
$M$}.  As an $R$-module, $M$ can have many different cellular structures and we
do not require maps to respect them.

\begin{lem}\label{lem:cellular-inclusions}
  Let $A \rightarrowtail B$ be an inclusion of simplicial sets.  Let
  $M$ be a cellular $R$-module.  Then $M[A] \rightarrow M[B]$ is a cellular
  inclusion of cellular $R$-modules.

  If $M \rightarrowtail N$ is a cellular inclusion, then $M[A] \rightarrow
  N[A]$ is a cellular inclusion.
\end{lem}

We give a detailed proof in Subsection~\ref{subsec:cellular-structure-proof}.

\subsubsection{Finiteness conditions}  Similarly to the case of CW-complexes or
simplicial sets, we call a cellular module \emph{finite} if it has only finitely
many cells and \emph{finite-dimensional}  if it has only cells of finitely
many dimensions.

\subsubsection{Dictionary} We compare the notions introduced in this section
to the corresponding notions of discrete rings.
  \begin{table}[htb]\center
    \begin{tabular}[c]{@{}ll@{}}
      \toprule
      simplicial $R$-modules, $R$ simplicial ring & discrete $R$-modules, $R$
      discrete ring\\ \midrule
      cellular module              & free module \\
      cellular structure           & choice of a basis \\
      cellular inclusion \phantom{with free complement}
      & direct summand with free complement \\
      $M[A]$ ($A$ a simplicial set)& $\bigoplus_{a\in A} M$ ($A$ a set) \\
      $\coprod_I R\left[\Delta^n\right]$    & $\bigoplus_I R$ \\
      finite-dimensional           & \ --- \\
      finite                       & finite dimensional\\
      \bottomrule
    \end{tabular} \caption{Dictionary between simplicial rings and modules.} \label{tab:dictionary} 
  \end{table}

\subsection{Control spaces}\label{subsec:control-spaces}
\begin{definition}[morphism control structure]
  Let $X$ be a Hausdorff topological space. A \emph{morphism control structure
  on $X$} consists of a set $\mathcal{E}$ of subsets $E \subseteq X\times X$
  (i.e.,~relations on $X$) satisfying:
  \begin{enumerate}
    \item \label{it:morph-cont-1} For $E, E' \in \mathcal{E}$ there is an
      $\overline{E}\in \mathcal{E}$ such that $E \circ E'
      \subseteq \overline{E}$ where ``\/$\circ$'' is the composition
      of relations.  (For two relations $E_1, E_2$ on $X$ their
      \emph{composition} is defined as
      \begin{equation*}
        E_2 \circ E_1 := \{(z,x)\mid \exists y \in X\colon (y,x) \in E_1,
        (z,y)\in E_2\}. )
      \end{equation*}
    \item \label{it:morph-cont-2} For $E, E' \in \mathcal{E}$ there is an
      $E''\in\mathcal{E}$ such that $E\cup E' \subseteq E''$.
    \item \label{it:morph-cont-3} Each $E\in \mathcal{E}$ is symmetric,
      i.e.,~$(x,y)\in E \Leftrightarrow (y,x) \in E$.
    \item \label{it:morph-cont-4} The diagonal $\Delta \subseteq X \times X$ is
      a subset of each $E \in \mathcal{E}$.
  \end{enumerate}
  An element $E \in \mathcal{E}$ is called a \emph{morphism control
  condition}. 
\end{definition}
The topology on $X$ is relevant for the finiteness conditions, see
Section~\ref{subsec:finiteness-conditions}.
\subsubsection{Thickenings}
For $U \subseteq X$ and
$E\in \mathcal{E}$ we call
\begin{equation*}
  U^E = \{x \in X\mid \exists y\in U\colon (x,y) \in E\}
\end{equation*}
the \emph{$E$-thickening} of $X$. 
\begin{definition}[object support structure]
  Given $X$ and a morphism control structure $\mathcal{E}$ on $X$. An
  \emph{object support structure on $(X, \mathcal{E})$} is a set
  $\mathcal{F}$ of subsets $F \subseteq X$ satisfying:
  \begin{enumerate}
    \item For $F, F'\in \mathcal{F}$ there is an $F''\in \mathcal{F}$ such that
      $F\cup F' \subseteq F''$.
    \item For $F \in \mathcal{F}$ and $E\in \mathcal{E}$ there is an $F'''\in
      \mathcal{F}$ such that $F^E \subseteq F'''$.
  \end{enumerate}
  An element $F \in \mathcal{F}$ is a called an \emph{object support
  condition}.
\end{definition}

\subsubsection{} 
We call the triple $(X, \mathcal{E}, \mathcal{F})$ a \emph{control space}.  If
$\mathcal{F}=\{X\}$ we often omit it from the notation.
If $\mathcal{E}$ is a morphism control structure on $X$ we can define a new
morphism control structure on $X$,
\begin{equation*}
  \mathcal{E}' := \{E' \subseteq X \times X \mid \exists E \in \mathcal{E}:
  \Delta \subseteq E' \subseteq E\}.
\end{equation*}
In all applications we can replace $\mathcal{E}$ with $\mathcal{E}'$
and obtain the \emph{same} category of controlled modules.  Similarly, we can
replace the object support structure $\mathcal{F}$ with $\mathcal{F}' := \{ F'
\subseteq X \mid \exists F \in \mathcal{F}: F' \subseteq F\}$ in all
applications.

\subsubsection{Maps}
A map of control spaces $(X_1, \mathcal{E}_1, \mathcal{F}_1) \rightarrow (X_2,
\mathcal{E}_2, \mathcal{F}_2)$ is a (not necessarily continuous) map $f\colon
X_1 \rightarrow X_2$ such that for each $E_1 \in \mathcal{E}_1$ and $F_1 \in
\mathcal{F}_1$ there are $E_2 \in \mathcal{E}_2$ and $F_2 \in \mathcal{F}_2$
with $(f\times f)(E_1) \subseteq E_2$ and $f(F_1) \subseteq F_2$.\bigskip

We give the most important examples,
see~\cite[Section~2.3]{Bartels-Farrell-Jones-Reich;;;} for more. 

\begin{example}[metric control]\label{ex:metric-first}
  Let $X$ have a metric $d$. Then
  \begin{equation*}
    \mathcal{E}_d = \left\{E \mid \exists \alpha  \text{ such that } E = \{(x,y)
    \mid d(x,y) \leq \alpha \} \right\}
  \end{equation*}
  is a morphism control structure on $X$.
\end{example}

\begin{example}[continuous control]\label{example:continous-control}
  Let $Z$ be a topological space and $[1,\infty)$ the half-open interval with
  closure $[1,\infty]$. Define \emph{continuous control} morphism control
  structure  $\mathcal{E}_{cc}$
  on $X := Z \times [1,\infty)$ as follows. $E$ is in $\mathcal{E}_{cc}$ if it
  is symmetric and the following two properties are satisfied:
  \begin{enumerate}
    \item For every $x\in Z$ and each neighborhood $U$ of $x\times \infty$ in
      $Z\times [1,\infty]$ there is a neighborhood $V \subseteq U$ of $x\times
      \infty$ in $Z\times [1,\infty]$ such that $E \cap ( (X\smallsetminus U)
      \times V) = \varnothing$.
    \item $p_{[1,\infty)} \times p_{[1,\infty)} (E) \in
      \mathcal{E}_d([1,\infty))$, where $d$ is the standard euclidean metric
      on $[1,\infty)$ and $p_{[1,\infty)}$ is the projection to $[1,\infty)$.
  \end{enumerate}
\end{example}

\begin{example}[compact support]
  Let $X$ be a topological space.  Set
  \begin{equation*}
    \mathcal{F}_c := \{F \subseteq X \mid F \text{ is compact}\}.
  \end{equation*}
  We call
  $\mathcal{F}_c$ the \emph{compact} object support structure on $X$. 
  If $X$ is a proper metric space (closed balls are compact), then
  $\mathcal{F}_c$ is an object support structure on $(X, \mathcal{E}_d)$.
  Also, if $Z$ is a topological space then $\mathcal{F}_c$ is an object
  support structure on $(Z \times[1,\infty), \mathcal{E}_{cc})$.
\end{example}
 
\subsection{Controlled simplicial modules}
\label{subsec:controlled-simplical-modules}
\subsubsection{Cellular submodules}
We required cellular $R$-modules to come with a chosen cellular structure
$\diamond_R M$.  A \emph{cellular submodule} is an $R$-submodule $M'$ of $M$
which is generated by a subset of $\diamond_R M$.  In particular, we have an
inclusion $\diamond_R M' \subseteq \diamond_R M$ induced by $M'
\hookrightarrow M$.

\begin{definition}
  For a set of simplices $Q \subseteq \bigcup M_n$ define $\left< Q\right>_M$
  as the smallest cellular submodule of $M$ containing $Q$.
\end{definition}
We abbreviate $\left<\{e\}\right>_M$ by $\left<e\right>_M$ or $\left<e\right>$.

\subsubsection{Modules over a space}
For a control space $(X, \mathcal{E}, \mathcal{F})$ define a \emph{general
module over $X$} to be a cellular module $(M, \diamond_R M)$ together with a
map $\kappa_R\colon \diamond_R M \rightarrow X$.  (We followed
\cite{weiss;excision} in the notation.)

\begin{definition}[Controlled module]
  \label{def:controlled-module}
  A controlled $R$-module over $X$ is a general $R$-module $(M, \diamond_R M,
  \kappa_R)$ over $X$ such that there are $E\in \mathcal{E}$, $F\in
  \mathcal{F}$ with:
  \begin{enumerate}
    \item For all $e\in \diamond_R M$ and $e' \in \diamond_R\left<e\right>_M$ we have
      $(\kappa_R(e), \kappa_R(e')) \in E$.
    \item $\kappa_R(\diamond_R M) \subseteq F$.
  \end{enumerate}
   
  We say $M$ is \emph{$E$-controlled and has support in $F$}, and often leave
  $\kappa_R$ understood from context.
\end{definition}

\begin{definition}[Controlled maps]
  A map $f\colon (M, \kappa_R^M) \rightarrow (N, \kappa_R^N)$ of controlled
  modules is a map $f\colon M \rightarrow N$ of simplicial $R$-modules such
  that there is
  an $E\in \mathcal{E}$ and for all $e\in \diamond_R M$, $e'\in \diamond_R
  \left<f(e)\right>_N$ we have $(\kappa^M_R(e), \kappa^N_R(e')) \in E$.

  We say $f$ is \emph{$E$-controlled}. We say $f$ is \emph{controlled} if
  there exists an $E$ such that $f$ is $E$-controlled.
\end{definition}

\subsubsection{Composition}\label{subsubsec:composition}
If $f, f_1, f_2 \colon M \rightarrow M'$ and $g\colon M' \rightarrow M''$ are
controlled maps of controlled modules, then $g\circ f$ and $f_1 + f_2$ are
controlled.  See Lemma~\ref{lem:sums-composition-of-controlled-maps} for the
proof.

\subsubsection{The category of controlled modules $\mathcal{C}^G(X;R)$}
For $(X, \mathcal{E}, \mathcal{F})$ a control space the controlled $R$-modules
over $X$ together with the controlled maps between them form a category which
we denote by $\mathcal{C}( X, \mathcal{E}, \mathcal{F}; R)$.  We will
usually abbreviate it by $\mathcal{C}(X;R)$, $\mathcal{C}(X)$, $\mathcal{C}(X,
\mathcal{E}, \mathcal{F})$ or $\mathcal{C}$.

If $M$ is a controlled module over $X$ and $A$ a simplicial set then $M[A]$ is
canonically a controlled module over $X$ with control map induced by the
projection $M[A] \rightarrow M$.  This is
functorial in $A$ and $M$.  We prove this as
Lemma~\ref{lem:adjoint-simp-module-controlled}.

\subsubsection{Cellular inclusion of controlled modules}
Define a cellular inclusion of controlled modules to be a map $(M, \kappa_R^M)
\rightarrow (N, \kappa_R^N)$ such that $(M,\diamond_R M) \rightarrow (N,
\diamond_R N)$ is a cellular inclusion of simplicial $R$-modules and the
inclusion $i\colon \diamond_R M \hookrightarrow \diamond_R N$ satisfies
$\kappa^M_R = \kappa^N_R \circ i$.  This is the right notion of a subobject in
$\mathcal{C}$.

If $A\rightarrowtail B$ is an inclusion of simplicial sets, then $M[A]
\rightarrow M[B]$ is a cellular inclusion of controlled modules.  If $M
\rightarrow N$ is a cellular inclusion of controlled modules, then $M[A]
\rightarrow N[A]$ is one.

\subsection{A $\Hom$-bijection}
\subsubsection{Controlled filtration on the $\HOM$-space}
Let $M,N\in \mathcal{C}(X, \mathcal{E}, \mathcal{F})$, $E \in \mathcal{E}$.
Define $\Hom^E_R(M, N)$ as the subset of maps $f\colon M \rightarrow N$ in
$\mathcal{C}(X)$ which are $E$-controlled.  Similarly, define $\HOM_R^E(M, N)$ as
the sub-simplicial set of $E$-controlled maps $M[\Delta^n]\rightarrow N$.
Boundaries and degeneracies respect $E$, so this is a well-defined simplicial
subset of $\HOM_R(M,N)$.  Define $\HOM^\mathcal{E}_R(M,N)$ as
$\bigcup_{E\in\mathcal{E}} \HOM^E_R(M,N)$.

\subsubsection{Uncontrolled adjunction and its controlled
counterparts}\label{subsubsec:adjunction}
If $A$ is a simplicial set, we have an obvious adjunction 
\begin{equation}
  \label{eq:sset-adjunction}
  \Hom_R(M[A], N) \cong
\Hom_{\sset}(A, \HOM_R(M,N)) 
\end{equation}
in simplicial $R$-modules.  This restricts to a bijection 
\begin{equation}
  \label{eq:hom-bijection}
  \Hom_R^E(M[A], N) \cong \Hom_{\sset}(A, \HOM_R^E(M,N)).
\end{equation}
If $A$ is a finite simplicial set, we have a bijection
\begin{equation}
  \Hom_R^\mathcal{E}(M[A], N) \cong \Hom_{\sset}(A, \HOM_R^\mathcal{E}(M,N)) 
  \label{eq:hom-bijection-all}
\end{equation}
which is natural in $A$, $M$ and $N$.  The two bijections
\eqref{eq:hom-bijection} and \eqref{eq:hom-bijection-all} are not quite
adjunctions.  Therefore, we refer to them as
\emph{Hom-bijections}.

\subsection{$G$-equivariance}
\label{subsec:G-equivariance}
Let $G$ be a group.  (All our groups are discrete.)  All notions
above generalize in a straightforward way to $G$-equivariant versions:

\subsubsection{$G$-equivariant cellular modules}
An \emph{action} of $G$ on a simplicial $R$-module $M$ is a group homomorphism
$\rho\colon G \rightarrow \Aut_R(M)$.  The action is called
\emph{cell-permuting} if it induces an action on $\diamond_R M$.  An action is
\emph{free} if it is cell-permuting and the action on $\diamond_R M$ is free.  If
$M$, $N$ are simplicial $R$-modules with $G$-actions $\rho_M$, $\rho_N$ a
map $f\colon M \rightarrow N$ of simplicial $R$-modules is
\emph{$G$-equivariant} if for each $g\in G$ the diagram
\begin{equation*}
  \xymatrix{
  M \ar[r]^f \ar[d]_{\rho_M(g)} & N \ar[d]^{\rho_N(g)} \\
  M \ar[r]^f & N
  }
\end{equation*}
commutes.  A $G$-equivariant map $L \rightarrow M$ is a \emph{cellular
inclusion}, if
it is one after forgetting the $G$-action.  

If $M$ is a simplicial $R$-module with $G$-action and $A$ a simplicial set then
$M[A]$ has a $G$-action by the functoriality in $M$.

\subsubsection{$G$-equivariant control spaces}
A control space $(X, \mathcal{E}, \mathcal{F})$ is $G$-equivariant if $X$ has
a continuous $G$-action such that $gE =E$ (diagonal action) and $gF =F$ for
all $g\in G$, $E\in \mathcal{E}$, $F\in\mathcal{F}$.  A free control space is
one where the action of $G$ on $X$ is free.  The examples of control spaces in
Section~\ref{subsec:control-spaces} have $G$-equivariant analogues, see
\cite[2.7, 2.9, 3.1, 3.2]{Bartels-Farrell-Jones-Reich;;;}.

\subsubsection{The category of $G$-equivariant controlled modules
$\mathcal{C}^G(X;R)$}
\label{subsubsec:G-equivariant-controlled-modules}
Assume $(X,\mathcal{E},\mathcal{F})$ is a free $G$-equivariant control space.  A
controlled simplicial $R$-module with $G$-action over $X$ is a controlled
module $(M, \diamond_R M, \kappa_R)$ with cell-permuting $G$-action such that
$\kappa_R$ is $G$-equivariant.  A morphism $(M, \kappa_R) \rightarrow
(N,\kappa_R)$ of such modules is a $G$-equivariant morphism $M \rightarrow N$
which is controlled over $X$.  Denote the category of these as
$\mathcal{C}^G(X,\mathcal{E}, \mathcal{F}; R)$.  We use abbreviations like
$\mathcal{C}^G$, etc.  All further definitions of
Section~\ref{subsec:controlled-simplical-modules} transfer to $\mathcal{C}^G$.

\subsubsection{The $G$-equivariant $\Hom$-bijection}\label{subsubsec:G-adjunctions}
The adjunction \ref{subsubsec:adjunction}.\eqref{eq:sset-adjunction} between
$M[-]$ and $\HOM_R(M,-)$ and its controlled counterparts generalize to the
$G$-equivariant setting.  Denote by $\Hom_R(M,N)^G$ the
subset of $\Hom_R(M,N)$ of $G$-equivariant maps.  Denote by
$\HOM_R(M,N)^G$ the sub-simplicial set of $\HOM_R(M,N)$ consisting of
$G$-equivariant maps.  The adjunction~\eqref{eq:sset-adjunction} and the
bijection~\eqref{eq:hom-bijection} of \ref{subsubsec:adjunction}
restrict to the adjunction
\begin{equation*}
  \Hom_R(M[A], N)^G \cong \Hom_{\sset}(A, \HOM_R(M,N)^G) 
\end{equation*}
and the bijection
\begin{equation}\label{eq:hom-bijection-G-E}
  \Hom_R^E(M[A], N)^G \cong \Hom_{\sset}(A, \HOM_R^E(M,N)^G).
\end{equation}
Similarly, if $A$ is a finite simplicial set~\eqref{eq:hom-bijection-all}
restricts to a
bijection 
\begin{equation}\label{eq:hom-bijection-all-G}
  \Hom_R^{\mathcal{E}}(M[A], N)^G \cong \Hom_{\sset}(A,
  \HOM_R^{\mathcal{E}}(M, N)^G) .
\end{equation} 
All of these are natural in $A, M, N$.\smallskip 

Most of the statements in this section are easy to check, so we will not
provide proofs.  The exceptions are proven in
Section~\ref{sec:proofs-i:control}.

\section{Waldhausen categories of controlled modules}
\label{sec:waldhausen-categories}
In the following, $(X,\mathcal{E}, \mathcal{F})$ is always a free
$G$-equivariant control space which we abbreviate as $X$.  We fix a
simplicial ring $R$ for this section.
We put some additional structure on $\mathcal{C}^G(X; R)$, which we now
describe.

\subsection{$\mathcal{C}^G(X, \mathcal{E}, \mathcal{F}; R)$ as a Waldhausen
category}\label{subsec:waldhausen-category}
First we make $\mathcal{C}^G(X;R)$ into a Waldhausen category.
We will use the definitions of category with cofibrations, category with weak
equivalences, cylinder functor and the saturation, cylinder and extension
axiom from~\cite[1.1, 1.2, 1.6]{Waldhausen;spaces;;;;1985}.
We will give detailed proofs of the statements below in
Section~\ref{sec:proofs-ii:waldhausen-category}.

\subsubsection{Cofibrations}
  Define a map $f\colon M \rightarrow N$ in $\mathcal{C}^G(X)$ to be a
  \emph{cofibration}
  if there are isomorphisms $\alpha\colon M' \rightarrow M$ and $\beta\colon N
  \rightarrow N'$ in $\mathcal{C}^G(X)$ such that $\beta\circ f\circ \alpha$ is
  a cellular inclusion.  Note that $\alpha$, $\beta$ do not need to preserve the
  cellular structures, so the notion is independent of chosen cellular
  structures.  The compositions of cofibrations is a cofibration.  We also
  denote cofibrations by $\rightarrowtail$.  If
  $A\rightarrowtail B$ is a cofibration and $A \rightarrow C$ a map then the
  pushout $B\cup_A C$ exists and $C\rightarrow B\cup_A C$ is a cofibration.  If
  $A \rightarrowtail B$ has been a cellular inclusion, then we have a
  preferred choice for this pushout, in particular, it can be
  chosen functorially. 
\begin{thm}\label{thm:c-g-cat-with-cofibrations}
  $\mathcal{C}^G(X;R)$ is a category with cofibrations.
\end{thm}

\subsubsection{Cylinders}
Consider the simplicial set $\Delta^1$, the \emph{interval}.  It
comes with inclusions $i_0, i_1\colon \Delta^0 \rightarrow \Delta^1$ and a projection
$p\colon \Delta^1 \rightarrow \Delta^0$.  For $M$ in $\mathcal{C}^G(X;R)$, this
induces the corresponding cellular inclusions $M \rightarrow M[\Delta^1]$ and
a projection $M[\Delta^1] \rightarrow M$ which makes $M[\Delta^1]$ into a
cylinder object. 

\begin{thm}[Cylinder Functor]
  \label{thm:cylinder-functor}
  For a map $f\colon A \rightarrow B$ define $T(f)$ as
  $A[\Delta^1]\cup_{ {i_1}_*} B$.  This is functorial
  in the arrow category and therefore gives a \emph{cylinder functor} on
  $\mathcal{C}^G(X;R)$ in the sense of
  Waldhausen~\cite[1.6]{Waldhausen;spaces;;;;1985}.
\end{thm}

\subsubsection{Weak equivalences}
Two maps $f, g\colon A \rightarrow B$ are homotopic if there is a
\emph{homotopy} $H\colon A[\Delta^1] \rightarrow B$ such that $H\circ {i_0}_* = f$
and $H\circ {i_1}_* = g$.  This gives rise to the obvious notion of homotopy
equivalence. 

\begin{thm}\label{thm:c-g-waldhausen-category}
The subcategory of homotopy equivalences in $\mathcal{C}^G(X;R)$
forms a category of weak equivalences, in particular, it satisfies the gluing
Lemma~\ref{lem:gluing-lemma}.  It also satisfies the saturation
axiom~\ref{lem:saturation-axiom} and the
extension axiom~\ref{def:extension-axiom}.  The cylinder functor satisfies
the cylinder axiom~\ref{lem:cylinder-axiom} with respect to these weak
equivalences.  
\end{thm}
Theorems~\ref{thm:c-g-cat-with-cofibrations}, \ref{thm:cylinder-functor}
and~\ref{thm:c-g-waldhausen-category} immediately imply the $G$-equivariant
version of Theorem~\ref{mainthm:c-waldh-cat}:

\begin{mainthm}[Waldhausen structure on $\mathcal{C}^G$]
\label{mainthm:c-g-waldh-cat}
  Let $X$ be a control space and $R$ a simplicial ring.  The category of
  controlled simplicial modules over $X$, $\mathcal{C}^G(X;R)$, together with
  the homotopy equivalences as weak equivalences is a Waldhausen category.
  Therefore, Waldhausen's algebraic K-theory of $\mathcal{C}(X;R)$ is defined.

  The category has a cylinder functor and it satisfies Waldhausen's cylinder
  axiom, his saturation axiom and his extension axiom.
\end{mainthm}

This category is too big, it has an Eilenberg-Swindle which causes its
algebraic K-theory to vanish.  But it contains interesting
full subcategories, which we discuss next.

\subsubsection{A remark on the proofs}
The main tool for the proofs are the adjunctions of
\ref{subsubsec:G-adjunctions} and a careful analysis of the control conditions
in these settings, often accompanied by an induction over the cells.  Here is a
prototype of such a proof.  We need to show that we have horn-filling in our
category.  In particular, the following (simplified) lemma should hold.

\begin{lem*}
  Given a map $M[\Lambda^n_i] \rightarrow P$.  Then there is an extension to a
  map $M[\Delta^n] \rightarrow P$.
\end{lem*}

\begin{proof}
  By the $\Hom$-bjiection
  \ref{subsubsec:G-adjunctions}.\eqref{eq:hom-bijection-G-E}, the situation is
  equivalent to finding a lift in
  the diagram denoted by $\dashrightarrow$.
  \begin{equation*}
    \xymatrix{
    \Lambda^n_i \ar[r] \ar[d] & \HOM^E_R(M, P) \\
    \Delta^n \ar@{-->}[ur] }
  \end{equation*}
  But the simplicial set $\HOM^E_R(M, P)$ is in fact a simplicial abelian
  group, hence the lift exists by the Kan-property, i.e., it is fibrant.
\end{proof}

The general proofs are considerably more involved.  We will devote
Section~\ref{sec:proofs-ii:waldhausen-category} to them:  cofibrations are
discussed in Section~\ref{subsec:cofibrations}, the
cylinder functor in Section~\ref{subsec:cylinder-functor}.  We discuss a
homotopy extension property in~\ref{subsec:homotopy-extension-property}, which
will settle the saturation axiom~\ref{lem:saturation-axiom}.  If $f\colon A
\rightarrow B$ is a homotopy equivalence, then $A \rightarrow T(f)$ is a
deformation coretraction, which we prove
in Section~\ref{subsec:cylinders-homotopy-equivalences}.
In~\ref{subsec:pushouts-of-weq}, we show that the pushout of a homotopy
equivalence is a homotopy equivalence which settles the gluing
Lemma~\ref{lem:gluing-lemma} in Subsection~\ref{subsec:gluing-lemma}.  The
extension axiom is treated in Subsection~\ref{subsec:extension-axiom}.
A summarizing proof of Theorem~\ref{thm:c-g-waldhausen-category} is given at
the end of Section~\ref{sec:proofs-ii:waldhausen-category}.

\subsection{Finiteness conditions}
\label{subsec:finiteness-conditions}
We define full subcategories of $\mathcal{C}^G(X;R)$ by specifying conditions
on the objects.

\subsubsection{bl-finite controlled modules}
Here we use the topology on $X$.  Let $(M, \diamond_R M, \kappa_R)$ be a
controlled module over $X$. 
\begin{definition}[bl-finite]
  $M$ is \emph{locally finite} if for
  each $x\in X$ there is a neighborhood $U$ of $x$ in $X$ such that
  $\kappa_R^{-1}(U)$ is a finite subset of $\diamond_R M$.  Then $M$ is called
  \emph{bounded locally finite}, or \emph{bl-finite}, if it is locally finite
  and finite-dimensional.  Denote the full subcategory of bl-finite modules by
  $\mathcal{C}^G_f(X;R)$.  
\end{definition}
\begin{thm}\label{thm:c-g-finite-waldhausen}
  $\mathcal{C}^G_f(X;R)$ inherits the structure of a Waldhausen category from
  $\mathcal{C}^G(X;R)$. It satisfies the saturation and extension axiom and
  has a cylinder functor satisfying the cylinder axiom.  
\end{thm}
Furthermore, cofibrations are isomorphic in $\mathcal{C}^G_f(X;R)$ to cellular
inclusions.  The proof needs the Hausdorff-property of $X$.  If $X$ is a
point, $M$ is bl-finite if and only if it is finite as cellular simplicial
module.

\subsubsection{Homotopy bl-finite controlled modules}
\begin{definition}[homotopy bl-finite]
  $M \in \mathcal{C}^G(X;R)$ is \emph{homotopy bl-finite} if there is a
  homotopy equivalence $M \xrightarrow{\sim} M'$ such that $M'$ is a bl-finite
  module.  We denote the full subcategory of homotopy bl-finite modules by
  $\mathcal{C}^G_{hf}(X;R)$.   
\end{definition}

\begin{thm}\label{thm:c-g-homotopy-finite-waldhausen}
  $\mathcal{C}^G_{hf}(X;R)$ inherits the structure of a Waldhausen category.
  It has a cylinder functor satisfying the cylinder axiom.  The saturation and
  extension axiom hold.
\end{thm}

\subsubsection{Homotopy bl-finitely dominated controlled modules}
\begin{definition}[homotopy bl-finitely dominated]
  An object $M\in \mathcal{C}^G(X;R)$ is \emph{homotopy bl-finitely dominated}
  if it is a strict retract of a homotopy bl-finite object.  We denote the
  full subcategory of homotopy bl-finitely dominated modules by
  $\mathcal{C}^G_{hfd}(X;R)$.     
\end{definition}
\begin{thm}\label{thm:c-g-hfd-waldhausen}
  $\mathcal{C}^G_{hfd}(X;R)$   inherits the structure of a Waldhausen
  category.  It has a cylinder functor satisfying the cylinder axiom.  The
  saturation and extension axiom hold.  
\end{thm}
Homotopy bl-finitely dominated modules are equivalently characterized by
being a retract up to homotopy of a bl-finite object.\bigskip

We use the abbreviations $\mathcal{C}^G_f$, $\mathcal{C}^G_{hf}$,
$\mathcal{C}^G_{hfd}$, etc.
Section~\ref{sec:proofs-iii:finiteness-conditions} is devoted to the proofs
that $\mathcal{C}^G_{f}$, $\mathcal{C}^G_{hf}$ and $\mathcal{C}^G_{hfd}$ are
Waldhausen categories.

\section{Algebraic K-theory of controlled modules}
\label{sec:algebraic-k-theory-control}
\subsection{Connective K-theory}
\label{subsec:connective-k-theory}
Let us make explicit that we can define algebraic K-theory of the 
Waldhausen categories $\mathcal{C}^G_{f}$, $\mathcal{C}^G_{hf}$ and
$\mathcal{C}^G_{hfd}$.

\begin{definition}[Algebraic K-theory of categories of controlled
  modules]\label{def:algebraic-k-theory-of-c-g}
  Let $G$ be a group, $(X, \mathcal{E}, \mathcal{F})$ be a free $G$-equivariant
  control space. Let $R$ be a simplicial ring.  Define the algebraic
  K-theory spectrum of the Waldhausen category 
  $\mathcal{C}^G_f(X,\mathcal{E}, \mathcal{F};R)$ as the connective spectrum
  \begin{equation*}
    K(w\mathcal{C}^G_f(X,\mathcal{E}, \mathcal{F};R))
  \end{equation*}
  where $K$ is Waldhausen's algebraic K-theory of spaces
  \cite{Waldhausen;spaces;;;;1985}.  We define similarly the algebraic K-theory
  of $\mathcal{C}^G_{hf}$ and $\mathcal{C}^G_{hfd}$.  We set
  $K_n(w\mathcal{C}^G) := \pi_n K(w\mathcal{C}^G)$.
\end{definition}

\begin{rem}
  In \cite{Waldhausen;spaces;;;;1985}, Waldhausen defines the K-theory as a
  space and then constructs a delooping, i.e., an $\Omega$-spectrum.  This is
  what we use here, because for the Cofinality
  Theorem~\ref{thm:thomason-cofinality} it is more
  convenient to work with spectra.  See also
  \cite[1.5.3]{Thomason-Trobaugh;higher-algebraic}.
\end{rem}

\begin{thm}[Different finiteness conditions]
  \label{thm:k-theory-comparison-f-hf-hfd}
  Let $(X, \mathcal{E}, \mathcal{F})$ be a control space and $R$ a simplicial
  ring.
  \begin{enumerate}
    \item \label{item:prop-k-theory:1} The inclusion $\mathcal{C}^G_f(X,
      \mathcal{E}, \mathcal{F};R)
      \rightarrow \mathcal{C}^G_{hf}(X, \mathcal{E}, \mathcal{F};R)$ is exact
      and induces a homotopy equivalence on K-Theory.
    \item \label{item:prop-k-theory:2} The inclusion $\mathcal{C}^G_{hf}(X, 
      \mathcal{E}, \mathcal{F};R) \rightarrow \mathcal{C}^G_{hfd}(X,
      \mathcal{E}, \mathcal{F};R)$ is exact and induces an isomorphism on $K_n$
      for $n\geq 1$ and an injection $K_0(\mathcal{C}^G_{hf}) \rightarrow
      K_0(\mathcal{C}^G_{hfd})$.
  \end{enumerate}
\end{thm}

\subsubsection{Proof of Theorem~\ref{mainthm:c-g-finite}}
Theorem~\ref{mainthm:c-g-finite} is a summary of
Theorems~\ref{thm:c-g-finite-waldhausen},
\ref{thm:c-g-homotopy-finite-waldhausen}, and \ref{thm:c-g-hfd-waldhausen} for
the Waldhausen structures on the subcategories defined by finiteness
conditions, and of Theorem~\ref{thm:k-theory-comparison-f-hf-hfd} for the
comparison of the algebraic K-theory of these subcategories.\smallskip

After we discuss a cofinality Theorem for algebraic K-Theory
in~\ref{subsec:cofinality-theorem} we prove
Theorem~\ref{thm:k-theory-comparison-f-hf-hfd} in
Section~\ref{subsec:change-of-finiteness-conditions}.  
\subsection{Further results}

\subsubsection{Separations of variables}
Let $\pt$ the one-element set.  
The categories $\mathcal{C}^G(G/1, \{G\times G\}, \{G\}; R)$ and
$\mathcal{C}(\pt, \{\pt\}, \{\pt\}; R[G])$ are equivalent.  Both are
equivalent to the category of cellular $R[G]$-modules.  The equivalences
respect the finiteness conditions $f$, $hf$ and $hfd$.

\begin{cor}\label{cor:sep-of-variables}
  The algebraic K-theory of $\mathcal{C}^G_{hfd}(G/1, \{G\times G\}, \{G\};
  R)$ is homotopy equivalent to the algebraic K-theory of the simplicial
  ring $R[G]$.
\end{cor}

\subsubsection{Change of rings}
For a map $\varphi\colon R
\rightarrow S$ of simplicial rings we obtain an induced functor $\mathcal{C}^G(X;R)
\rightarrow \mathcal{C}^G(X;S)$.  It restricts to a functor on the bl-finite,
homotopy bl-finite and homotopy bl-finitely dominated subcategories.

\begin{thm}[Change of rings]\label{thm:change-of-rings}
  Let $\varphi\colon R \rightarrow S$ be map of simplicial rings which is a weak
  equivalence. Then $\varphi$ induces a map $\mathcal{C}^G_{f}(X; R) \rightarrow
  \mathcal{C}^G_{f}(X; S)$ which is an equivalence on algebraic K-Theory.
\end{thm}

Theorem~\ref{thm:change-of-rings} is proved in
Section~\ref{subsec:change-of-rings}.

\section{Proofs I: control}
\label{sec:proofs-i:control}
Section~\ref{sec:simplicial-modules-control} introduced our basic categories
of controlled simplicial modules.  Most results are straightforward to check.
Therefore, we discuss only a few important lemmas and leave
the rest to the reader.  For the structure of a Waldhausen
category on $\mathcal{C}^G$ which we introduced in
Section~\ref{sec:waldhausen-categories} we will provide much more detailed
proofs in Section~\ref{sec:proofs-ii:waldhausen-category}.

\subsection{Cellular structure}
\label{subsec:cellular-structure-proof}


We start with a remark about cellular inclusions.
Let $M$ be a cellular simplicial $R$-module with filtration
$\bigcup_i M^i = M$.  Let $\sigma\colon R [\partial \Delta^n] \rightarrow M$ be
a map of $R$-modules.  By the adjunction
\ref{subsubsec:adjunction}.(\ref{eq:sset-adjunction}), this is the
same as a map $\sigma'\colon \partial \Delta^n \rightarrow M$ of simplicial
sets.  As $M^{n-1} \rightarrowtail M$ induces an isomorphism in simplicial degree
$\leq n-1$ and as $\sigma'$ is determined by
its image in degree
$n-1$ of the simplicial set $M$, $\sigma'$ and hence $\sigma$ factors over
$M^{n-1}$.  Therefore:

\begin{lem}\label{lem:composition-of-cellular-inclusions}
  Let $B$ be a cellular $R$-module relative to $A$ with cells
  $\diamond^{\mathrm{rel} A}_R B$.  Let $M$ be a cellular $R$-module
  relative to $B$ with cells $\diamond^{\mathrm{rel} B}_{R} M$.  Then $M$ is a
  cellular $R$-module relative to $A$ with cells $\diamond^{\mathrm{rel} A}_R
  B \cup \diamond^{\mathrm{rel} B}_{R} M$.  In particular, the composition of
  cellular inclusions is a cellular inclusion.

  Let $A \rightarrowtail A_1 \rightarrowtail A_2 \rightarrowtail \dots$ be a
  sequence of cellular inclusions, then $\colim_i A_i$ is a cellular
  $R$-module relative to $A$.\qed
\end{lem}
Lemma~\ref{lem:composition-of-cellular-inclusions} will be used freely in the
proof of the Lemma~\ref{lem:cellular-inclusions}, which we give next.  Let us
repeat the statement.

\begin{replem}{lem:cellular-inclusions}\samepage\  
  \begin{enumerate}
    \item Let $A \rightarrowtail B$ be an inclusion of simplicial sets.  Let
      $M$ be a cellular module.  Then $M[A] \rightarrow M[B]$ is a cellular
      inclusion of cellular $R$-modules.
    \item If $M \rightarrowtail N$ is a cellular inclusion, then $M[A]
      \rightarrow N[A]$ is a cellular inclusion.
  \end{enumerate}
\end{replem}

\begin{proof}[Proof of Lemma~\ref{lem:cellular-inclusions}]
  $B$ arises from $A$ by attaching cells (in the sense of simplicial sets).
  That is, $B = \colim_i B^{i,A}$ with $B^{0,A}=A$ and $B^{i,A} = \coprod
  \Delta^n \cup_{\coprod \partial \Delta^n} B^{i-1,A}$.
  $M[-]$ commutes with colimits, as it is a composition of two left-adjoint
  functors.  Hence, it suffices to show
  the case of the inclusion $\partial \Delta^n \rightarrow \Delta^n$ of
  simplicial sets.  Then $M[B] = \colim_i M[B^{i,A}]$ with 
  \begin{equation*}
    M[B^{i,A}] \cong \coprod M[\Delta^n] \cup_{\coprod M[\partial \Delta^n]}
    M[B^{i-1,A}] 
  \end{equation*}
  and hence
  $M[A] = M[B^{0,A}]\rightarrow M[B]$ is a cellular inclusion of simplicial
  $R$-modules, as pushouts, coproducts, and (infinite) compositions of
  cellular inclusions of simplicial sets are cellular inclusions.

  To show that $M[\partial \Delta^n] \rightarrow M[\Delta^n]$ is a cellular
  inclusion we do induction over the skeleta of $M$.  For simplicity we only
  consider the case of one cell. Attaching cells of the same dimension
  simultaneously works in the same way.  Assume $L$ and $M'$ are cellular
  subcomplexes of $M$ and that $M'$ arises from $L$ by attaching a $q$-cell,
  for some $q\geq 0$.

  We prove that
  \begin{equation*}
    \varphi\colon M[\partial\Delta^n] \cup_{L[\partial \Delta^n]} L[\Delta^n]
    \rightarrow M[\partial\Delta^n] \cup_{M'[\partial \Delta^n]} M'[\Delta^n] 
  \end{equation*}
  is a cellular inclusion.  As $M'$ arises from $L$ by attaching
  a $q$-cell and $R[A][B] \cong R[A \times B]$ the characteristic map $(R[
  \Delta^q], R[\partial \Delta^q]) \rightarrow (M', L)$ gives pushouts
  \begin{equation*}
    \xymatrix{
      R [\partial \Delta^n \times \partial \Delta^q] \ar[r] \ar[d] &
      R [\partial \Delta^n \times          \Delta^q]        \ar[d] \\
      L [\partial \Delta^n ] \ar[r] &
      M'[\partial \Delta^n ] } \ \text{ and }\ 
    \xymatrix{
      R [ \Delta^n \times \partial \Delta^q] \ar[r] \ar[d] &
      R [ \Delta^n \times          \Delta^q]        \ar[d] \\
      L [ \Delta^n ] \ar[r] &
      M'[ \Delta^n ] } 
  \end{equation*}
  Now consider the commutative diagram
  \begin{equation}
    \label{eq:attaching-cells-cellular-inclusion}
    \xymatrix{
      R [\partial \Delta^n \times          \Delta^q]               &
      R [\partial \Delta^n \times          \Delta^q] \ar[r] \ar[l] &
      R [         \Delta^n \times          \Delta^q]               \\
      R [\partial \Delta^n \times          \Delta^q] \ar[u] \ar[d]^{\tau_1} &
      R [\partial \Delta^n \times \partial \Delta^q] 
          \ar[r] \ar[l] \ar[u] \ar[d]^{\tau_2} &
      R [         \Delta^n \times \partial \Delta^q] \ar[u] \ar[d]^{\tau_3} \\
      M [\partial \Delta^n] &
      L [\partial \Delta^n] \ar[r] \ar[l] &
      L [         \Delta^n] }.
  \end{equation}
  The lower vertical maps come from the attaching data of the $q$-cell,
  namely $\tau_1$ comes from the $q$-cell $R[\Delta^q] \rightarrow M$ and
  $\tau_2$, $\tau_3$ from the boundary $R[\partial \Delta^q] \rightarrow L$. The
  other maps are the obvious inclusions.

  We can form the pushouts  of the columns of
  Diagram~\eqref{eq:attaching-cells-cellular-inclusion}, which
  yields the diagram $M[\partial \Delta^n] \leftarrow M'[\partial \Delta^n]
  \rightarrow M'[\Delta^n]$, whose pushout is the target of $\varphi$.  If we
  form the pushouts of the rows of
  Diagram~\eqref{eq:attaching-cells-cellular-inclusion} we obtain the diagram
  \begin{equation*}
    M[\partial\Delta^n] \cup_{L[\partial \Delta^n]} L[\Delta^n]
      \xleftarrow{\alpha}
    R[\partial \Delta^n \times \Delta^q \cup \Delta^n \times \partial\Delta^q] 
      \xrightarrow{\beta}
    R[\Delta^n \times \Delta^q]
  \end{equation*}
  whose pushout is therefore isomorphic to the target of $\varphi$.  Also,
  $\beta$ is a cellular inclusion because $R[-]$ commutes with pushouts.
  Hence,
  $\varphi$ is a pushout of a cellular inclusion and therefore itself
  cellular.  Thus by induction and taking a colimit over $q$ the map
  $M[\partial \Delta^n] \rightarrow M [\Delta^n]$ is cellular.

  The second part is now easier.  It suffices to consider the case where $N$
  arises from $M$ by attaching a single $R$-cell, as $M \mapsto M[A]$ commutes with
  pushouts.  If we attach a $q$-cell to $M$ to obtain $N$ we can obtain the
  pushout 
  \begin{equation*}
    \xymatrix{
    R[\partial \Delta^n][A] \ar[r] \ar[d]^{\tau}  & M[A] \ar[d]^{\tau'} \\
    R[         \Delta^n][A] \ar[r]         & N[A] }
  \end{equation*}
  and the $\tau$ is a cellular inclusion by the previous argument.
  Therefore, $\tau'$ is a cellular inclusion. This finishes the
  proof.
\end{proof}

\subsection{Controlled maps}
We denote the support of a controlled module by $\supp(M)$.  We have the
following precise statement about the control of maps.

\begin{lem}\label{lem:sums-composition-of-controlled-maps}\ 
  \begin{enumerate}
    \item If $f\colon M \rightarrow M'$ is $E$-controlled and $g\colon M'
      \rightarrow M'' $ is $E'$-controlled, then $g\circ f$ is $E' \circ
      E$-controlled.
    \item  If $f_1, f_2\colon M \rightarrow M'$ are $E_1$-,
      resp.~$E_2$-controlled and $E_1 \cup E_2 \subseteq E_3$, then $f_1 +
      f_2$ is $E_3$-controlled.
    \item If $M$ is an $E$-controlled module, then $\id_M$ is $E$-controlled.
  \end{enumerate}
\end{lem}

\begin{proof}
  Let $e\in \diamond_R M$.  Then $\supp(\left<f(e)\right>) \subseteq
  \{\kappa^M(e)\}^E$.  Furthermore, for $e'\in \diamond_R \left<f(e)\right>$ we have
  $\supp(\left<g(e')\right>) \subseteq\{\kappa^{M'}(e')\}^{E'}$.  By
  minimality, we have 
  \begin{equation*}
    \left<(g\circ f)(e))\right>_{M''} \subseteq
    \left<g(\left<f(e)\right>_{M'})\right>_{M''} 
  \end{equation*}
  so the support of $\left<(g\circ f)(e))\right>_{M''}$ is contained in
  $\{\kappa^M(e)\}^{E'\circ E}$.

  For the second part, $\left<(f_1 + f_2)(e)\right>  \subseteq
  \left<\{(f_1)(e),(f_2)(e)\}\right>$ and 
  \begin{equation*}
    \diamond_R\bigl<\{(f_1)(e),(f_2)(e)\}\bigr> = 
    \diamond_R \bigl<(f_1)(e)\bigr>\ \cup\ 
    \diamond_R \bigl<(f_2)(e)\bigr> .
  \end{equation*}
  Hence, the support of $\left<(f_1 + f_2)(e)\right>$ is contained in
  \begin{equation*}
    \{\kappa^M(e)\}^{E_1} \cup \{\kappa^M(e)\}^{E_2} \subseteq
    \{\kappa^M(e)\}^{E_3} .
  \end{equation*}
  The third part is clear.
\end{proof}

We say that a map $f\colon M \rightarrow N$ of cellular $R$-modules is
\emph{$0$-controlled} if it induces a map $\diamond_R M \rightarrow \diamond_R
N$ and $\kappa^M_R = \kappa^N_R \circ f$.  Cellular inclusions are
$0$-controlled.  The name $0$-controlled is a misuse of notation, such a map
has the control of its image.  However, $g\circ f$ has the control of $g$ if
$f$ is $0$-controlled.

\begin{lem}\label{lem:adjoint-simp-module-controlled}
  Let $(M, \kappa_R)$ be an $E$-controlled $R$-module.  Let $A$ be a
  simplicial set.  Then $M[A]$ can be made canonically into an $E$-controlled
  $R$-module.

  Furthermore, each map $A \rightarrow B$ of simplicial sets induces a
  $0$-controlled map $M[A] \rightarrow M[B]$.
\end{lem}

\begin{proof}
  From Lemma~\ref{lem:cellular-inclusions} it follows that
  each cell $e$ of $M[A]$ arises from exactly one cell $p(e)$ of $M$.  Define
  $\kappa^{M[A]}(e) := \kappa^M(e)$.  It makes $M[A]$ into an $E$-controlled
  module:  for $e\in M[A]$ we have $p(e) \in \diamond_R M$.  Then $e\in
  \left<p(e)\right>[A] \subseteq M[A]$, and $\left<p(e)\right>[A]$ is supported on
  $\{\kappa(e)\}^E$.  This shows the first part.

  Another way to describe the control map is to note that the map $M[A]
  \rightarrow M[\pt] \cong M$ is $0$-controlled.
  A cell of $M$ is given by a map $R[\Delta^n] \rightarrow M$.  Hence, for each
  cell of $M$ we obtain a commutative diagram
  \begin{equation*}
    \xymatrix{ 
        M[A] \ar[r]^f & M[B] \\
        R[\Delta^n][A]\ar[u]\ar[r] & R[\Delta^n][B]\ar[u]
        }
  \end{equation*}
  which shows that $f$ maps cells to cells.  As $f$ commutes with the map to
  $M[\pt]$ this shows that $f$ is $0$-controlled.
\end{proof}

\section{Proofs II: controlled simplicial modules as a Waldhausen category}
\label{sec:proofs-ii:waldhausen-category}
In this section we establish that $\mathcal{C}^G(X, \mathcal{E},
\mathcal{F};R)$ is indeed a Waldhausen
category, thus proving all the claims we made in
Subsection~\ref{subsec:waldhausen-category}.  In particular, we prove
Theorem~\ref{thm:c-g-waldhausen-category}, which says that the homotopy
equivalences form a class of weak equivalences for $\mathcal{C}^G$.  We 
use the definitions from~\ref{subsec:waldhausen-category} without further
notice.

If the category $\mathcal{C}^G$
were a subcategory of cofibrant objects of a Quillen model category, the we
would basically obtain the structure of a Waldhausen category for
free (cf.~\cite[Proposition~II.8.1]{goerss-jardine;simplicial-homotopy-theory}
and Section~\ref{subsec:gluing-lemma}).
Many examples in the literature are of that form, e.g.~chain complexes
\cite[1.2.13]{Thomason-Trobaugh;higher-algebraic} or
Waldhausen's retractive spaces over a space \cite{Waldhausen;spaces;;;;1985}.
Unfortunately, there
does not
seem to be a suitable Quillen model category which contains our category of
controlled modules.  In fact, neither general pushouts nor infinite unions
exist in general in any of our categories of controlled modules.

Therefore, we prove all results directly.  The hardest part is to prove the
gluing Lemma~\ref{lem:gluing-lemma}.  We follow a strategy the author learned
from a proof of Waldhausen of the gluing lemma for topological
spaces.

We expect the experienced reader to be able to fill gaps easily, but otherwise
refer to the author's thesis~\cite{thesis-ullmann}.  Note that the thesis
works with a slightly different definition of controlled module.


\subsection{Cofibrations}
\label{subsec:cofibrations}

We assume familiarity with Section 1.1 to 1.6 of
\cite{Waldhausen;spaces;;;;1985} and use the language from there freely. 

\begin{lem}[Pushouts along cellular inclusions] 
  \label{lem:cellular-pushouts-in-CG}
  Let $(A, \kappa_R^A) \rightarrow (B, \kappa_R^B)$ be a cellular inclusion in
  $\mathcal{C}^G(X,\mathcal{E},\mathcal{F};R)$, let $f\colon (A, \kappa_R^A) \rightarrow (C, \kappa_R^C)$
  be any controlled map in $\mathcal{C}^G(X,\mathcal{E},\mathcal{F};R)$. The pushout $D := C \cup_A B$, 
  \begin{equation*}
    \xymatrix{A \ar@{ >->}[r] \ar[d]^f  & B \ar[d] \\
    C \ar@{ >->}[r]                     & D}
  \end{equation*}
  of simplicial $R$-modules can be chosen canonically and, furthermore, it has a
  canonical structure of an object $(D, \kappa_R^D)$ in $\mathcal{C}^G$.
  Furthermore, $(C, \kappa_R^C) \rightarrow (D, \kappa_R^D)$ is a cellular
  inclusion.

  Hence, $\mathcal{C}^G$ has canonical pushouts along cellular inclusions.
\end{lem}

\begin{rem}
  Being canonical should mean that for each diagram there is a
  \emph{preferred choice} of the pushout, only depending on the diagram.  In
  general for a cofibration $A \rightarrowtail M$ there is no preferred choice
  of a cellular inclusion isomorphic to it.  Having canonical
  pushouts along cellular inclusions allows us to define a cylinder functor
  in Subsection~\ref{subsec:cylinder-functor} which does not depend on
  making choice for each diagram.  In particular, we know the control
  conditions involved.
\end{rem}

\begin{proof}[Proof of Lemma~\ref{lem:cellular-pushouts-in-CG}]
  The category of simplicial $R$-modules can be equipped with
  canonical pushouts.  For example, in the above situation one could form the
  coproduct of $B$ and $C$ and divide out the relations from $A$.  We discuss
  the cellular structure.

  Let $e$ be a cell in $B$ not in $A$ with attaching map $\alpha$.  This gives
  a cell in $D$ with attaching map $f\circ \alpha$.  This way one shows that
  $C\rightarrow D$ is cellular, and therefore so is $D$.  Hence, there is a
  canonical isomorphism 
  \begin{equation*}
    \diamond_R D \cong \diamond_R C \cup (\diamond_R B \smallsetminus
    \diamond_R A).
  \end{equation*}
  Define $\kappa^D_R\colon \diamond_R D
  \rightarrow X$ as the composition of this isomorphism with the map
  $\kappa^C_R \cup \kappa^B_R \colon \diamond_R C \cup (\diamond_R B
  \smallsetminus
      \diamond_R A) \rightarrow X$.  This produces the control conditions of
  Table~\ref{tab:cc-on-pushouts}.  As $\varphi$ is $G$-equivariant, $D$ is a
  controlled $G$-equivariant cellular $R$-module.
\end{proof}

  \begin{table}[tb]
    \center
    \mbox{
    \begin{tabular}[c]{@{}cl@{}}
      \toprule
      \multicolumn{2}{l}{control conditions we have} \\
      \cmidrule(rl){1-2}
      $A,B$ & $E_B$-controlled  \\
      $C$   & $E_C$-controlled  \\
      $f$   & $E_f$-controlled  \\
      $g_C,g_B$ & $E$-controlled\\
      \bottomrule
    \end{tabular}\ 
    $\vcenter{\xymatrix{
      A \ar@{ >->}[r] \ar[d]^f       & B \ar[d]^{\widetilde{f}} \ar@/^/[ddr]^{g_B} \\
      C \ar@{ >->}[r] \ar@/_/[drr]_{g_C} & D \ar@{-->}[dr]^g \\
      && T}}$ \ 
    \begin{tabular}[c]{@{}cl@{}}
      \toprule
      \multicolumn{2}{l}{control conditions we get} \\
      \cmidrule(rl){1-2} 
      $D$ & $E_C \cup E_f \circ E_B$-controlled \\
      $\widetilde{f}$ & $E_f\circ E_B$-controlled \\
      $g$ & $E$-controlled \\
      \bottomrule
    \end{tabular}} \caption{Control conditions on pushouts along cellular
    inclusions in $\mathcal{C}^G$.}
    \label{tab:cc-on-pushouts} 
  \end{table}

There are no canonical pushouts along cofibrations in $\mathcal{C}^G$, but
as cofibrations are isomorphic to cellular inclusions, pushouts along
cofibrations also exist in~$\mathcal{C}^G$.

\subsubsection{The subcategory of cofibrations}
\label{subsubsec:subcategory-cofibrations}
By Lemma~\ref{lem:cellular-pushouts-in-CG} the
composition of cofibrations is again a cofibration.  Isomorphisms are
cofibrations and the map $* \rightarrow M$ from the trivial module to any
controlled module $M$ is a cofibration.
Lemma~\ref{lem:cellular-pushouts-in-CG} immediately implies that pushouts
along cofibrations exist. 
We have just verified:

\begin{repthm}{thm:c-g-cat-with-cofibrations}
  $\mathcal{C}^G(X;R)$ is a category with cofibrations.
\end{repthm}

Note that Lemma~\ref{lem:cellular-pushouts-in-CG} implies, in particular, that
for $A, B \in \mathcal{C}^G$ the coproduct $A \vee B$ in $\mathcal{C}^G$ exists.
In our setting a retract of a cofibration might not be a cofibration
in general, as pushouts along such maps do not need to be cellular.  This is
already true for discrete rings and free modules.

\subsection{The cylinder functor}\label{subsec:cylinder-functor}
The goal of this section is to prove Theorem~\ref{thm:cylinder-functor}.
\begin{repthm}{thm:cylinder-functor}[Cylinder functor]
  For a map $f\colon A \rightarrow B$ define
  \begin{equation*}
    T(f) := A[\Delta^1]\cup_{ {i_1}_*} B .
  \end{equation*}
  This is functorial
  in the arrow category and therefore gives a \emph{cylinder functor} on
  $\mathcal{C}^G(X;R)$ in the sense of
  Waldhausen~\cite[1.6]{Waldhausen;spaces;;;;1985}.
\end{repthm}

If we need to refer to the object $T(f)$, we sometimes call it the
\emph{mapping cylinder} of $f$.
$T$ gives a functor from $\Ar \mathcal{C}^G$, the category of arrows in
$\mathcal{C}^G$, into diagrams in $\mathcal{C}^G$,
taking $f\colon A \rightarrow B$ to a commutative diagram
  \begin{equation}\label{eq:diag:cylinder-T-is-cylinder-functor}
    \xymatrix{ 
      A \ar[r]^{\iota_0} \ar[dr]_f   &T(f)\ar[d]^{p}   & B \ar[l]_{\iota_1}
      \ar[dl]^{\id}\\
    &B& }.
  \end{equation}
Here $\iota_0$ is called the \emph{front inclusion}, $\iota_1$ is called the
\emph{back inclusion} and $p$ is called the \emph{projection}.  Waldhausen
requires the following two axioms to be satisfied.
\begin{enumerate}
  \item (Cyl~1) Front and back inclusion assemble to an exact functor
    \begin{equation*}
      \begin{gathered}
        \Ar \mathcal{C}^G \longrightarrow \mathcal{F}_1 \mathcal{C}^G \\
         f \mapsto \big( \iota_0 \vee \iota_1 \colon A \vee B \rightarrowtail
         T(f) \big).
      \end{gathered}
    \end{equation*}
  \item (Cyl~2) $T( * \rightarrow A ) = A$ for every $A\in \mathcal{C}^G$ and
    the projection and the back inclusion are the identity map on $A$.
\end{enumerate}

Here $\mathcal{F}_1 \mathcal{C}^G$ is the full subcategory of
$\Ar\mathcal{C}^G$
with objects the cofibrations. Both can be made into categories with
cofibrations, with the cofibrations of $\mathcal{F}_1 \mathcal{C}^G$ and the
notion of an exact functor as in~\cite[1.1]{Waldhausen;spaces;;;;1985}.  We
use the notion $A\vee B$ from
\cite{Waldhausen;spaces;;;;1985} for the coproduct of $A$
and $B$. 

Waldhausen first defines weak equivalences and then defines the cylinder
functor.  We proceed in the opposite order.

(Cyl~2) is directly verified using $*[\Delta^1] = *$ and choosing the right
canonical pushouts along $* \rightarrow *$.

\begin{lem}\label{lem:cylinder-functor-assembled-gives-functor}
  Front and back inclusion give a functor $\Ar \mathcal{C}^G \rightarrow
  \mathcal{F}_1 \mathcal{C}^G$,
  \begin{equation*}
    f\mapsto (A \vee B\rightarrowtail T(f)).
  \end{equation*}
\end{lem}

\begin{proof}
  This proof is a template for later proofs.
  The only thing to show is that $A \vee B \rightarrow T(f)$ is a cellular
  inclusion. Consider the diagram
  \begin{equation*}
    \xymatrix{
    {*} \ar@{ >->}[r]\ar@{ >->}[d]    & A \ar@{ >->}[d] \ar@{}[dl]|{\mathrm{I}}  &   \\
    A \ar@{ >->}[r]\ar[d]_f           
         & A\vee A \ar@{ >->}[r]\ar[d] \ar@{}[dl]|{\mathrm{II}} 
         & A[\Delta^1] \ar[d] \ar@{}[dl]|{\mathrm{III}}\\
    B \ar@{ >->}[r]                   & A\vee B \ar[r]        & T(f)}.
  \end{equation*}
  Here $A \vee A \rightarrowtail A [\Delta^1]$ is the cellular inclusiona
  \begin{equation*}
    \iota_0 \vee \iota_1 \colon A[0] \vee A[1] = A[0 \amalg 1] \rightarrowtail
    A[\Delta^1].
  \end{equation*}
  We claim that
  every possible square is a pushout along a cellular inclusion.
  $\mathrm{I}$~is a pushout square by definition, as well as $\mathrm{I} +
  \mathrm{II}$. It
  follows that II is one. Furthermore, II + III  is a pushout square by definition of
  $T(f)$, so III is one. Hence, the lower map $A \vee B \rightarrow T(f)$ is a
  cellular inclusion by Lemma~\ref{lem:cellular-pushouts-in-CG}.
\end{proof}

\begin{lem}
  The functor of Lemma~\ref{lem:cylinder-functor-assembled-gives-functor} is
  exact.
\end{lem}

\begin{proof}
  We need to show that the functor respects the structure of a category with
  cofibrations.  Pushouts and the zero object are defined pointwise in
  $\Ar \mathcal{C}^G$ and $\mathcal{F}_1 \mathcal{C}^G$.  Therefore, $A \vee B$
  and $T(f)$ commute with pushouts.  So we only have to show that the functor
  maps cofibrations to cofibrations.

Let us briefly recall the cofibrations in $\Ar \mathcal{C}^G$ and $\mathcal{F}_1
\mathcal{C}^G$, cf.~\cite[Lemma 1.1.1]{Waldhausen;spaces;;;;1985}. For notation
let 
\begin{equation}\label{eq:diag:mor-in-ar-cg}
  \vcenter{\xymatrix{ A \ar[r] \ar[d]^{f} & A' \ar[d]^{f'}  \\
	      B\ar[r]     & B'}}
\end{equation}
be a map in $\Ar\mathcal{C}^G$ from $A\rightarrow B$ to $A'\rightarrow B'$.
It is a cofibration in $\Ar\mathcal{C}^G$ if both horizontal maps are
cofibrations. 

The category $\mathcal{F}_1 \mathcal{C}^G$ is the full
subcategory of $\Ar\mathcal{C}^G$ with objects being the cofibrations in
$\mathcal{C}^G$.  Hence, if $f$ and $f'$ are cofibrations, the diagram also
shows a map in $\mathcal{F}_1 \mathcal{C}^G$.  It is a cofibration in
$\mathcal{F}_1 \mathcal{C}^G$ if $A \rightarrow A'$ and $A' \cup_A B
\rightarrow B'$ are cofibrations in $\mathcal{C}^G$.  (It follows that $B
\rightarrow B'$ is a cofibration.)
See~\cite[Lemma~1.1.1]{Waldhausen;spaces;;;;1985} for details and a proof that
the composition of cofibrations in $\mathcal{F}_1 \mathcal{C}^G$ is again a
cofibration.

  We have to show that for a map (\ref{eq:diag:mor-in-ar-cg}) which
  is a cofibration in $\Ar\mathcal{C}^G$ the maps $A \vee B\rightarrow A' \vee
  B'$ and $(A'\vee B') \cup_{A \vee B} T(f) \rightarrow T(f')$ are
  cofibrations in $\mathcal{C}^G$. As functors respect isomorphisms we can
  assume that all cofibrations are cellular inclusions. 

  Assume we have a diagram (\ref{eq:diag:mor-in-ar-cg}) where the vertical
  maps are cellular inclusions. We can factor (\ref{eq:diag:mor-in-ar-cg})
  into 
  \begin{equation*}
    \xymatrix{
      A \ar[d]^f \ar@{=}[r]^{\id}   & A \ar[d]^{f^*} \ar[r]  & A' \ar[d]^{f'} \\
      B          \ar[r]         & B'     \ar@{=}[r]^{\id}    & B' }.
  \end{equation*}
  It suffices to check each map individually.  The map $A \vee B \rightarrow A
  \vee B'$ is a pushout along the cofibration $B \rightarrow B'$, similarly,
  for $A \vee B' \rightarrow A' \vee B'$.  Hence, both are cofibrations.

  Recalling that by definition $T(f)$ is the pushout $B \cup_f A[\Delta^1]$
  \begin{equation*}
    \xymatrix{ A \ar[r]\ar[d]^f   & A[\Delta^1] \ar[d] \\
     B \ar[r]                  & T(f)}
  \end{equation*}
  we see that $T(f^*)$ is the pushout 
  \begin{equation*}
    \xymatrix{ B \ar[r]\ar[d]   & T(f) \ar[d] \\
     B' \ar[r]                  & T(f^*)}
  \end{equation*}
  and hence (by ``canceling $A$'' by a similar pushout argument as in the
  proof of Lemma~\ref{lem:cylinder-functor-assembled-gives-functor}) it is the
  pushout
  \begin{equation*}
    \xymatrix{ A\vee B \ar[r]\ar[d]   & T(f) \ar[d] \\
     A \vee B' \ar[r]                  & T(f^*)}
  \end{equation*}
  so the map $(A\vee B') \cup_{A \vee B} T(f) \rightarrow T(f^*)$ is an
  isomorphism and therefore a cellular inclusion.  Using the canceling
  argument
  for $B'$ we can write the other map 
  \begin{equation*}
    (A' \vee B') \cup_{A\vee B'} T(f^*) \rightarrow T(f')
  \end{equation*}
  as
  \begin{equation}\label{eq:map:cylinder-functor-cofibrations}
    A' \cup_{A[0]} A[\Delta^1] \cup_{f^*} B' \rightarrow A'[\Delta^1]
    \cup_{f'} B'.
  \end{equation}
  Here the first object is a cylinder where we glued in spaces at both sides.
  But because $A \rightarrowtail A'$ is a cellular inclusion so is 
  \begin{equation*}
    A'[0]
  \cup_{A[0]} A[\Delta^1] \cup_{A[1]} A'[1] \rightarrowtail A'[\Delta^1].
  \end{equation*}
  We have the commutative diagram
  \begin{equation}
    \xymatrix{
    {A[1]} \ar@{ >->}[r]\ar@{ >->}[d] \ar@/_2pc/@{ >->}[dd]_{f^*}    & 
        A'[0] \cup_{A[0]} A[\Delta^1] \ar@{ >->}[d]   &   \\
    A'[1] \ar@{ >->}[r]\ar[d]_{f'}          
        & A'[0]\cup_{A[0]} A[\Delta^1] \cup_{A[1]} A'[1] \ar@{ >->}[r]\ar[d] 
        & A'[\Delta^1] \ar[d] \\
    B' \ar@{ >->}[r]                   
    & A'[0] \cup_{A[0]} A[\Delta^1] \cup_{f^*} B' \ar[r]^-{\tau}
        & A'[\Delta^1]\cup_{f'} B'}
        \label{diag:cylinder-functor-exact}
  \end{equation}
  where every square and, in particular, the lower right one is a pushout (by
  the same reasoning as in the proof of
  Lemma~\ref{lem:cylinder-functor-assembled-gives-functor}). The map $\tau$ is
  the same as the map (\ref{eq:map:cylinder-functor-cofibrations}).
  Hence, using Lemma~\ref{lem:cellular-pushouts-in-CG} one last time it follows
  that the map~\eqref{eq:map:cylinder-functor-cofibrations} is a cellular
  inclusion.
\end{proof}

Weak equivalences in $\Ar\mathcal{C}^G$ and $\mathcal{F}_1 \mathcal{C}^G$ are
defined pointwise.  As $A \rightarrow A[\Delta^1]$ respects homotopy
equivalences, the cylinder functor respects homotopy equivalences, too.  
This finishes the proof of Theorem~\ref{thm:cylinder-functor}.
\subsection{The homotopy extension property}
\label{subsec:homotopy-extension-property}

We want to prove the gluing lemma for the homotopy equivalences in
$\mathcal{C}^G(X;R)$.  The main ingredient is the relative homotopy extension
property, which we will prove first.  Recall that the $i$th horn $\Lambda^n_i
\subseteq \Delta^n$ is $\partial \Delta^n$ minus the $i$th face, see e.g.
\cite[I.1, p.~6]{goerss-jardine;simplicial-homotopy-theory}.

\begin{lem}[Relative Horn-Filling]\label{lem:horn-filling-relative}
  Let $M, P \in \mathcal{C}^G$. Let $A$ be a cellular submodule of $M$, let
  $\Lambda^n_i \subseteq \Delta^n$ be a horn.  Any controlled maps
  $A[\Delta^n] \rightarrow P$ and $M[\Lambda^n_i] \rightarrow P$ which agree
  on $A[\Lambda^n_i]$ can be extended to a controlled map $M[\Delta^n]
  \rightarrow P$.  

  If $M$ is $E_M$-controlled and both maps to $P$ are $E_f$-controlled, then
  the extended map can be chosen to be $E_f \circ E_M$-controlled.
\end{lem}

\begin{proof}
  First we prove the claim of the lemma if $G= \{1\}$.  It suffices to produce a
  controlled retraction $r\colon M[\Delta^n] \rightarrow  M[\Lambda^n_i]
  \cup_{A[\Lambda^n_i]} A[\Delta^n]$ with section the inclusion.
  
  Let $B_k := A \cup M_k$, where $M_k$ is the submodule of $M$ generated
  by all cells of dimension $\leq k$.  We do induction over $k$.  We assume
  the following induction hypothesis:
  \begin{enumerate}
    \item There is a retraction $g_k\colon M[\Lambda^n_i]
      \cup_{B_k[\Lambda^n_i]} B_k[\Delta^n] \rightarrow M[\Lambda^n_i]
      \cup_{A[\Lambda^n_i]} A[\Delta^n]$.
    \item For each $e_0\in \diamond_R M$ the map $g_k$ restricts to
      \begin{multline}\label{eq:horn-filling-ind-hypothesis}
        \left<e_0\right>_M [\Delta^n]\ \cap\ 
        \left( M[\Lambda^n_i] \cup_{B_k[\Lambda^n_i]} B_k[\Delta^n]\right)
        \xrightarrow[\hspace{2em}]{g_k} \\
        \left<e_0\right>_M [\Delta^n]\  \cap \ 
        \left( M[\Lambda^n_i] \cup_{A[\Lambda^n_i]} A[\Delta^n]\right)
      \end{multline}
  \end{enumerate}
  The second condition is needed to keep track of the control condition during
  the induction.  It is
  important that the condition holds for all cells of $M$ and not only
  those from $B_k$.  We abbreviate $N_k := \left( M[\Lambda^n_i]
  \cup_{B_k[\Lambda^n_i]} B_k[\Delta^n]\right)$.  Hence, $g_k$ is a map $N_k
  \rightarrow N_{-1}$.

  For $k=-1$ the induction hypothesis is satisfied because $g_{-1} = \id$.

  So assume the induction hypothesis holds for $k-1$.  As $B_{k}$ arises
  from $B_{k-1}$ by attaching cells of dimension $k$, it suffices to treat
  the case of attaching one cell $e$, as cells of the same dimension can be
  attached independently.

  We obtain a diagram
  \begin{equation}\label{eq:rel-horn-filling}
    \vcenter{\xymatrix{
    \underline{R}[\Delta^k \times \Lambda^n_i \cup \partial \Delta^k \times \Delta^n]
    \ar[d]\ar[r]^{\partial e_*} &
      B_k[\Lambda^n_i] \cup_{B_{k-1}[\Lambda^n_i]}  
      B_{k-1}[\Delta^n] \ar[d]\ar[r]^-{g_{k-1}} & 
      N_{-1}\\
      \underline{R}[\Delta^k \times \Delta^n] \ar[r]^{e_*} & B_k[\Delta^n]
      \ar@{-->}[ur]_{\tau}}},
  \end{equation}
  where we want to construct the dashed arrow $\tau$.  Given such a map
  $\tau_i$ for each $k$-cell $e_i$, we define $g_k\colon N_k \rightarrow
  N_{-1}$ as the union of $g_{k-1}\colon N_{k-1} \rightarrow N_{-1}$ and the
  maps~$\tau_i$.
  
  In~\eqref{eq:rel-horn-filling}, $\underline{R}$
  is the module $R[\Delta^0]$ over $\kappa^M(e)$.  But as the square is a
  pushout it suffices to find a lift
  $\underline{R}[\Delta^k\times
  \Delta^n] \rightarrow N_{-1}$. The map $e_*$ 
  factors over $\left<e\right>_M[\Delta^n]$, as does $\partial e_*$.
  By the induction hypothesis, $g_{k-1} \circ \partial e_* \subseteq
  \left<e\right>_M [\Delta^n] \cap N_{-1}$.  It suffices to find a lift
  $\underline{R}[\Delta^k \times \Delta^n] \rightarrow \left<e\right>_M
  [\Delta^n] \cap N_{-1}$ in order to construct $\tau$.

  By the $\Hom$-bijection
  \ref{subsubsec:G-adjunctions}.\eqref{eq:hom-bijection-all-G}, it
  suffices to find a
  lift in the diagram of simplicial sets
  \begin{equation}
    \xymatrix{
    \Delta^k \times \Lambda^n_i \cup \partial \Delta^k \times \Delta^n
    \ar[d]_{\iota}\ar[r] &
      \HOM_R^\mathcal{E}(\underline{R},\left<e\right>_M [\Delta^n] \cap
  N_{-1})\\
    \Delta^k \times \Delta^n \ar@{-->}[ur]}.
    \label{diag:extension-axiom-lift}
  \end{equation}
  But
  \begin{equation*}
    \HOM_R^\mathcal{E}(\underline{R},\left<e\right>_M [\Delta^n] \cap N_{-1})
    = \left<e\right>_M [\Delta^n] \cap N_{-1}
  \end{equation*}
  as $\left<e\right>_M$ has bounded support.  Such a lift exists because
  first, $\left<e\right>_M [\Delta^n] \cap N_{-1}$ is an abelian group and
  hence fibrant, and second, the vertical inclusion $\iota$ in
  \eqref{diag:extension-axiom-lift} arises by repeated horn-filling
  (cf.~\cite[p.~18/19]{goerss-jardine;simplicial-homotopy-theory}). 
  
  This gives a lift
  $\underline{R}[\Delta^k \times
  \Delta^n] \rightarrow \left< e\right>_M [\Delta^n] \cap N_{-1}$, which
  induces $\tau$ and hence provides $g_k$.
  Note that $g_k$ restricts to a map
  \begin{equation}\label{eq:horn-filling-bounded-condition}
    \left<e\right>_M [\Delta^n]
    \rightarrow 
    \left<e\right>_M [\Delta^n]  \cap N_{-1}
  \end{equation}
  which is exactly the second condition of the induction hypothesis for $k$
  if $e_0 = e \in \diamond_R M$.  For general $e_0 \in
  \diamond_R M$ consider first the case $e\not\in \diamond_R
  \left<e_0\right>_M$. Then $\left<e_0\right>_M [\Delta^n]
  \cap N_k
  \subseteq\left<e_0\right>_M [\Delta^n] \cap N_{k-1}$ and the induction
  hypothesis follows from the induction hypothesis
  for $k-1$.  Otherwise $\left<e\right>_M \subseteq
  \left<e_0\right>_M$ and then $g_k$ restricts to
  \begin{gather*}  
    \begin{split}
      \left< e_0\right>_M [\Delta^n] \cap N_{k} =  
      & \left(\left<e_0\right>_M [\Delta^n] \cap N_{k-1} \right) \ \cup\
      \left<e\right>_M [\Delta^n] \\
      \longrightarrow
      &\left(\left<e_0\right>_M [\Delta^n] \cap N_{-1}\right) \ \cup\ \left(
      \left<e\right>_M [\Delta^n] \cap N_{-1} \right) \\
      & \subseteq \left<e_0\right>_M [\Delta^n] \cap N_{-1} 
    \end{split}
  \end{gather*}

  Setting $r := \colim_{k\rightarrow \infty} g_k$ yields the retraction, which
  has the property that $r(\left<e\right>_M [\Delta^n] ) \subseteq
  \left<e\right>_M [\Delta^n]  \cap N_{-1}$, so it is $E_M$-controlled when
  $E_M$
  is the control of $M$.

  If $G \neq \{1\}$ we can choose the above lifts equivariantly, e.g.~by
  constructing first a lift for one cell in a $G$-orbit and then extending
  equivariantly.  This shows the general case.
\end{proof}

As a special case it follows that cofibrations have the homotopy extension
property.  The usual arguments
show that being homotopic is an equivalence relation. We obtain:

\begin{lem}[saturation axiom]
  \label{lem:saturation-axiom}
  The homotopy equivalences satisfy the $2$-out-of-$3$ property.  That is,
  assume we have $f\colon A\rightarrow B$ and $g\colon B \rightarrow C$ in
  $\mathcal{C}^G$ and two of $f$, $g$, $g\circ f$ are homotopy
  equivalences, then so is the third.\qed
\end{lem}

\subsection{Cylinders and homotopy equivalences}
\label{subsec:cylinders-homotopy-equivalences}

We need a little bit more homotopy theory to proceed.

\begin{definition}[Deformation retraction]\label{def:deformation-retract}
  Let $i\colon A \rightarrowtail M$ be a cellular inclusion in
  $\mathcal{C}^G$. We can consider $A$ as a submodule of $M$. $A$ is a
  \emph{deformation retract} of $M$ if there is a map $r \colon M \rightarrow A$ such
  that $r \circ i$ is $\id_A$ and $i \circ r$ is homotopic to $\id_M$ relative
  $A$.

  The map $i$ is called the \emph{inclusion} and $r$ is called the
  \emph{retraction} or \emph{deformation retraction}.
\end{definition}

For $f\colon A \rightarrow B$ the target $B$ is a retract of the mapping
cylinder $T(f)$ via the map $p\colon T(f) \rightarrow B$
from~\ref{subsec:cylinder-functor}.\eqref{eq:diag:cylinder-T-is-cylinder-functor}.  

\begin{lem}[cylinder axiom]\label{lem:cylinder-axiom}
  $p\colon T(f) \rightarrow B$ is a homotopy equivalence.
  More precisely, is it even a deformation retraction.
\end{lem}

\begin{proof}
  We only have to prove that $\iota_1 \circ p \colon T(f) \rightarrow B
  \rightarrow T(f)$ is homotopic relative $B$ to $\id_{T(f)}$. Recall that
  $T(f)$ is defined as the pushout of $B \leftarrow A[1] \rightarrow
  A[\Delta^1]$. We see that $\iota_1 \circ p$ is induced by $p_1 \colon
  A[\Delta^1] \rightarrow A[1] \rightarrow A[\Delta^1]$.  It suffices to give
  a homotopy from the identity to $A[\Delta^1] \rightarrow A[1] \rightarrow
  A[\Delta^1]$ which is relative to~$A[1]$.

  But there is a well-known map $\widehat{H}\colon \Delta^1 \times \Delta^1
  \rightarrow \Delta^1$ of simplicial sets inducing such a map.
  Thus the homotopy $H$ which is induced by $\widehat{H}$ is a homotopy relative to $A[1]$ which
  induces the desired homotopy.
\end{proof}
For the next part we need a diagram language.

\subsubsection{Describing maps by diagrams}\label{subsubsec:diagrams}
In the following, we will often have to describe maps of the form $A[\Delta^1
\times \Delta^1] \rightarrow B$, for $A,B \in \mathcal{C}^G$.  We give concise
ways to describe
them.  The simplicial set $\Delta^1 \times \Delta^1$ comes from a simplicial
complex, so it suffices to give compatible maps on the $0$-, $1$- and
$2$-simplices.  We use the pictures 
\begin{equation}\label{diag:subsets-of-I2}
  \sIbig11111,\quad \sIbig11000,\quad \sIbig00100,\quad 
  \sIbig11011 \quad \text{ etc.}
\end{equation}
to denote simplicial subsets of $\Delta^1\times \Delta^1$ which are generated
by the shown $1$- and $2$-simplices.  A $2$-simplex is in the subset if its
boundary is.  The dots are only drawn to specify the corresponding subset and
are only in the subset if they are a boundary.  As an example, we write a map
\begin{equation*}
  A[\Delta^1 \times \{0,1\}] \cup A[0 \times \Delta^1] \rightarrow B
  \quad \text{ as } \quad
  A[\sI11010] \rightarrow B.
\end{equation*}
We can use the same kind of diagrams for other
simplicial sets like $\Delta^1\cup_{\Delta^0} \Delta^1$
and products of them.

If $P$ is a subset of $\Delta^1 \times \Delta^1$ and $A,B \in \mathcal{C}^G$
we want concisely describe maps $A[P] \rightarrow B$.
We often draw diagrams like in~\eqref{diag:subsets-of-I2} and write the maps on
the simplices.  We give examples to illustrate this.

Let $\alpha, \beta\colon A \rightarrow B$ be maps in $\mathcal{C}^G$ and
$H\colon A[\Delta^1] \rightarrow B$ a homotopy from $\alpha$ to $\beta$.
We can extend $H$ to a map $A[\Lambda^2_0] \rightarrow B$ with a trivial
homotopy on the first face and $H$ on the second face.  By horn-filling this
extends to a map $A[\Delta^2] \rightarrow B$.  The following diagrams show the
obtained maps $A[\Delta^1] \rightarrow B$, $A[\Lambda^2_0]
\rightarrow B$, $A[\Delta^2] \rightarrow B$, where
$\Tr$ denotes a trivial homotopy and $\overline{H}$ the
homotopy obtained by horn-filling:
\begin{equation}\label{diag:homotopy-horn-diagram-language}
  \vcenter{\smash{ \xymatrix{
  \alpha \ar[r]|H & \beta}}} \qquad\qquad
  \vcenter{\xymatrix{
	   &  \alpha \\
	   \alpha \ar[ur]|\Tr \ar[r]|H & \beta }} \qquad\qquad
  \vcenter{\xymatrix{
	   &  \alpha \\
	   \alpha \ar[ur]|\Tr \ar[r]|H & \beta \ar[u]|{\overline{H}}}}
\end{equation}
Sometimes we omit the decorations for vertices, as they are uniquely
determined by the decorations on the arrows, and draw dots instead.  We
usually omit the decoration for the $2$-simplices as well, as the actual
maps are usually less important for us.
All of this works for more complicated simplicial sets as long as we can draw
diagrams for them.

\subsubsection{Rectifying homotopy commutative
diagrams}\label{subsubsec:rectifying-diagrams}
If we have the homotopy commutative diagram on the left below, we can turn it
into a strictly commutative diagram on the right below.
\begin{equation*}
  \vcenter{\xymatrix{ A \ar[r]^f \ar[dr]_h & B \ar[d]^g \\ & P }} \qquad
  \rightsquigarrow\qquad 
  \vcenter{\xymatrix{ A \ar[r]^-{\iota_0} \ar[dr]_h & T(f)\ar[d]^{g'} \\ & P }},
\end{equation*}
Define $g'\colon A[\Delta^1] \cup B \rightarrow P$ as induced by the homotopy
and by $g$.  The back inclusion $\iota_1\colon B \rightarrow T(f)$ is a homotopy
equivalence and $g = g' \circ \iota_1$.  

If $f\colon A \rightarrow B$ is a homotopy equivalence, then the
two-out-of-three property implies that then $i\colon A \rightarrow T(f)$ is a
homotopy equivalence.  

\begin{prop}\label{prop:cylinder-deformation-retract}
  If $f\colon A \rightarrow B$ is a homotopy equivalence, then
  $A$ is a deformation retract of $T(f)$ via $\iota_0$.
\end{prop}

We prove the proposition in the rest of this section.  It is an adaption of
the corresponding proof for topological spaces which the author learned from
F.~Waldhausen \cite[pp.~140ff.]{Waldhausen;A-Top-I}.
Let $g$ be the homotopy inverse of $f$, then we have a homotopy commutative
diagram $\id_A = g\circ f$.  By the argument above, we obtain a map $T(f)
\rightarrow A$.  This is the retraction $r$.

\begin{lem}\label{lem:homotopy-not-relative}
  The composition $s\colon T(f)\xrightarrow{r} A \xrightarrow{\iota_0} T(f)$
  is homotopic to the identity.
\end{lem}

\begin{proof}
  By Lemma~\ref{lem:cylinder-axiom}, the composition $t\colon T(f)
  \xrightarrow{p} B
  \xrightarrow{\iota_1} T(f)$ is homotopic to the identity. So we pre- and
  postcompose $s$ with $t$ and obtain a map which is homotopic to $s$. This
  can be written as
  \begin{equation*}
    \xymatrix{
    T(f) \ar[r]^r & A \ar[r]^{\iota_0}\ar[dr]_f & T(f) \ar[d]^p \\
    B \ar[u]^{\iota_1}\ar[ur]_g   &             & B \ar[d]^{\iota_1} \\
    T(f) \ar[u]^{p}               &             & T(f)}
  \end{equation*}
  with compositions identified as $f$ and $g$. But $f\circ g$ is homotopic to
  $\id_B$ by assumption. We are left with
  \begin{equation*}
    \xymatrix{ B \ar[rr]^\id && B \ar[d]^{\iota_1} \\
    T(f) \ar[u]^{p}          && T(f)}
  \end{equation*}
  which is again $t$ and therefore homotopic to $\id_{T(f)}$. Being homotopic
  is an equivalence relation so $s$ is homotopic to $\id_{T(f)}$.  
\end{proof}

The homotopy $H$ of Lemma~\ref{lem:homotopy-not-relative} does not need to be
relative to $A$, but we can improve it as follows.
We have $s \circ \iota_0 =
\iota_0$ as well as $\id_{T(f)} \circ \iota_0 = \iota_0$ so on the endpoints
$H$ is relative to the cellular inclusion $\iota_0 \colon A \rightarrow T(f)$.
We want to make the whole homotopy relative to $A$, i.e.,
\begin{equation*}
  A[\Delta^1] \xrightarrow{\iota_0[\Delta^1]} T(f)[\Delta^1] \xrightarrow{H}
  T(f)
\end{equation*}
should be equal to $A[\Delta^1] \xrightarrow{p} A \xrightarrow{\iota_0} T(f)$.

\begin{lem}\label{lem:cyl-homotopy-relative-A}
  Let $s$ be the map $T(f) \xrightarrow{r} A \xrightarrow{\iota_0} T(f)$.
  There is a homotopy relative $A$ from the identity on $T(f)$ to $s$.
\end{lem}

\begin{proof} We use the diagram notation of~\ref{subsubsec:diagrams}.
  $A$ is a retract of $T(f)$ and $\iota_0\colon A \rightarrow T(f)$
  has the homotopy extension property. We will use this homotopy
  extension property to construct a certain map $T(f)[\Delta^1 \times \Delta^1]
  \rightarrow T(f)$ which restricted to $1\times \Delta^1$ will be the
  desired homotopy from $\id_{T(f)}$ to $s$ relative to~$A$.

  Note that $s$ is an idempotent, i.e.,~$s^2= s$. We use the notation from
  above.  The proof will proceed as follows. We will prescribe the map
  $T(f)[\Delta^1\times \Delta^1] \rightarrow T(f)$ on the subspace $A[\Delta^1
  \times \Delta^1]
  \rightarrowtail T(f)[\Delta^1\times \Delta^1]$ and on the top, bottom and
  left part of $\Delta^1\times \Delta^1 = \sI11111$, i.e., on
  $T(f)[\sI11010]$. Then we check that the two maps are compatible. This will
  give a map
  \begin{equation*}
    { T(f)[\sI11010] \cup A[\sI11111] \rightarrow T(f)}
  \end{equation*}
  which can be extended by the homotopy extension property to the desired map
  $T(f)[\sI11111] \rightarrow T(f)$.\smallskip

  Both maps will be constructed from the same map, which we describe first.
  Horn-filling gives for any map $T(f)[\sI10001] \rightarrow T(f)$ a map
  $T(f)[\sI10011] \rightarrow T(f)$, in particular, we obtain for the first
  diagram below the second one, where $\overline{H}$ is the inverse homotopy.
  Extending this as in the third diagram below gives a map
  $G\colon T(f)[\Delta^1 \times \Delta^1] \rightarrow T(f)$.
  \begin{equation*}
    \xymatrix{ \bdot & \bdot \\
      \bdot \ar[u]|H \ar[ru]|{\Tr} &} \qquad\qquad
    \xymatrix{ 
      \bdot \ar@{.>}[r]|{\overline{H}} & \bdot \\
      \bdot \ar[u]|H \ar[ru]|{\Tr}  & } \qquad\qquad
    \xymatrix{
      \bdot \ar@{.>}[r]|{\overline{H}} & \bdot \\
      \bdot \ar[u]|H \ar[ru]|{\Tr} \ar[r]|\Tr & \bdot \ar[u]|\Tr } 
  \end{equation*}

  Define the map $G_1\colon A[\Delta^1 \times \Delta^1] \rightarrow T(f)$ as the
  restriction of $G$ to $A[\Delta^1 \times \Delta^1]$. Define the map
  $G_2\colon T(f)[\sI11010] \rightarrow T(f)$ as 
  \begin{equation}\label{diag:horn-filling-homotopy}
    \xymatrix{
      \bdot \ar@{->}[r]|{\overline{H}\circ s} & \bdot \\
      \bdot \ar[u]|H  \ar[r]|\Tr & \bdot  } 
  \end{equation}
  so on the $\sI11000$-part it is the restriction of $G$, but on the upper
  part $\sI00010$ we replace the homotopy $\overline{H}$ by $\overline{H}\circ s$. This
  replacement is crucial for the proof.

  We check that these maps are compatible. First $\overline{H}$ is a homotopy
  from $s$ to $\id$, hence $\overline{H}\circ s$ is a homotopy from $s^2$ to
  $s$; but $s^2 = s$ so it agrees with $H$ on the upper left vertex. Second,
  restricted to $A$ the map $s$ is the inclusion $\iota_0\colon A \rightarrow
  T(f)$, hence $\overline{H}\circ s \circ \iota_0 = \overline{H} \circ
  \iota_0$. So $G_1$ and $G_2$ glue to a map
  \begin{equation*}
     { T(f)[\sI11010] \cup A[\sI11111] \rightarrow T(f)}.
  \end{equation*}
  This map can be interpreted as a map $T(f)[0\times \Delta^1] \rightarrow T(f)$ together with a
  homotopy on the submodule $T(f)[0\times\{0,1\}] \cup A[0 \times \Delta^1]$.
  Using the homotopy extension property
  we
  obtain a map $T(f)[\Delta^1\times \Delta^1]
  \rightarrow T(f)$. This map in turn defines a homotopy when restricting
  along $T(f)[1\times \Delta^1] \rightarrow  T(f)[\Delta^1 \times
  \Delta^1]$ (which is $T(f)[\sI00100] \rightarrow  T(f)[\sI11111]$). This homotopy
  starts at the identity, ends at the map~$s$ and is the trivial homotopy
  on~$A$. Hence, it is the desired homotopy.
\end{proof}

\subsection{Pushouts of weak equivalences}
\label{subsec:pushouts-of-weq}

\begin{lem}\label{lem:pushout-along-acyclic-cofibration}
  Let
  \begin{equation*}\xymatrix{
    A \ar[r]\ar[d] & B \ar[d] \\
    C \ar[r]       & D}
  \end{equation*}
  be a pushout diagram in $\mathcal{C}^G$ where $A \rightarrow C$ is a
  cofibration and a homotopy equivalence. Then $B\rightarrow D$ is a homotopy
  equivalence.
\end{lem}

This is a key result on the way to prove the gluing lemma for homotopy
equivalences.  We remark that almost exactly the same proof works if we assume
that $A \rightarrow B$ is a cofibration instead of $A \rightarrow C$.

We can factor $f\colon A\rightarrow C$ into $A \rightarrowtail T(f) \rightarrow C$.
Taking the pushouts along the cellular inclusion $A \rightarrowtail T(f)$
and along the cofibration $A\rightarrowtail C$ gives a commutative diagram
\begin{equation}\label{diag:factoring-pushout-square}
  \xymatrix{
  A   \ar[r] \ar[d]  & B \ar[d] \\
  T(f)\ar[r] \ar[d]  & Q \ar[d] \\
  C   \ar[r]         & D
  }
\end{equation}
and the induced map $Q \rightarrow D$ completes the lower square to a
pushout square. 

The following lemma shows that both maps $B \rightarrow Q$ and $Q
\rightarrow D$ are homotopy equivalences, so their composition $B \rightarrow
D$ is one.

\begin{lem}\label{lem:two-pushouts} 
  In the situation of Diagram~\eqref{diag:factoring-pushout-square} the
  following holds.
  \begin{enumerate}
    \item \label{lem:2p:claim:BQ} The map $B \rightarrow Q$ is a homotopy
      equivalence. \\
      (This uses that $A \rightarrow C$ is a homotopy equivalence.)
    \item \label{lem:2p:claim:QD} The map $Q \rightarrow D$ is a homotopy
      equivalence. \\
      (This uses that $A \rightarrow C$ is a cofibration.)
  \end{enumerate}
\end{lem}

\begin{proof}[Proof of part~(\ref{lem:2p:claim:BQ})]
  By Proposition~\ref{prop:cylinder-deformation-retract}, $A$ is a
  deformation retract of $T(f)$ via the cellular inclusion $A \rightarrowtail
  T(f)$.  One may check directly that then $B$ is a deformation retract of $Q$.
  In particular, $B \rightarrow Q$ is a homotopy equivalence.
\end{proof}

\begin{proof}[Proof of part~(\ref{lem:2p:claim:QD})]
  Written out $Q \rightarrow D$ is the map
  \begin{equation*}
    B \cup_{A[0]} A[\Delta^1] \cup_{A[1]} C \rightarrow B \cup_A C
  \end{equation*}
  induced by $A[\Delta^1] \rightarrow A$. We have to construct a homotopy
  inverse for this map. We will construct a homotopy equivalence $A[\Delta^1]
  \cup_{A[1]} C \rightarrow C$ and its inverse which is relative to $\iota_0^A
  \colon A[0]
  \rightarrowtail A[\Delta^1] \cup_{A[1]} C$, resp.~to $j_A\colon A
  \rightarrowtail C$, hence glues along $A[0] \rightarrow B$ to the desired
  homotopy equivalence
  \begin{equation*}
    B \cup_{A[0]} A[\Delta^1] \cup_{A[1]} C \xrightarrow{\simeq} B \cup_A C\ ,
  \end{equation*}
  as $M \mapsto M[\Delta^1]$ commutes with pushouts. Therefore, we have to construct for
  \begin{equation*}
    e  \colon A[\Delta^1]\cup_{A[1]} C \rightarrow C 
  \end{equation*}
  (induced by $A[\Delta^1]\rightarrow A$) maps
  \begin{equation*}
      g  \colon C \rightarrow A[\Delta^1]\cup_{A[1]} C
  \end{equation*}
  and homotopies
  \begin{gather*}
      H  \colon C[\Delta^1] \rightarrow C\\
      G  \colon \left(A[\Delta^1]\cup_{A[1]} C\right)\![\Delta^1] \rightarrow A[\Delta^1]
      \cup_{A[1]} C
  \end{gather*}
  with the properties
  \begin{align*}
    H_0 &= e\circ g      & H_1 &= \id \\
    G_0 &= \id           & G_1 &= g\circ e \\
    G\circ \iota_0^A[\Delta^1] &= \iota_0^A  & H \circ j_A[\Delta^1] &= j_A \\
    g\circ j_A &= \iota_0^A       & e \circ \iota_0^A &= j_A\ .
  \end{align*}

  Using the homotopy extension property of the cofibration $j_A\colon A \rightarrowtail
  C$ there is a retraction $R \colon
  C[\Delta^1] \rightarrow A[\Delta^1]\cup_{A[1]} C$. Define $g$ as the
  composition
  \begin{equation*}
    C \xrightarrow[\hspace{2em}]{\iota_0^C} C[\Delta^1]
    \xrightarrow[\hspace{2em}]{R} A[\Delta^1] \cup_{A[1]} C.
  \end{equation*}
  We obtain $g \circ j_A = \iota^A_0$.

  Define $H$ as the composition $e \circ R\colon C[\Delta^1] \rightarrow
  A[\Delta^1] \cup_{A[1]} C \rightarrow C$. One may check that $H$ is a homotopy
  from $e\circ g$ to $\id_C$ relative to $A$.

  For the other composition consider the commutative diagram
  \begin{equation*}
    \xymatrix@C=12pt{  
    A[\Delta^1]\cup_{A[1]} C  \ar[rd]_j \ar[rr]^-e  && 
       C \ar[rr]^-g\ar[rd]_{\iota_0^C}   && 
       A[\Delta^1] \cup_{A[1]} C \\
    &  C[\Delta^1] \ar[ur]_{\pr}\ar@{-->}[rr]  && 
       C[\Delta^1] \ar[ur]_R
    }
  \end{equation*}
  where dashed map is the projection to $C[0]$. It is homotopic relative
  $C[0]$ to the
  identity. This gives a homotopy $G$ from the identity to the composition
  $g\circ e$, using that $R$ is a retraction for $j$. One may check that $G$ is
  relative to $A[0]$.

  This shows that $e$ is a homotopy equivalence and therefore makes $Q
  \rightarrow D$ into one.
\end{proof}

\subsection{The gluing lemma}
\label{subsec:gluing-lemma}

We now prove the gluing lemma in $\mathcal{C}^G$.  It is the essential
ingredient to prove Theorem~\ref{thm:c-g-waldhausen-category}, which says that
the homotopy equivalences are a category of weak equivalences for
$\mathcal{C}^G$.
\begin{lem}[gluing lemma]\label{lem:gluing-lemma}
  If we have the diagram in $\mathcal{C}^G$
  \begin{equation*}
    \xymatrix{
    B \ar[d]^{\sim}    & A \ar[d]^{\sim} \ar@{ >->}[l]\ar[r]   & C \ar[d]^{\sim} \\
    B'          & A'       \ar@{ >->}[l]\ar[r]   & C'}
  \end{equation*}
  with $A \rightarrowtail B$ and $A' \rightarrowtail B'$ cofibrations and
  all three vertical arrows are homotopy equivalences, then the induced map
  \begin{equation*}
    B \cup_A C \rightarrow B' \cup_{A'} C' 
  \end{equation*}
  on the pushouts is also a homotopy equivalence. 
\end{lem}

\begin{proof}
  It it shown in Lemma~II.8.8
  in~\cite[p.~127]{goerss-jardine;simplicial-homotopy-theory} that a
  \emph{category of cofibrant objects} satisfies the gluing lemma.  We recall
  that notion from \cite[p.~122]{goerss-jardine;simplicial-homotopy-theory}.
  It was first introduced by Kenneth Brown in
  \cite{brown;abstract-homotopy-sheaf-cohomology}, where he treats the dual
  version.
   
  A \emph{category of cofibrant objects} is a category $\mathcal{D}$ which
  satisfies the following axioms.
  \begin{enumerate} \setcounter{enumi}{-1}
    \item \label{it:cofibrant-obj:coproducts} 
      The category contains all finite coproducts.
    \item \label{it:cofibrant-obj:two-of-three} 
      The $2$-out-of-$3$ property holds for weak equivalences.
    \item \label{it:cofibrant-obj:composition-cofibrations} 
      The composition of cofibrations is a cofibration, isomorphisms are
      cofibrations.
    \item \label{it:cofibrant-obj:pushouts-of-cofibrations} 
\label{it:rem:cat-of-cofib-pushout}
      Pushout diagrams of the form
      \begin{equation*}
        \xymatrix{A \ar[r] \ar@{->}[d]^i  & B \ar@{->}[d]^{i_*} \\
                  C \ar[r]                  & D}
      \end{equation*}
      exist when $i$ is a cofibration. In this case $i_*$ is a cofibration which is
      additionally a weak equivalence if $i$ is one.
    \item \label{it:cofibrant-obj:cylinder-object} 
      For each object there is a cylinder object.
    \item \label{it:cofibrant-obj:cofibrant-objects} 
      For each $X$ the unique map $* \rightarrow X$ from the initial
      object is a cofibration.
  \end{enumerate}
  The notion of a cylinder object in
  \cite[p.~123]{goerss-jardine;simplicial-homotopy-theory} is slightly
  different from our notion, but as the cylinder
  axiom~\ref{lem:cylinder-axiom} holds our cylinder functor from
  Section~\ref{subsec:cylinder-functor} applied to the identity yields a
  cylinder object in the sense of
  \cite[p.~123]{goerss-jardine;simplicial-homotopy-theory}.  This verifies
  condition (\ref{it:cofibrant-obj:cylinder-object}.).

  We have already established that the other conditions hold for
  $\mathcal{C}^G$: as $\mathcal{C}^G$ is a category with cofibrations
  (Section~\ref{subsubsec:subcategory-cofibrations}) every object is cofibrant
  (\ref{it:cofibrant-obj:cofibrant-objects}.),
  we have twofold and hence finite coproducts
  (\ref{it:cofibrant-obj:coproducts}.), isomorphisms as well as
  the composition of cofibrations are cofibrations
  (\ref{it:cofibrant-obj:composition-cofibrations}.), and the pushout along a
  cofibration exists and is a cofibration
  (half of \ref{it:cofibrant-obj:pushouts-of-cofibrations}.).
  The 2-out-of-3 property (\ref{it:cofibrant-obj:two-of-three}.) is the
  Saturation Axiom~\ref{lem:saturation-axiom}.   Finally, the pushout of a
  cofibration which is a homotopy equivalence is a homotopy equivalence
  (\ref{it:cofibrant-obj:pushouts-of-cofibrations}.) by
  Lemma~\ref{lem:pushout-along-acyclic-cofibration}.
  Hence, by Lemma II.8.8 of~\cite{goerss-jardine;simplicial-homotopy-theory},
  the gluing lemma holds in $\mathcal{C}^G$.
\end{proof}


\subsection{The extension axiom}
\label{subsec:extension-axiom}

Next we want to prove the extension axiom for our category
$\mathcal{C}^G(X;R)$.  We will need to use explicitly that we can sum maps.
Unlike the results in the previous sections, the extension axiom does not hold
in Waldhausen's category of retractive spaces over a point,
see~\cite[1.2]{Waldhausen;spaces;;;;1985}.

Let $\mathcal{C}$ be a category with cofibrations. A \emph{cofiber sequence}
in $\mathcal{C}$ is a sequence $A \rightarrowtail B \twoheadrightarrow C$ in
$\mathcal{C}$ where $A \rightarrowtail B$ is a cofibration and $B
\twoheadrightarrow C$ is isomorphic to the map $B \twoheadrightarrow B/A := B
\cup_A *$.

\begin{definition}[extension axiom]
  \label{def:extension-axiom}
  A subcategory $w\mathcal{C}$ of weak equivalences of $\mathcal{C}$ satisfies
  the \emph{extension axiom} if for each map of cofiber sequences
  \begin{equation*}
    \xymatrix{
    A \ar[d]^{f_A} \ar@{ >->}[r]  & B \ar[d]^{f_B} \ar@{->>}[r] & C \ar[d]^{f_C} \\
    A'             \ar@{ >->}[r]  & B'             \ar@{->>}[r] & C'}
  \end{equation*}
  where $f_A$ and $f_C$ are weak equivalences, the map $f_B$ is a weak
  equivalence. Sometimes $B$ (resp.~$f_B$) is called an \emph{extension} of
  $A$ by $C$ (resp.~of $f_A$ by $f_C$).
\end{definition}

We first need a relative homotopy lifting property.  We directly prove a more
general horn-filling property.

\begin{lem}[Horn-filling relative to a map]\label{lem:kan-fibration-horn-filling}\label{lem:kan-fibration-lifting}
  Let $A \rightarrowtail M$ be a cellular inclusion in $\mathcal{C}^G$. Let $U
  \rightarrowtail P$ also be a cellular inclusion in $\mathcal{C}^G$ and let
  $P \twoheadrightarrow Q := P/U$ the quotient map. Then $A \rightarrow M$
  has the \emph{relative horn-filling property} with respect to $P
  \rightarrow Q$. This means, given a horn $\Lambda^n_i \subseteq \Delta^n$
  and a solid commutative diagram of controlled maps
  \begin{equation}\label{eq:diag-hep-lifting}
    \vcenter{\xymatrix{
        M[\Lambda^n_i] \cup A[\Delta^n]\ar[r] \ar[d]   & P \ar[d]  \\
        M[\Delta^n]  \ar[r] \ar@{-->}[ur]          & Q}}
  \end{equation}
  then the dashed lift exists.
\end{lem}

The case $n=1$, $i=1$ gives the homotopy lifting property with respect
to $P \rightarrow Q$.  The proof proceeds similarly to the proof of
Lemma~\ref{lem:horn-filling-relative}.  It is not stated there in the full
generality, as we need the generalized version only in this section.

We will need the following extra ingredient: any surjective map $B
\twoheadrightarrow C$ of simplicial abelian groups is a Kan fibration.  This
follows e.g. from \cite[Corollary V.2.7,
p.~263]{goerss-jardine;simplicial-homotopy-theory}.  Consequently, for a
cellular inclusion of simplicial $R$-modules $A \rightarrowtail B$, the map $B
\twoheadrightarrow B/A$ is a Kan fibration of simplicial sets.

\begin{proof}[Proof of Lemma~\ref{lem:kan-fibration-horn-filling}]
  The proof is very similar to the proof of
  Lemma~\ref{lem:horn-filling-relative}, but more involved.  The main point is
  that we need to find a lift relative to the map $P \rightarrow Q$.  To still
  keep the control, we have to strengthen the induction hypothesis.
  
  We first treat the case $G=\{1\}$.  Let $B_k := A \cup M_k$, where $M_k$ is
  the submodule of $M$ generated by all cells of dimension $\leq k$.  We do
  induction over $k$.  We abbreviate $N_k := \left( M[\Lambda^n_i]
  \cup_{B_k[\Lambda^n_i]} B_k[\Delta^n]\right)$, $N_{\infty} := M[\Delta^n] =
  \bigcup_k N_k$.  We have to find a lift in the diagram (which also fixes our
  notation for the maps)
  \begin{equation}\label{eq:horn-fill-map-lift}
    \vcenter{\xymatrix{
    N_{-1}\ar[r]^{f} \ar[d]^i   & P \ar[d]^p  \\
    N_{\infty} \ar[r]^h \ar@{-->}[ur]          & Q}}.
  \end{equation}

  We need to be able to restrict $p$
  ``locally'', such that it is still a fibration.  It suffices that we
  ``locally''
  construct maps which are surjections of abelian simplicial
  groups after forgetting control and $R$-module structure.  We make the
  following choices.  For each $e_Q \in \diamond_R Q$ choose an $\vartheta(e_Q)
  \in \diamond_R P$ with $p(\vartheta(e_Q)) = e_Q$.  Such a map
  $\vartheta\colon \diamond_R Q \rightarrow \diamond_R P$ exists as
  $\diamond_R P \cong \diamond_R Q \cup \diamond_R U$.
  
  We assume the following induction hypothesis:
  \begin{enumerate}
    \item There is a map $g_k\colon N_k \rightarrow P$ which extends $f$ over
      $h$, i.e., is a partial lift in the diagram~\eqref{eq:horn-fill-map-lift}.
    \item For each $e_0\in \diamond_R M$ the map $g_k$ restricts to
      \begin{multline*}
        \left<e_0\right>_M [\Delta^n]\ \cap\ 
        N_k
        \xrightarrow[\hspace{2em}]{g_k} \\
        \left<f\left(\left<e_0\right>_M [\Delta^n]  \cap  
        N_{-1} \right)\right>_P\ \cup \ 
        \bigcup
        \Bigl\{\left<\vartheta(e_Q)\right>_P\, \Big|\, e_Q\in \diamond_R  
        \left<h(\left<e_0\right>_{M}[\Delta^n]) \right>_Q\Bigr\}
      \end{multline*}
  \end{enumerate}
  The second condition implies that $g_k$ is $E_f \circ E_M \cup E_M \circ
  E_h \circ E_P$-controlled.  Roughly speaking it ensures that the lift does
  not hit a module which is uncontrollably large.  Here is a reason for why it
  has to be at least that size.  First we must allow a cell $e_0$ to at least
  hit the image of $f$ of the part of the cell
  intersecting $N_{-1}$.  Second, the cell hits certain elements in $Q$, so we
  must have possible lifts for all of them.

  We do induction over $k$.  We can attach cells of the same dimension
  independently, so we only treat the case of attaching one cell $e$ of
  dimension $k$.  As before the left square of the following
  diagram is a pushout.
  \begin{equation}\label{eq:rel-horn-filling-map}
    \vcenter{\xymatrix{
    R[\Delta^k \times \Lambda^n_i \cup \partial \Delta^k \times \Delta^n]
    \ar[d]\ar[r]^{\partial e_*} &
      B_k[\Lambda^n_i] \cup_{B_{k-1}[\Lambda^n_i]}  
      B_{k-1}[\Delta^n] \ar[d]\ar[r]^-{g_{k-1}} & 
      P \ar[d]^p\\
      R[\Delta^k \times \Delta^n] \ar[r]^{e_*} & B_k[\Delta^n] \ar@{-->}[ur]
      \ar[r]^h &
      Q} }
  \end{equation}
  We can replace the middle column by $N_{k-1} \rightarrow N_k$ and the
  diagram remains commutative and the left square a pushout.  We only have
  to find a lift in the outer
  diagram of~\eqref{eq:rel-horn-filling-map}.  We abbreviate
  \begin{equation*}
    \begin{split}
        P^f(e) &:=\left<f\left(\left<e\right>_M [\Delta^n]\cap  
        N_{-1} \right)\right>_P  \\
        P^h(e) & :=\bigcup
        \Bigl\{\left<\vartheta(e_Q)\right>_P\, \Big|\, e_Q\in \diamond_R  
        \left<h(\left<e\right>_{M}[\Delta^n]) \right>_Q\Bigr\}
    \end{split}
  \end{equation*}
  Both are cellular submodules of $P$.  We obtain a factorization of the outer
  diagram of~\eqref{eq:rel-horn-filling-map}
  \begin{equation*}
    \vcenter{\xymatrix{
    R[\Delta^k \times \Lambda^n_i \cup \partial \Delta^k \times \Delta^n]
    \ar[d]\ar[r]^-{\partial e_*} & 
    P^f(e) \cup P^h(e) \ar[r]\ar[d] & P \ar[d]^p\\
      R[\Delta^k \times \Delta^n] \ar[r]^{e_*} &
      \left<h(\left<e\right>_M[\Delta^n]) \right>_Q \ar[r]
      & Q } }
  \end{equation*}
  by the induction hypothesis and it suffices to find a lift in the left
  diagram.  By the fundamental lemma, and because the middle column in the
  diagram has bounded support on some $\{x\}^{E}$, it suffices to find a
  (dashed) lift in the diagram of simplicial sets 
  \begin{equation*}
    \xymatrix{
    \Delta^k \times \Lambda^n_i \cup \partial \Delta^k \times \Delta^n
      \ar[d]\ar[r] &
      P^f(e) \cup P^h(e) \ar[d]\\
    \Delta^k \times \Delta^n \ar@{-->}[ur] \ar[r] &
      \left<h(\left<e\right>_M[\Delta^n]) \right>_Q }
  \end{equation*}
  Such a lift exists if the right map is a Kan Fibration.  But as it is a
  homomorphism of simplicial abelian groups it suffices to show that it is
  surjective.  But $P^h(e)
  \rightarrow \left<h(\left<e\right>_M[\Delta^n]) \right>_Q$ is already
  surjective by construction, as $P^h(e)$ has exactly one cell $e_P$ for each
  cell $e_Q$ of $\left<h(\left<e\right>_M[\Delta^n]) \right>_Q$ and by
  definition of $\vartheta$ the cell $e_P$ is mapped to $e_Q$.  This gives the lift
  $g_k$, and by construction it satisfies the first condition of the
  induction hypothesis for $k$.

  The second condition is satisfied for $e$ by construction and for $e_0$ with
  $e \not \in \left<e_0\right>_M$ by the induction hypothesis for $k-1$.
  Otherwise $\left<e\right>_M \subseteq \left<e_0\right>_M$ which implies
  $P^f(e) \subseteq P^f(e_0)$ and $P^h(e) \subseteq P^h(e_0)$.  Then $g_k$
  restricted to $\left<e_0\right>_M[\Delta^n]$ factors as
  \begin{multline*}
      \left<e_0\right>_M[\Delta^n] \cap N_k = 
      \left<e_0\right>_M[\Delta^n] \cap (N_{k-1} \cup
      \left<e\right>_M[\Delta^n]) \\
      \xrightarrow[\hspace{2em}]{g_k} 
      P^f(e_0) \cup P^h(e_0)\  \cup\ P^f(e) \cup P^h(e) 
      =  P^f(e_0) \cup P^h(e_0)
  \end{multline*}
  Therefore, the second condition is also satisfied.

  If $G \neq \{1\}$ we can choose the above lifts equivariantly, e.g.~by
  constructing first a lift for one cell in a $G$-orbit and then extending
  equivariantly.  This shows the general case.
\end{proof}

\begin{lem}[extension axiom]\label{lem:waldh-cat:extension-axiom}
  Let
  \begin{equation*}
    \xymatrix{
    A \ar@{ >->}[r] \ar[d]^\sim   & B \ar[d] \ar@{->>}[r]  & C \ar[d]^\sim \\
    A'\ar@{ >->}[r]               & B'       \ar@{->>}[r]  & C'}
  \end{equation*}
  be a map of cofiber sequences in $\mathcal{C}^G$. Assume that $A \rightarrow A'$ and
  $C\rightarrow C'$ are homotopy equivalences. Then $B\rightarrow B'$ is a
  homotopy equivalence.
\end{lem}

\begin{proof}
  We can factor the vertical maps functorially by using the cylinder functor.
  As a cylinder functor is exact it respects the cofiber sequences.  We obtain a
  diagram
  \begin{equation*}
    \xymatrix{
    A \ar@{ >->}[r] \ar[d]^\sim   & B \ar[d] \ar@{->>}[r]  & C \ar[d]^\sim \\
    T_A \ar@{ >->}[r] \ar[d]^\sim & T_B \ar[d]^\sim \ar@{->>}[r]  & 
          T_C \ar[d]^\sim \\
    A'\ar@{ >->}[r]               & B'       \ar@{->>}[r]  & C'}.
  \end{equation*}
  By Proposition~\ref{prop:cylinder-deformation-retract}, $A$ and
  $C$ are deformation retracts of $T_A$ and $T_C$, respectively, with the
  inclusions being the left and the right vertical upper maps.
  What remains to be shown is that the vertical upper middle map is a
  homotopy equivalence. This is proved in
  Lemma~\ref{lem:extension-mapping-cylinder} below, where it is shown that 
  $B$ is a deformation retract of $T_B$.
\end{proof}

\begin{lem} \label{lem:extension-mapping-cylinder}
  Assume we have a cofiber sequence $A \rightarrowtail B \twoheadrightarrow
  \overline{B}$ in $\mathcal{C}^G$ where $\overline{B} = B/A$ for brevity.
  Suppose we have a diagram
  \begin{equation*}
    \xymatrix{
    A \ar@{ >->}[d]\ar@{ >->}[r] & B \ar@{ >->}[d]\ar@{->>}[r] & \overline{B} \ar@{ >->}[d] \\
    T_A     \ar@{ >->}[r]        & T_B     \ar@{->>}[r]        & T_{\overline{B}} }
  \end{equation*}
  in $\mathcal{C}^G$ where the horizontal lines are cofiber sequences and the
  vertical arrows are cellular inclusions. 
  Suppose that $A$ and $\overline{B}$ are deformation retracts of $T_A$ and
  $T_{\overline{B}}$ with inclusions the left and right vertical maps.
  Then $B$ is a deformation retract of $T_B$ with inclusion the middle
  vertical map.
\end{lem}

We prove a slightly stronger statement than
Lemma~\ref{lem:extension-mapping-cylinder}:

\begin{lem}\label{lem:extension-general}
  Assume that we are in the situation of
  Lemma~\ref{lem:extension-mapping-cylinder}. Let $D_0$ be a cellular
  submodule of $D$.  Then each controlled map $(D, D_0) \rightarrow (T_B,B)$
  of pairs in $\mathcal{C}^G$ is controlled homotopic relative $D_0$ to a map
  into $B$.
\end{lem}

\begin{proof}[Proof of Lemma~\ref{lem:extension-mapping-cylinder} using
  \ref{lem:extension-general}]
  By Lemma~\ref{lem:extension-general}, the map $\id\colon (T_B, B) \rightarrow (T_B,
  B)$ is controlled homotopic relative $B$ to a map $T_B \rightarrow B$. This
  is the desired deformation retraction.  
\end{proof}

\begin{rem}[Toy situation]
  Assume we have a commutative diagram of abelian groups
  \begin{equation*}
    \xymatrix{
    A \ar@{->}[d]\ar@{->}[r] & B \ar@{->}[d]_{f_B}\ar@{->}[r] & \overline{B} \ar@{->}[d] \\
    A'     \ar@{->}[r]     & B' \ar@{->}[r]             & \overline{B}' }
  \end{equation*}
  where the horizontal lines are short exact sequences. Assume that the outer
  maps are surjective.  We want to show that the middle map is surjective.
  The proof proceeds exactly like the proof of
  Lemma~\ref{lem:extension-general}, but is easier.  We give it to help with
  the general proof.

  Let $\alpha$ be an element in $B'$. We will denote the constructed elements
  by consecutive Greek letters and denote projections to the quotient by a bar.
  So $\overline{\alpha}$ is an element in $\overline{B}'$. As $\overline{B}
  \rightarrow \overline{B}'$ is surjective there is an element $\beta$ in
  $\overline{B}$ which maps to $\overline{\alpha}$. As $B \rightarrow
  \overline{B}$ is surjective there is an element $\gamma$ in $B$ which maps to
  $\beta\in \overline{B}$. The elements $f_B(\gamma)$ and $\alpha$ do not need
  to be equal in $B'$, but they become equal when projected to $\overline{B}'$,
  so $\alpha - f_B(\gamma) $ factors through $A' \rightarrowtail B'$. As $A
  \rightarrow A'$ is surjective there is an element $\delta$ in $A$ which maps
  to $\alpha - f_B(\gamma)$ in $B'$. Hence, considered in $B$, $f_B(\delta +
  \gamma)$ equals $\alpha$.
\end{rem}

Lemma~\ref{lem:kan-fibration-lifting} applies to the maps $B \rightarrow
\overline{B}$ and $T_B \rightarrow T_{\overline{B}}$, so we have the relative
homotopy lifting property with respect to these maps.

\begin{proof}[Proof of Lemma~\ref{lem:extension-general}] 
  Let $\alpha\colon (D, D_0) \rightarrow (T_B, B)$ be a controlled map. This
  gives a map $\overline{\alpha}$ into $(\overline{B}, T_{\overline{B}})$.
  As $\overline{B}$ is a deformation retract of $T_{\overline{B}}$ we obtain a
  homotopy $\overline{H}\colon D[\Delta^1] \rightarrow T_{\overline{B}}$ from
  $\overline{\alpha}$ to a map into $\overline{B}$ which is trivial on $D_0$.
  It comes from the deformation of $(\overline{B}, T_{\overline{B}})$
  precomposed with $\alpha$. Lemma \ref{lem:kan-fibration-lifting} applies to
  the map $T_B \rightarrow T_{\overline{B}}$. We obtain a lift $H$
  of $\overline{H}$, relative to $\alpha$ and $D_0$.
  \begin{equation*}
    \xymatrix{ 
    D_0[\Delta^1] \cup D[0]  \ar[r]^-\alpha \ar[d]
    &  T_B \ar@{->>}[d]   \\
    D[\Delta^1] \ar@{-->}[ur]^H \ar[r]^{\overline{H}} & T_{\overline{B}}
    }
  \end{equation*}

  This is a homotopy from $\alpha$ to a better map, call it $\beta \colon
  D \rightarrow T_B$. However, $\beta$ might not yet factor through $B$ in
  which case the lemma would follow.
  But composition with $T_{B\vphantom{\overline{B}}} \rightarrow
  T_{\overline{B}}$ gives a map $\overline{\beta}$ to $T_{\overline{B}}$ which
  factors through $\overline{B}$.  
  \begin{equation*}
    \xymatrix{
        & B \ar@{->>}[r] \ar@{->}[d]     & \overline{B} \ar[d] \\
      D \ar[r]_\beta \ar@/_1.2ex/@{->}[urr]_(.75){\overline{\beta}}& 
          T_B  \ar@{->>}[r] & T_{\overline{B}}}
  \end{equation*}
  Using Lemma \ref{lem:kan-fibration-lifting} again this time for $B
  \twoheadrightarrow \overline{B}$ and the trivial homotopy of $\overline{\beta}$ in
  $\overline{B}$ we obtain some lift of $\overline{\beta}$ to $B$, call it $\gamma$.
  \begin{equation*}
    \xymatrix{
    {*}  \ar[r] \ar@{ >->}[d]     & B \ar@{->>}[r]  &  \overline{B} \\
      D \ar@{-->}[ur]^\gamma \ar@/_/[rru]_(.75){\overline{\beta}} &}
  \end{equation*}
  It follows that the difference $\beta - \gamma\colon D \rightarrow
  T_B$ is zero when composed with $T_B \rightarrow T_{\overline{B}}$. Hence, it
  factors through $T_A$. As the restrictions of $\beta$ and $\gamma$ to $D_0$
  both lie in $B$ the restriction of $\beta - \gamma$ to $D_0$ factors through
  $A$. So $\beta - \gamma$ gives a map $(D, D_0) \rightarrow (T_A, A)$. We can
  show the situation by the following commuting diagrams.
  \begin{equation*}
    \xymatrix{T_A \ar@{ >->}[r]   & T_B \ar@{->>}[r]  & T_{\overline{B}} \\
    & D \ar[u]_(.6){\beta-\gamma} \ar[ur]_{0}
    \ar@{-->}[ul]^{\beta-\gamma} & }, \qquad
    \xymatrix{A \ar@{ >->}[r]   & B \ar@{->>}[r]  & {\overline{B}} \\
    & D_0 \ar[u]_(.6){\beta-\gamma} \ar[ur]_{0}
    \ar@{-->}[ul]^{\beta-\gamma} & }, \qquad
    \xymatrix{ 
      A \ar[r]    & T_A   \\
      D_0 \ar@{-->}[u]^{\beta-\gamma} \ar[r]  & D \ar@{-->}[u]^{\beta-\gamma} }
  \end{equation*}
  Hence, as $A \rightarrow T_A$ is a deformation retraction, there is a
  homotopy $G$ relative to $D_0$ of $\beta - \gamma$ to a map
  into $A$. It comes from the deformation of $(A,T_A)$ precomposed with
  $\beta-\gamma$. Call the resulting map $\delta\colon D \rightarrow A$. Via the
  inclusion $(T_A,A) \rightarrow (T_B,B)$ the map $G$ can be viewed as a
  homotopy to $T_B$ with:
  \begin{equation*}
    \begin{array}[]{ll}
      G  \colon & D[\Delta^1] \rightarrow T_B \\
      G_{|0} & =  \beta - \gamma \\
      G_{|1} & = \delta \\
      G_{|D_0[\Delta^1]} & = \beta-\gamma_{|D_0}
    \end{array}.
  \end{equation*}
  Therefore, $G + \gamma \colon D[\Delta^1] \rightarrow T_B$ is a homotopy from
  $\beta$ to $\delta + \gamma$, where $\delta$ and $\gamma$ factor through $B$
  so the sum also factors through~$B$.  Furthermore, the homotopy is
  trivial on $D_0$.  Concatenating the two homotopies $H$ and $G$ thus gives
  a homotopy relative $D_0$ from $\alpha$ to a map into $B$. This is what we
  wanted to show.

  Note that all maps above are in fact in $\mathcal{C}^G$, because maps in
  $\mathcal{C}^G$ form an abelian group and being homotopic relative a
  subspace is an equivalence relation in~$\mathcal{C}^G$.
\end{proof}

\begin{proof}[Proof of Theorem~\ref{thm:c-g-waldhausen-category}]
  We established that $\mathcal{C}^G(X, \mathcal{E},
  \mathcal{F};R)$ has the structure of a category with cofibrations
  $co\mathcal{C}^G$ in~\ref{subsec:cofibrations}.  The subcategory of homotopy
  equivalences in $\mathcal{C}^G$ contains all the isomorphisms and satisfies
  the gluing Lemma~\ref{lem:gluing-lemma}. In other words, it is a subcategory
  of weak equivalences for $(\mathcal{C}^G,co\mathcal{C}^G)$.  The the
  saturation axiom~\ref{lem:saturation-axiom}, the cylinder
  axiom~\ref{lem:cylinder-axiom} and extension
  axiom~\ref{lem:extension-general} hold.
  This settles Theorem~\ref{thm:c-g-waldhausen-category}.
\end{proof}

\section{Proofs III: finiteness conditions}
\label{sec:proofs-iii:finiteness-conditions}
\subsection{Finiteness conditions}

Let $(X, \mathcal{E}, \mathcal{F})$ be a $G$-equivariant control space.
We have shown that the category $\mathcal{C}^G(X, \mathcal{E}, \mathcal{F}; R)$ has the
structure of a Waldhausen category.  The
weak equivalences satisfy the saturation and extension axiom.  The category
$\mathcal{C}^G(X;R)$ has 
a cylinder functor which satisfies the cylinder axiom.

Let $\mathcal{C}^G_?$ be a full subcategory of $\mathcal{C}^G(X;R)$.  Define
map in $\mathcal{C}^G_?$ to be a cofibration, resp.~a weak equivalence, if it is
a cofibration, resp.~a weak equivalence in $\mathcal{C}^G$.  
Assume the following three conditions as satisfied:
  \begin{itemize}
    \item[(C0)] \label{enum:condition:c0}
      We have $* \in \mathcal{C}^G_{?}$.
    \item[(C1)] \label{enum:condition:c1} 
      For $C \leftarrow A \rightarrowtail B$ in $\mathcal{C}^G_{?}$ the
      pushout is in $\mathcal{C}^G_{?}$.
    \item[(C2)] \label{enum:condition:c2}
      For $A$ in $\mathcal{C}^G_{?}$, $A[\Delta^1]$ is in $\mathcal{C}^G_{?}$.
  \end{itemize}
Then $\mathcal{C}^G_?$ has the structure of a Waldhausen category, it has a cylinder functor and satisfies the saturation,
extension and the cylinder axiom.  We show that the bl-finite, homotopy bl-finite
and homotopy bl-finitely dominated objects satisfy these conditions.

Recall that $(M,\kappa) \in \mathcal{C}^G(X;R)$ is \emph{bl-finite} if
it is finite-dimensional and locally finite, i.e., each $x \in X$ has a
neighborhood $U$ such that $\kappa^{-1}(U) \subseteq \diamond_R M$ is a
finite cellular $R$-module (as uncontrolled module).
$M$ is \emph{homotopy bl-finite} if there is a weak equivalence $M
\xrightarrow{\sim} M'$ where $M'$ is bl-finite.  $M$ is \emph{homotopy
bl-finitely
dominated} if it is a strict retract of a homotopy bl-finite object.  The
corresponding categories obtain as decoration a subscript $f$, $hf$ or $hfd$,
respectively.  While $\mathcal{C}^G_{hf}$ and $\mathcal{C}^G_{hfd}$ contain
all objects isomorphic to one already in the category,
this is not true for $\mathcal{C}^G_f$.  Hence, in
$\mathcal{C}^G_{hf}$ we can define the Waldhausen structure only referring to
$\mathcal{C}^G_{hf}$: a cofibration is a map which is isomorphic (in
$\mathcal{C}^G_{hf}$) to a cellular inclusion of homotopy bl-finite controlled
simplicial $R$-modules over $X$ and a weak equivalence is a homotopy equivalence
in $\mathcal{C}^G_{hf}$.  The same holds for $\mathcal{C}^G_{hfd}$.  We prove
in Lemma~\ref{lem:finite-cofibrations-are-isom-to-finite-cell-inclusions}
below that it also holds in $\mathcal{C}^{G}_f$, which requires $X$ to be
Hausdorff.

Note that obviously $*\in \mathcal{C}^G_f \subseteq
\mathcal{C}^G_{hf}\subseteq \mathcal{C}^G_{hfd}$.

\subsection{Bl-finite modules}
There is an obvious notion of a \emph{set over $X$} and of a \emph{controlled
map of sets over $X$}.  Let $(M, \diamond_R M, \kappa)$ be a controlled module
over $X$.  Our prime example of a set over $X$ is $(\diamond_R M, \kappa)$.  If
$(M,\kappa_R^1)$, $(M,\kappa_R^2)$ are controlled modules over $X$ such that
$(\diamond_R M, \kappa_R^1)$ and $(\diamond_R M,\kappa_R^2)$ are controlled
isomorphic as sets over $X$ then $(M,\kappa_R^1)$ and $(M,\kappa_R^2)$ are
controlled isomorphic.

Note that a controlled module $(M, \diamond_R M, \kappa)$ over $X$ is locally
finite if $(\diamond_R M, \kappa)$ is \emph{a locally finite set} over $X$,
i.e., each $x\in X$ has a neighborhood $U$ with $\kappa^{-1}(U)\subseteq
\diamond_R M$ being finite.

\begin{rem}
  Modules isomorphic to bl-finite modules do not need to be bl-finite
  again, if the control space is not ``good''.   An example is
  $\Real\smallsetminus \{0\}$ with metric control.

  A control space $(X,\mathcal{E}, \mathcal{F})$ is called \emph{proper}, if for each
  compact subset $K$ and $E\in \mathcal{E}$, $F\in \mathcal{F}$ we have that
  $(F\cap K)^E \cap F$ is contained in a compact set.  If the control space is
  proper then modules isomorphic to bl-finite modules are again bl-finite, see also
  Remark~\ref{rem:proof-of-finite-and-proper}.
\end{rem}

We want to show that $\mathcal{C}^G_f$, the full subcategory of bl-finite
objects, is indeed a Waldhausen category.  Is is clear that $A[\Delta^1]$ is again bl-finite if $A$ is
bl-finite.  The main part is to show that the pushout of $C \leftarrow A
\rightarrowtail B$ exists when $A$, $B$, $C$ are bl-finite modules and $A
\rightarrowtail B$ is a cofibration \emph{in $\mathcal{C}^G$}.  If we know
that the pushout is isomorphic to a cellular inclusion of bl-finite modules
then we are done.

But that is not obvious, we mentioned above that not every module which is
isomorphic to a bl-finite module needs to be bl-finite again.  Furthermore, the problem
is not only that we could change the map $\kappa$ to $X$, but we could
have a different cellular structure.

\begin{lem}\label{lem:finite-cofibrations-are-isom-to-finite-cell-inclusions}
  Let $f'\colon A \rightarrow B$ be a cofibration in $\mathcal{C}^G_f$.  Then it is
  isomorphic to a cellular inclusion in $\mathcal{C}^G_f$.
\end{lem}

The proof is fairly complicated, so we only provide a sketch. The first step
is the following:

By definition, $f'$ is only isomorphic to a cellular inclusion in
$\mathcal{C}^G$, which does not need to be in $\mathcal{C}^G_f$ by
the Remark above.

By Lemma~\ref{lem:cellular-pushouts-in-CG}, the pushout of $f'$ along $\id_A$
can be chosen as
\begin{equation*}
  \xymatrix{
  A \ar[r]^{f'} \ar[d]^{\id_A}  & B \ar[d] \\
  A \ar[r]^f   & D}
\end{equation*}
such that $f$ is a cellular inclusion.  Then $D$ is isomorphic to the bl-finite
module $B$, but need not be bl-finite itself.  We show in the next lemma that
$D$ can indeed be made into a bl-finite module.  That lemma finishes the proof.

\begin{lem}
  Let $f\colon A \rightarrow D$ be a cellular inclusion in $\mathcal{C}^G$
  such that $A$ is bl-finite and $(D,\kappa^D)$ is isomorphic to a bl-finite module
  $(B,\kappa^B)$.  Then there is a control map $\overline{\kappa}^D\colon
  \mbox{$\diamond_R D$}
  \rightarrow X$ such that $(D, \overline{\kappa})$ is a bl-finite module which
  is isomorphic to $(D, \kappa^D)$.

  It follows that $A \rightarrow (D, \overline{\kappa})$ is isomorphic to a
  cellular inclusion in $\mathcal{C}^G_f$.
\end{lem}

\begin{proof}
  We only sketch the proof.  The difficult part is, of course, that $D$ and
  $B$ might have different cellular structures.

  We have to find for $(\diamond_R D, \kappa^D)$ a set over $X$ which is
  controlled isomorphic to it and locally finite.  We will do that in two steps,
  first improve $\kappa_0 := \kappa^D$ to $\kappa_1$, and then to $\kappa_2$.
  All maps $\kappa_1, \kappa_2\colon \diamond_R D \rightarrow X$
  are controlled isomorphic to $\kappa_0$
  and improve $\kappa_0$, in particular, $\kappa_2$ is a locally finite set
  over $X$.  Then we can set $\overline{\kappa} := \kappa_2$.  We prove the
  non-equivariant case first, i.e.,~assume $G= \{e\}$.

  First we define $\kappa_1$ such that its image is contained in the image of
  $\kappa^B$.  As $(B,\kappa^B)$ and $(D,\kappa_0)$ are controlled isomorphic,
  there is an $E\in \mathcal{E}$ such that for each $e\in \diamond_R D$ there
  is an $x(e) \in \Imag \kappa^B\subseteq X$ such that $(\kappa_0(e), x(e))\in
  E$.  Set $\kappa_1(e) := x(e)$.

  As $\kappa^B \colon \diamond_R B \rightarrow X$ is a locally
  finite set over $X$, its image has no accumulation points, i.e.,~for each
  point $x\in X$ there is a neighborhood $U_x$ containing only finitely many
  points of the image of $\kappa^B$.  Thus the same is true for the image of
  $\kappa_1$.

  Set 
  \begin{equation*}
    T := \{ x\in X \mid \kappa_1^{-1}(x) \text{ is infinite}\}.
  \end{equation*}
  As $X$ is Hausdorff we have that $(D,\kappa_1)$ is bl-finite if and only if $T$
  is empty.  So $T$ are the ``trouble points''.  We change $\kappa_1$ on
  $\kappa_1^{-1}(T)$.  We can proceed degreewise.  The rough idea is that if
  $\kappa_1^{-1}(T)$ is infinite in that degree, it needs to come from an
  infinite submodule of $B$.  

  Let $\theta\colon B \rightarrow D$ be the controlled isomorphism.  We
  change $\kappa_1$ on, say, $d\in \diamond_R D$ to map to $\kappa_B(b)$ for
  some $b\in \diamond_R B$ with $d\in \left<\theta(b)\right>_D$.  Careful
  checking shows that the changed map is locally finite, at least in that
  degree.  Also, the new map is controlled isomorphic to the old one.

  For $G$ being non-trivial, we can make all choices $G$-equivariant and are
  done.
\end{proof}

\begin{rem}\label{rem:proof-of-finite-and-proper}
  The proof of the Lemma implies the following for the control space $X$ if
  $T$ was not empty:  there are points $x\in X$ and $E\in \mathcal{E}$ such
  that $\{x\}^E$ is not contained in a compact subset.  Namely $i^x_n$ must
  hit infinitely many cells of $B$ over points in $\{x\}^E$, but $B$ is
  locally finite.  In particular, $X$ is not a proper control space.
\end{rem}

If $X$ and $Y$ are $G$-equivariant control spaces, then any $G$-equivariant
map $f\colon X \rightarrow Y$ induces a functor $\mathcal{C}^G(X) \rightarrow
\mathcal{C}^G(Y)$.  For the bl-finite objects we have the following obvious
criterion.

\begin{lem}\label{lem:induced-maps-on-finite-objects}
  Let $\varphi\colon (X, \mathcal{E}_X, \mathcal{F}_X) \rightarrow (Y,
  \mathcal{E}_Y, \mathcal{E}_Y)$ be a map of control spaces which maps locally
  finite sets over $X$ to locally finite sets over $Y$.  Then $\varphi$
  induces a functor $\mathcal{C}^G_f(X, \mathcal{E}_X, \mathcal{F}_X;R)
  \rightarrow \mathcal{C}^G_f(Y, \mathcal{E}_Y,
  \mathcal{F}_Y;R)$. \hfill$\qed$
\end{lem}

\begin{rem}\label{rem:closed-inclusions-give-functor}
  Note that inclusions of subspaces do not map locally finite sets to locally
  finite sets in general. A counterexample is the inclusion $\Real
  \smallsetminus \{0\} \rightarrow \Real$.  However closed inclusions
  do map locally finite sets to locally finite sets and hence do induce
  functors of categories of controlled modules.
\end{rem}

\subsection{Homotopy bl-finite objects}
We show that the full subcategory of homotopy bl-finite objects of
$\mathcal{C}^G$ is a Waldhausen category.  It follows directly from the
cylinder axiom that the cylinder functor of a homotopy bl-finite object is
again homotopy bl-finite.  Hence, we are left to
show~\eqref{enum:condition:c1}.

\begin{lem}
  Assume that $C \leftarrow A \rightarrowtail B$ is a diagram of homotopy
  bl-finite objects and $A \rightarrowtail B$ a cofibration. 
  Then $C\cup_A B$ is homotopy bl-finite.
\end{lem}

\begin{proof}
  Assume that there are bl-finite objects $A', B', C'$ weakly equivalent to
  $A,B,C$.  Note that we have inverses for weak equivalences, which we will
  use freely.  Below we denote mapping cylinders by $M_A$, $M_B$, etc.~and
  cofibrations by $\rightarrowtail$.
  
  We obtain a chain of maps of diagrams. In the following, the arrows marked with
  $\xrightarrow{\bullet}$ are defined by composition. The first step is
  \begin{equation*}
    \xymatrix{ 
    C & A    \ar@{ >->}[r]       \ar[l]   & B \\
    C       \ar@{=}[u]  & 
    A'      \ar[u]^\sim \ar[l]_{\bullet} \ar@{ >->}[r] & 
    M_B     \ar[u]^\sim
    },
  \end{equation*}
  where $A'$ is bl-finite and $M_B$ is the mapping cylinder of $A' \rightarrow A
  \rightarrowtail B$, which still is homotopy bl-finite. Next we obtain a map
  \begin{equation*}
    \xymatrix{ 
    C       \ar[d]_\sim  & 
    A'      \ar@{=}[d] \ar[l] \ar@{ >->}[r] & 
    M_B     \ar@{=}[d] \\
    C' & A'    \ar@{ >->}[r]       \ar[l]_{\bullet}   & M_B 
    }
  \end{equation*}
  by $C$ being homotopy bl-finite. Then consider
  \begin{equation*}
    \xymatrix{ 
    C' & A'    \ar@{ >->}[r]       \ar[l]   & M_B \\
    M_{C'}     \ar[u]^\sim  & 
    A'      \ar@{=}[u] \ar@{ >->}[l] \ar@{ >->}[r] & 
    M_B     \ar@{=}[u] 
    }
  \end{equation*}
  with $M_{C'}$ being the cylinder of $A' \rightarrow C'$ which is bl-finite as
  $A'$ and $C'$ are bl-finite. Finally, we obtain a map
  \begin{equation*}
    \xymatrix{ 
    M_{C'}  \ar@{=}[d]  & 
    A'      \ar@{=}[d] \ar@{ >->}[l] \ar@{ >->}[r] & 
    M_B     \ar[d]_\sim \\
    M_{C'} & A'    \ar[r]^{\bullet}  \ar@{ >->}[l]   & B'
    }
  \end{equation*}
  as $M_B$ is weakly equivalent to $B'$. Using the gluing lemma four times
  gives that $C \cup_A B$ is weakly equivalent to the bl-finite object
  $M_{C'}\cup_{A'} B'$.
\end{proof}

\subsection{Homotopy bl-finitely dominated objects}
We give a three different characterizations of homotopy bl-finitely dominated
objects.

\begin{definition}
  \label{def:retract-homotopy-retract}
  Let $M, M'$ be objects in $\mathcal{C}^G$.
  \begin{enumerate}
    \item $M$ is called a \emph{retract} of $M'$ if there are maps $i\colon M
      \rightarrow M'$, $r\colon M' \rightarrow M$ such that $r\circ i =
      \id_M$.
    \item $M$ is called a \emph{homotopy retract} of $M'$, or \emph{dominated
      by $M'$}, if there are maps $i\colon M \rightarrow M'$,
      $r\colon M' \rightarrow M$ and a homotopy $H \colon M[\Delta^1]
      \rightarrow M$ from $r\circ i$ to $\id_M$.
  \end{enumerate}
\end{definition}

\begin{lem} \label{lem:homotopy-retracts-finite}
  Let $A\in \mathcal{C}^G$. Then the following are equivalent.
  \begin{enumerate}
    \item \label{enum:lem:retracts-hr-f}
      $A$ is a homotopy retract of a bl-finite module $A'$.
    \item \label{enum:lem:retracts-r-hf}
      $A$ is a retract of a homotopy bl-finite
      module $A''$.
    \item \label{enum:lem:retracts-hr-hf}
      $A$ is a homotopy retract of a homotopy bl-finite module $A'''$.
  \end{enumerate}
\end{lem}
\begin{proof}
  Clearly \eqref{enum:lem:retracts-hr-f} $\Rightarrow$
  \eqref{enum:lem:retracts-hr-hf} and \eqref{enum:lem:retracts-r-hf}
  $\Rightarrow$ \eqref{enum:lem:retracts-hr-hf} hold. We show
  \eqref{enum:lem:retracts-hr-hf} $\Rightarrow$
  \eqref{enum:lem:retracts-hr-f} first.

  As $A'''$ is homotopy bl-finite, there is a bl-finite module $B$ and maps $f\colon
  A'''\rightarrow B$, $g\colon B \rightarrow A'''$ such that $g\circ f \simeq
  \id_{A'''}$, so $A$ is a homotopy retract of $B$ via $A \xrightarrow{i} A'''
  \xrightarrow{f} B$ and $B\xrightarrow{g} A''' \xrightarrow{r} B$.

  Now we show \eqref{enum:lem:retracts-hr-f} $\Rightarrow$
  \eqref{enum:lem:retracts-r-hf}. We have maps $i \colon A \rightarrow A'$,
  $r\colon A'\rightarrow A$ with $r\circ i \simeq \id_A$.  We can make the
  homotopy commutative diagram
  \begin{equation*}
    \xymatrix{A \ar[r]^i \ar[dr]_\id  & A' \ar[d]^r \\& A}
  \end{equation*}
  into a strict commutative one, namely
  \begin{equation*}
    \xymatrix{ A \ar[r] \ar[dr]_\id & T(i) \ar[d] \\& A} .
  \end{equation*}
  Hence, $A$ is a retract of $T(i)$ and as $T(i) \xrightarrow{\sim} A'$ is a
  homotopy equivalence, $T(i)$ is homotopy bl-finite.
\end{proof}

\begin{lem}\label{lem:homotopy-finitely-dominated-is-w}
  $\mathcal{C}^G_{hfd}$ is a Waldhausen category.
  It has a cylinder functor satisfying the cylinder axiom and the class of
  weak equivalences satisfies the extension and the saturation axiom.
\end{lem}

\begin{proof}
  Again we only show \hyperref[enum:condition:c1]{(C1)} and
  \hyperref[enum:condition:c2]{(C2)} from before. 
  For (C2), $A[\Delta^1]$ is dominated by $A'[\Delta^1]$.
  
  For (C1), assume that $A,B,C$ are retracts of homotopy bl-finite objects
  $A', B', C'$. Note that
  we can make the co-retraction into a cofibration by replacing $A'$ with the
  mapping cylinder of $A \rightarrow A'$, so we will assume that the
  co-retractions $i_A, i_B, i_C$ are actually cofibrations.

  We want to show that $C \cup_A B$ is a retract of a homotopy bl-finite object.
  We reduce this to the case where $A$ is homotopy bl-finite. Consider the
  commutative diagram
  \begin{equation*}
    \xymatrix{ A \ar@{ >->}[r]   \ar@{ >->}[d]  & B \ar@{=}[d] \\
               A'\ar[r]^\bullet   \ar[d]         & B \ar@{=}[d] \\
               A \ar@{ >->}[r]                  & B }.
  \end{equation*}
  (As before $\xrightarrow{\bullet}$ denotes a map defined by composition.) We
  can factor the horizontal maps into cofibrations simultaneously using the
  cylinder functor. We obtain
  \begin{equation*}
    \xymatrix{ 
      A \ar@{ >->}[r]   \ar@{ >->}[d] & M \ar[d] \ar[r]^\sim  & B \ar@{=}[d] \\
      A'\ar@{ >->}[r] \ar[d] &\overline{M}\ar[d] \ar[r]^\sim  & B \ar@{=}[d] \\
      A \ar@{ >->}[r]                 & M        \ar[r]^\sim  & B }.
  \end{equation*}
  Here and in all following diagrams the composition of the vertical arrows is
  always the identity, which holds in the diagram above by the functoriality
  of the cylinder functor. By the gluing lemma, $C \cup_A M$ is weakly
  equivalent to $C \cup_A B$. Then the diagram, extended by $C$, is
  \begin{equation*}
    \xymatrix{ 
      C \ar@{=}[d] & A \ar[l] \ar@{ >->}[r]   \ar@{ >->}[d] & M \ar[d] \\
      C \ar@{=}[d] & A'\ar[l]_\bullet \ar@{ >->}[r] \ar[d] &\overline{M}\ar[d]  \\
      C            & A \ar[l] \ar@{ >->}[r]                 & M },
  \end{equation*}
  which shows that $C\cup_A M$ is a retract of $C \cup_{A'} \overline{M}$. We are
  done if we show that $C \cup_{A'} \overline{M}$ is bl-finitely dominated.
  As $M$ is homotopy equivalent to the homotopy bl-finitely dominated object $B$,
  Lemma \ref{lem:homotopy-retracts-finite} shows that $\overline{M}$ is again
  a retract of a homotopy bl-finite module $M'$.

  Now we can use that we have co-retractions $C\rightarrowtail C'$,
  $\overline{M} \rightarrowtail M'$ with $C', M'$ homotopy bl-finite objects,
  which are also cofibrations. This gives a commuting retraction diagram
  \begin{equation*}
    \xymatrix{ 
      C  \ar[d] & A'\ar[l] \ar@{ >->}[r]   \ar@{=}[d] & \overline{M} \ar@{ >->}[d]  \\
      C' \ar[d] & A'\ar[l]_\bullet \ar@{ >->}[r]^\bullet \ar@{=}[d] &M'\ar[d]   \\
      C         & A'\ar[l] \ar@{ >->}[r]                 & \overline{M}   },
  \end{equation*}
  where we want to emphasize, that the map $A' \rightarrowtail M'$, defined by
  composition, is a cofibration. Thus $C \cup_{A'} \overline{M}$ is a retract
  of $C' \cup_{A'} M'$, which is homotopy bl-finite, as being a pushout of
  homotopy bl-finite objects along a cofibration.

\end{proof}

\section{Proofs IV: connective algebraic K-theory of categories of controlled
simplicial modules}
\label{sec:proofs-iv:connective-alg-k-theory}
\subsection{Algebraic K-theory}
In the last section we showed that the categories $\mathcal{C}^G_{f}$,
$\mathcal{C}^G_{hf}$ and $\mathcal{C}^G_{hfd}$ are all categories with
cofibrations and weak equivalences, so we can use Waldhausen's
$\mathcal{S}_{\bdot}$-construction from \cite{Waldhausen;spaces;;;;1985} to
produce an algebraic K-Theory spectrum $K(\mathcal{C}^G_{?})$ and therefore
also the corresponding infinite loop space. Define
$K_n(\mathcal{C}^G_{?})$ for $n\geq 0$ as the $n$th homotopy group $\pi_n
K(\mathcal{C}^G_{?})$. This algebraic K-Theory spectrum is always connective
so we do not assign any name to its negative homotopy groups.

\begin{rem}
  There is a slight set-theoretical problem, as $\mathcal{C}^G$ is not a
  small category according to our definition but it needs be one to apply the
  K-theory construction. However, we follow the usual
  approach (see e.g.~\cite[Remark before~2.1.1]{Waldhausen;spaces;;;;1985})
  and fix a suitably large set-theoretical small category of simplicial
  $R$-modules to begin with. Then all the categories we consider are again
  small. (We could obtain such a category by fixing a large cardinal and
  requiring it to contain all elements.)
  We will assume such a choice from now on.
\end{rem}

\subsection{A cofinality theorem}
\label{subsec:cofinality-theorem}

We want to show that $K_n(\mathcal{C}^G_f)$
$K_n(\mathcal{C}^G_{hf})$ and
$K_n(\mathcal{C}^G_{hfd})$ agree for $n \geq 1$.  For this we need a
cofinality theorem first.

\begin{thm}[Waldhausen-Thomason cofinality]\label{thm:thomason-cofinality}
  Let $\mathcal{A}$ and $\mathcal{B}$ be Waldhausen categories.  Suppose that
  $\mathcal{A}$ is a full subcategory of $\mathcal{B}$ which satisfies the
  following conditions
  \begin{enumerate}
    \item \label{item:cofinality:cofib}
      $\mathcal{A}$ is a Waldhausen subcategory: $f\colon X\rightarrow B$
      in $\mathcal{A}$ is a cofibration if and only if it is a cofibration in
      $\mathcal{B}$ with cokernel in $\mathcal{A}$.
    \item \label{item:cofinality:weq}
      A map in $\mathcal{A}$ is a weak equivalence if and only if it is
      one in $\mathcal{B}$.
    \item \label{item:cofinality:saturated}
      $\mathcal{A}$ is saturated: every object in $\mathcal{B}$ which is
      weakly equivalent to an object in $\mathcal{A}$ is itself in
      $\mathcal{A}$.
    \item \label{item:cofinality:extensions}
      $\mathcal{A}$ is closed under extensions: if $X\rightarrowtail Y
      \twoheadrightarrow Z$ is a cofiber sequence in $\mathcal{B}$ and $X,Z$
      are in $\mathcal{A}$, then $Y$ is in $\mathcal{A}$.
    \item \label{item:cofinality:cylinders} 
      $\mathcal{B}$ has mapping cylinders satisfying the cylinder axiom
      and $\mathcal{A}$ is closed under them.
    \item \label{item:cofinality:cofinal} 
      $\mathcal{A}$ is cofinal in $\mathcal{B}$:  for every object $X$ in
      $\mathcal{B}$ there is an object $\overline{X}$ in $\mathcal{B}$ such
      that $X \vee \overline{X}$ is in $\mathcal{A}$.
  \end{enumerate}
  Then 
  \begin{equation*}
    K(\mathcal{A}) \rightarrow K(\mathcal{B}) \rightarrow
    \text{``}K_0(\mathcal{B}) / K_0(\mathcal{A})\text{''} 
  \end{equation*}
  is a homotopy fiber sequence of connective spectra.  Here
  $\text{``}K_0(\mathcal{B}) / K_0(\mathcal{A})\text{''}$ denotes
  the Eilenberg-MacLane spectrum with the group
  $K_0(\mathcal{B}) / K_0(\mathcal{A})$ in degree~$0$.
\end{thm}

\begin{rem}[Similar results]
  Theorem~\ref{thm:thomason-cofinality}
  is inspired from Thomason-Trobaugh \cite[Exercise
  1.10.2]{Thomason-Trobaugh;higher-algebraic} and Vogell
  \cite[Theorem~1.6]{Vogell;bounded}.  However, we have
  slightly different assumptions.  In particular, in
  \cite{Thomason-Trobaugh;higher-algebraic} the saturatedness assumption is
  missing.  We will provide a counterexample in 
  Section~\ref{subsubsec:saturated-counterexample}.  Parts of
  the proof were also inspired by Weibel's
  K-Book~\cite[Corollary V.2.3.1 and prerequisites]{weibel;k-book}, which,
  unfortunately, suffers from the same missing assumption, see again
  Section~\ref{subsubsec:saturated-counterexample}.  Staffeldt,
  \cite[Thm.~2.1]{Staffeldt;on-fundamental}, has a similar result as ours in
  the context of exact categories, which is strictly less general as his class
  of weak equivalences are only the isomorphisms.
\end{rem}

\begin{rem}
  Instead of proving the theorem directly, we prove a lemma about $K_0$ and
  then rely on Thomason-Trobaugh's cofinality theorems~\cite[Thm.~10.1]{Thomason-Trobaugh;higher-algebraic} and
  Waldhausen's strict cofinality
  theorem~\cite[1.5.9]{Waldhausen;spaces;;;;1985}.  (The latter has an
  implicit assumption that the subcategory is full, see
  \ref{subsubsec:fullness-necessary}.)
\end{rem}

\begin{lem}\label{lem:cofinality-K-0}
  Assume we are in the same situation as in
  Theorem~\ref{thm:thomason-cofinality}.  Assume furthermore that $\mathcal{A}
  \rightarrow \mathcal{B}$ induces a surjection $K_0(\mathcal{A})
  \rightarrow K_0(\mathcal{B})$.  Then $\mathcal{A}$ is \emph{strictly
  cofinal} in $\mathcal{B}$ in the sense
  of~\cite[1.5.9]{Waldhausen;spaces;;;;1985}, i.e., for each $B \in
  \mathcal{B}$ there is a $A\in\mathcal{A}$ such that $B\vee A \in
  \mathcal{A}$.  (In this situation we do not need the cylinder functor
  assumption, but all the other ones are used.)
\end{lem}

Under the hypothesis of the lemma Waldhausen's strict cofinality theorem applies
and shows that $K(\mathcal{A}) \simeq K(\mathcal{B})$.

\begin{proof}
  Recall (e.g.~from \cite[1.5.6]{Thomason-Trobaugh;higher-algebraic}) that
  $K_0(\mathcal{A})$ is the abelian group generated by isomorphism classes of
  objects $[A]$, $A\in \mathcal{A}$ with relations
  \begin{enumerate}
    \item $[A] = [B]$ if there is a weak equivalence $A\xrightarrow{\sim} B$
    \item $[A] + [C] = [B]$ if there is a cofiber sequence $A \rightarrowtail
      B\twoheadrightarrow C$.
  \end{enumerate}
  Let $K'_0(\mathcal{A})$ be the group where we ignore the weak equivalences
  and consider only split cofiber sequences.  That is, it is the group
  generated by isomorphism classes of objects $[A]$, $A\in\mathcal{A}$ with
  relation
  \begin{enumerate}
    \item $[A] + [C] = [A\vee C]$ for $A,C \in \mathcal{A}$.
  \end{enumerate}
  We do the same for $\mathcal{B}$. 
  From the inclusion $i\colon \mathcal{A} \rightarrow \mathcal{B}$ we obtain a
  (solid) commutative diagram
  \begin{equation*}
    \xymatrix{
    K'_0(\mathcal{A}) \ar[d]\ar[r]  &
    K'_0(\mathcal{B}) \ar[d]\ar@{-->}[r]  & G' \ar@{-->}[d]\\
    K_0(\mathcal{A}) \ar[r]  &
    K_0(\mathcal{B}) \ar@{-->}[r]  & G 
    }
  \end{equation*}
  which we can extend to the cokernels $G$ and $G'$ as shown.  We claim $G'
  \rightarrow G$ is an isomorphism.

  First, $G$ is the abelian group with generators $[B]$ for $B\in \mathcal{B}$
  and relations
  \begin{enumerate}
    \item \label{enum:cofib} $[A] + [C] = [B]$ if there is a cofiber sequence
      $A \rightarrowtail B\twoheadrightarrow C$.
    \item \label{enum:from-A} $[A] = 0$ if $A \in \mathcal{A}$.
    \item \label{enum:weq} $[A] = [B]$ if there is a weak equivalence
      $A\xrightarrow{\sim} B$
  \end{enumerate}
  We claim that (\ref{enum:weq}) is redundant:  by cofinality, there is an
  $\overline{A} \in \mathcal{B}$ such that $A \vee \overline{A}
  \in\mathcal{A}$.  By the gluing lemma, $A \vee\overline{A} \xrightarrow{\sim}
  B \vee \overline{A}$ is still a weak equivalence.  By saturation, $B
  \vee \overline{A}$ is also in $\mathcal{A}$.  Therefore, by
  \eqref{enum:from-A}, $[A \vee \overline{A}] = 0$ and $[B \vee \overline{A}]
  = 0$ in $G$.  By \eqref{enum:cofib}, then $[A] = -[\overline{A}] = [B]$ in
  $G$, which implies (\ref{enum:weq}).

  Furthermore, $G'$ is the abelian group with generators $[B]$ for
  $B\in\mathcal{B}$ and relations
  \begin{enumerate}
    \item $[A] + [B] = [A \vee B]$.
    \item $[A] = 0$ if $A \in \mathcal{A}$.
  \end{enumerate}

  We show that the stronger relation \eqref{enum:cofib} for $G$ comes from
  $G'$ and therefore $G$ and $G'$ are isomorphic.  Let $A \rightarrowtail B
  \twoheadrightarrow C$ be a cofiber sequence in $\mathcal{B}$.  By
  cofinality, there are $\overline{A}, \overline{C}$ in $\mathcal{B}$ such
  that $\overline{A} \vee A$ and $\overline{C} \vee C$ are in $\mathcal{A}$.
  Then
  \begin{equation}\label{eq:cofib-seq}
    A \vee \overline{A} \rightarrow B \vee \overline{A} \vee \overline{C}
    \rightarrow C \vee\overline{C}
  \end{equation}
  is a cofiber sequence in $\mathcal{B}$ by the pushout axiom, and the first
  and last object are in $\mathcal{A}$.  Since $\mathcal{A}$ is closed under
  extensions in $\mathcal{B}$,
  the middle term $B \vee \overline{A} \vee \overline{C}$ is also in
  $\mathcal{A}$.  It follows that in $G'$ we have
  \begin{equation*}
    \begin{split}    
      [B] + [\overline{A}] + [\overline{C}] = 0 \\
        [A] + [\overline{A}]  = 0 \\
          [C] + [\overline{C}]  = 0
    \end{split}
  \end{equation*}
  and therefore $[B] = [A] + [C]$ in $G'$.  It follows that $G\cong G'$.\medskip

  We now prove that $\mathcal{A} \subseteq \mathcal{B}$ is strictly cofinal.
  If $K_0(\mathcal{A}) \rightarrow K_0(\mathcal{B})$ is a surjection, $G$
  and therefore $G'$ is trivial.  Hence, $K'_0(\mathcal{A}) \rightarrow
  K'_0(\mathcal{B})$ is surjective.  Let $B\in \mathcal{B}$, then there is an
  $A \in \mathcal{A}$ such that $[A] = [B] \in K'_0(\mathcal{B})$.  Now
  $K'_0(\mathcal{B})$ is just a group completion with respect to $\vee$ after
  taking isomorphism classes.  Therefore, $[A] = [B]$ if and only if there is a
  $C \in \mathcal{B}$ with $A\vee C \cong B\vee C$.  (This is the usual
  algebraic argument.)  As $\mathcal{A}$ is cofinal, there is a
  $\overline{C} \in \mathcal{B}$ such that $C\vee \overline{C} \in
  \mathcal{A}$.  Therefore,
  \begin{equation*}
    A \vee ( C \vee \overline{C}) \cong B \vee C \vee \overline{C}.
  \end{equation*}
  with, of course, $B\vee C \vee \overline{C} \in \mathcal{A}$.  This shows
  strict cofinality.
\end{proof}

\begin{proof}[Proof of Thm.~\ref{thm:thomason-cofinality}]
  We want to apply \cite[1.10.1]{Thomason-Trobaugh;higher-algebraic}.  For the
  convenience of the reader we quote the result:
  \begin{quote}
    1.10.1. \textbf{Cofinality Theorem.}
    Let $v\mathcal{B}$ be a Waldhausen category with a cylinder functor
    satisfying the cylinder axiom. Let $G$ be an abelian group, and $\pi
    \colon K_0(v\mathcal{B}) \rightarrow G$ an epimorphism.  Let
    $\mathcal{B}^w$ be the full subcategory of those $B$ in $\mathcal{B}$ for
    which the class $[B]$ in $K_0(v\mathcal{B})$ has $\pi[B] = 0$ in $G$.
    Make $\mathcal{B}^w$ a Waldhausen category with $v(\mathcal{B}^w) =
    \mathcal{B}^w \cap v(B)$, $\mathrm{co}(\mathcal{B}^w) = \mathcal{B}^w \cap
    \mathrm{co}(B)$.  Let ``$G$'' denote $G$ considered as Eilenberg-MacLane
    spectrum whose only non-zero homotopy group is $G$ in dimension $0$.

    Then there is a homotopy fibre sequence
    \begin{equation*}
      K(v\mathcal{B}^w) \rightarrow K(v\mathcal{B}) \rightarrow
      \text{``}G\text{''}
    \end{equation*}
  \end{quote}
  Define $G$ as $K_0(\mathcal{B})/
  K_0(\mathcal{A})$.  We obtain the above fiber sequence and, in particular,
  that $K_0(v\mathcal{B}^w) = \ker \pi$, where $\pi$ is the projection
  $K_0(v\mathcal{B}) \rightarrow G$. As $\pi[A]=0$ for $A\in \mathcal{A}$ by
  the definition of $G$, the inclusion $\mathcal{A} \rightarrow \mathcal{B}$
  factors as $\mathcal{A} \rightarrow \mathcal{B}^w$.   It suffices to show
  that
  $K(\mathcal{A}) \rightarrow K(\mathcal{B}^w)$ is a weak equivalence.
  Like in $\mathcal{B}$, $\mathcal{A}$ is cofinal in
  $\mathcal{B}^w$, as one can choose the same complement.  Then
  $\mathcal{A}\subseteq \mathcal{B}^w$ satisfies the assumptions of
  Theorem~\ref{thm:thomason-cofinality}.  But $K_0(\mathcal{A}) \rightarrow
  K_0(\mathcal{B}^w)$ is surjective, so by Lemma~\ref{lem:cofinality-K-0},
  $\mathcal{A}$ is strongly cofinal in $\mathcal{B}^w$.  By Waldhausen's
  strict cofinality Theorem~\cite[1.5.9]{Waldhausen;spaces;;;;1985}, there is
  a homotopy equivalence $K(\mathcal{A}) \rightarrow K(\mathcal{B}^w)$.  This
  shows that we obtain the desired homotopy fiber sequence
  \begin{equation*}
    K(\mathcal{A}) \rightarrow K(\mathcal{B}) \rightarrow
    \text{``}K_0(\mathcal{B})/ K_0(\mathcal{A})\text{''}.\qedhere
  \end{equation*}
\end{proof}

\subsubsection{Saturated is necessary: a counterexample}
\label{subsubsec:saturated-counterexample}
The assumption \eqref{item:cofinality:saturated} in
Theorem~\ref{thm:thomason-cofinality} saturatedness is
necessary.  Here is a counterexample, which the author developed during a
discussion with Chuck Weibel.  Let $\mathcal{C}$ be the following Waldhausen
category:
\begin{enumerate}
  \item Objects are finite pointed sets $X$ with decomposition $X =A \vee B
    \vee C$.  (Think of the elements as being colored, while the basepoint is
    black.)
  \item Morphisms are maps $A \vee B \vee C \rightarrow A' \vee B' \vee C'$
    such that they restrict to maps $A \rightarrow A'$,  $B \rightarrow A'
    \vee B'$, $C \rightarrow A' \vee C'$.  That is, you can change any color
    to $A$ or map to the basepoint, or do not change the color.
  \item Let Cofibrations be split injections. That is maps $i\colon X
    \rightarrow X'$ such that there is a map $p \colon
    X' \rightarrow X$ with $p\circ i = \id$.  It follows that $X' \cong X \vee
    Y$ for some $Y$ in $\mathcal{C}$.  This is, they come from the direct sum
    ``$\vee$'' in pointed sets.
  \item Let a weak equivalence be a bijection of pointed sets.
\end{enumerate}
This category has $K_0(\mathcal{C}) \cong \mathbb{Z}$, as each object is
weakly equivalent to an object $A \vee * \vee *$.

Consider the full subcategory $\mathcal{B}$ of objects $A \vee B
\vee C$ with $|A| = |C|$.  This is a cofinal subcategory in $\mathcal{C}$.  It
is not saturated:  while $A \vee * \vee C$ is equivalent to $(A\vee C) \vee *
\vee *$ in $\mathcal{C}$, the latter is not in $\mathcal{B}$.  It satisfies
all the other assumptions of Theorem~\ref{thm:thomason-cofinality}  except for
the cylinder
functor~\ref{thm:thomason-cofinality}.\eqref{item:cofinality:cylinders},
so Lemma~\ref{lem:cofinality-K-0} would apply and show that $K_0(\mathcal{B})
\subseteq K_0(\mathcal{C})$.
However, one sees that $K_0(\mathcal{B}) = \mathbb{Z} \oplus \mathbb{Z}$, with
generators represented by $A \vee * \vee C$ and $* \vee B \vee *$, as all weak
equivalences in $\mathcal{B}$ are isomorphisms.  Hence, $K_0(\mathcal{B})
\rightarrow K_0(\mathcal{C})$ cannot be an injection.

This counterexample shows that the saturatedness assumption is necessary and
missing in Exercise 1.10.2 in \cite{Thomason-Trobaugh;higher-algebraic}, as
well as in Corollary V.2.3.1 in Weibel's K-book \cite{weibel;k-book}.   The latter is deduced
in \cite{weibel;k-book} from Theorem II.9.4 through a chain along exercise
II.9.14 (``Grayson's trick''), Theorem IV.8.9 and Remark IV.8.9.1, (which
states $K_0(\mathcal{B}) = K_0(\mathcal{C})$ is equivalent to $\mathcal{B}$
being strictly cofinal in $\mathcal{C}$).  The gap is in the proof of Theorem
II.9.4, which claims that the proof of Lemma II.7.2 applies verbatim.  In
particular, the
above counterexample applies to Theorem II.9.4.

\subsubsection{Fullness is necessary}\label{subsubsec:fullness-necessary}
For completeness, let us remark that the fullness assumption in
\ref{thm:thomason-cofinality} is also necessary:  in
Waldhausen's cofinality theorem~\cite[1.5.9]{Waldhausen;spaces;;;;1985}, he
does not mention that one needs to assume that the subcategory $\mathcal{A}
\subseteq \mathcal{B}$ is full.  There is a counterexample due to Inna
Zakharevich~\cite{inna-cofinality-mathoverflow}
which shows that one has to assume it:
\begin{quote}
``Consider the following example. Let $C$ be the category of pairs of pointed
finite sets, whose morphisms $(A,B)\rightarrow(A',B')$ are pointed maps
$A\vee B\rightarrow A' \vee B'$, and let
$B$ be the category of pairs of pointed finite sets whose morphisms
$(A,B)\rightarrow (A',B')$ are pairs of pointed maps $A\rightarrow B$ and
$A'\rightarrow B$. We make $C$ a Waldhausen
category by defining the weak equivalences to be the isomorphisms, and the
cofibrations to be the injective maps. $B$ is clearly cofinal in $C$, but
$K_0(B)= \mathbb{Z}\times \mathbb{Z}$, while $K_0(C)=
\mathbb{Z}$.[\ldots]''  (from
\url{http://mathoverflow.net/q/23515})
\end{quote}

\subsection{Change of finiteness conditions}
\label{subsec:change-of-finiteness-conditions}
We turn to the proofs for Section~\ref{sec:algebraic-k-theory-control}.
We need Waldhausen's Approximation Theorem, which we recall for convenience.

\begin{definition}[Approximation Property~{\cite[1.6]{Waldhausen;spaces;;;;1985},%
  \cite[1.9.1]{Thomason-Trobaugh;higher-algebraic}}]
\label{def:approximation-property} \label{def:appendix:approximation-property}
 Let $F\colon \mathcal{A} \rightarrow \mathcal{B}$ be an exact functor of
 Waldhausen categories. $F$ has the
  \emph{Approximation Property} if the following two axioms hold.
  \begin{enumerate}
    \item[(App~1)] A map $f$ in $\mathcal{A}$ is a weak equivalence if (and only
      if) its image $F(f)$ in $\mathcal{B}$ is a weak equivalence.
    \item[(App~2)] Given any object $A$ in $\mathcal{A}$ and a map $x\colon
      F(A) \rightarrow B$ in $\mathcal{B}$ there exists a map $a\colon A
      \rightarrow A'$ in $\mathcal{A}$ and a weak equivalence $x'\colon F(A')
      \rightarrow B$ in $\mathcal{B}$ such that the triangle
      \begin{equation*}
	\xymatrix{ F(A) \ar[r]^{x} \ar[d]_{F(a)} & B \\
	F(A') \ar[ru]_{x'} }
      \end{equation*}
      commutes.
  \end{enumerate}
\end{definition}

\begin{thm}[Approximation Theorem~{\cite[1.6.7]{Waldhausen;spaces;;;;1985},%
  \cite[1.9.1]{Thomason-Trobaugh;higher-algebraic}}]
\label{thm:approximation-theorem} 
  Let $\mathcal{A}$, $\mathcal{B}$ be Waldhausen categories which satisfy the saturation axiom.  Assume $\mathcal{A}$ has a
  cylinder functor satisfying the cylinder axiom.  Let $F\colon \mathcal{A}
  \rightarrow \mathcal{B}$ be an exact functor with the
  Approximation Property.  Then $F$ induces an equivalence 
  \begin{equation*}
    K(F) \colon K(w\mathcal{A}) \rightarrow K(w\mathcal{B})
  \end{equation*}
  on connective algebraic K-theory spectra.  
\end{thm}

We used Thomason-Trobaugh's remark
in~\cite[1.9.1]{Thomason-Trobaugh;higher-algebraic} that we can use a weaker
version of the Approximation Property.  In~\cite{Waldhausen;spaces;;;;1985},
there is the further requirement in (App~2) that $a$ is a cofibration, which
we can always arrange due to the existence of a cylinder functor.

\begin{proof}[Proof of Theorem~\ref{thm:k-theory-comparison-f-hf-hfd}]
  To prove (\ref{item:prop-k-theory:1})~we use Waldhausen's Approximation
  Theorem~\ref{thm:approximation-theorem} and apply it to the inclusion
  functor.  We check the conditions.  First, all our categories satisfy the
  saturation axiom~\ref{lem:saturation-axiom} and have a cylinder
  functor~\ref{thm:cylinder-functor}
  satisfying the cylinder axiom~\ref{lem:cylinder-axiom}.
  
  A map is a homotopy equivalence in $\mathcal{C}^G_{f}$ if and only if it is
  one in $\mathcal{C}^G_{hf}$, hence (App~1) is satisfied.

  Assume we have $A \in \mathcal{C}^G_{f}$, $B \in \mathcal{C}^G_{hf}$ and a map
  $f\colon A \rightarrow B$. For $B$ there is, by definition, a bl-finite object $B_f
  \in \mathcal{C}^G_f$ which is homotopy equivalent to $B$, i.e.,~there are maps
  $g\colon B_f \rightarrow B$ and $\overline{g}\colon B \rightarrow B_f$ with
  both compositions being homotopic to the identity. Define $j\colon A
  \rightarrow B_f$ as $j:= \overline{g} \circ f$. Then $g\circ j$ is homotopic
  to $f$.  Using the cylinder functor (and
  Section~\ref{subsubsec:rectifying-diagrams}) we can rectify the homotopy
  commutative diagram on the left below to the strict commutative diagram on
  the right:
  \begin{equation*}
    \vcenter{\xymatrix{
    A \ar[r]^f \ar[d]_{j}    & B \\
    B_f \ar[ur]_g }} \qquad \rightsquigarrow\qquad
    \vcenter{\xymatrix{
    A \ar[r]^f \ar[d]_{\iota_0}    & B \\
    T(j) \ar[ur]_H }}
  \end{equation*}
  As $B_f \rightarrow T(j)$ is a homotopy equivalence, by the saturation axiom
  $H$ is a homotopy equivalence. This shows (App~2) and therefore
  (\ref{item:prop-k-theory:1}).\smallskip
  
  For (\ref{item:prop-k-theory:2})~we use the Waldhausen-Thomason cofinality
  Theorem~\ref{thm:thomason-cofinality}.  We have to check the conditions (1)
  to (6). Most of them are clear or shown in the previous sections. 
  
  In particular, a map $A \rightarrow A'$ in
  $\mathcal{C}^G_{hf}$ which is a cofibration in $\mathcal{C}^G_{hfd}$ is a
  cofibration in $\mathcal{C}^G_{hf}$, and therefore its quotient is again in
  $\mathcal{C}^G_{hf}$.  This is \eqref{item:cofinality:cofib}.  By
  definition,
  $\mathcal{C}^G_{hf}$ is full in $\mathcal{C}^{G}_{hfd}$.  A map in the former
  is a weak equivalence if and only if it is in the latter.  Similarly, the
  cylinder functor is just inherited.  This shows \eqref{item:cofinality:weq}
  and \eqref{item:cofinality:cylinders}.  An object homotopy
  equivalent to an object in $\mathcal{C}^{G}_{hf}$ is homotopy bl-finite, hence
  itself in $\mathcal{C}^G_{hf}$, this shows
  \eqref{item:cofinality:saturated}.
  
  We are left with first showing the cofinality
  \eqref{item:cofinality:cofinal} and then that
  $\mathcal{C}^G_{hf}$ is closed under extensions in
  $\mathcal{C}^{G}_{hfd}$ \eqref{item:cofinality:extensions}.

  For $B\in \mathcal{C}^{G}_{hfd}$ there is an $A \in \mathcal{C}^G_{hf}$ such
  that $B$ is a retract of $A$, i.e.,~there are maps $r\colon A \rightarrow B$,
  $i\colon B \rightarrow A$ such that $r\circ i = \id_B$. By replacing $A$
  with $T(i)$, we can assume that $i$ is a cofibration, hence there is a
  cofiber sequence
  \begin{equation*}
    B \stackrel{i}{\rightarrowtail} A \stackrel{p}{\twoheadrightarrow} C :=
    A/B.
  \end{equation*}
  The retraction $r\colon A \rightarrow B$ and the map $* \rightarrow C$ give
  a map $A \rightarrow B \vee C$, and $* \rightarrow B$ and  $A \rightarrow C$
  give another one. The sum of these maps makes the diagram
  \begin{equation*}
    \xymatrix{
      B \ar@{ >->}[r] \ar[d]_= & A \ar@{->>}[r]        \ar[d]& C \ar[d]_= \\
      B \ar@{ >->}[r]          & B \vee C \ar@{->>}[r]       & C        }
    \end{equation*}
  commutative and both rows are cofiber sequences. By the extension
  axiom~\ref{lem:waldh-cat:extension-axiom}, the map $A \rightarrow B \vee C$
  is a homotopy equivalence, hence $B \vee C \in \mathcal{C}^G_{hf}$.  This
  show cofinality.

  Next we need to show that $\mathcal{C}^{G}_{hf}$ is closed under
  extensions in $\mathcal{C}^G_{hfd}$. Let
  \begin{equation*}
    A \rightarrowtail B \twoheadrightarrow C
  \end{equation*}
  be a cofiber sequence in $\mathcal{C}^G_{hfd}$ with $A,C \in
  \mathcal{C}^G_{hf}$ and $B$ (``the extension of $A$ by $C$'') in
  $\mathcal{C}^{G}_{hfd}$. As $\mathcal{C}^G_{hf}$ is cofinal there is a
  $B'\in \mathcal{C}^G_{hfd}$ such that $B \vee B' \in
  \mathcal{C}^G_{hf}$.  
  Then 
  \begin{equation*}
    A \rightarrowtail B \vee B' \twoheadrightarrow C \vee B'
  \end{equation*}
  is a cofiber sequence with $A, B \vee B' \in \mathcal{C}^G_{hf}$, hence
  the quotient $C \vee B'$ is in $\mathcal{C}^G_{hf}$ by the gluing lemma.
  Similarly, but more easily, we use the cofiber sequence
  \begin{equation*}
    C \rightarrowtail C \vee B' \twoheadrightarrow B'
  \end{equation*}
  to show $B' \in \mathcal{C}^G_{hf}$ and then the cofiber sequence
  \begin{equation*}
    B' \rightarrowtail B \vee B' \twoheadrightarrow B
  \end{equation*}
  to show $B \in \mathcal{C}^G_{hf}$.

  The Cofinality Theorem~\ref{thm:thomason-cofinality} therefore gives us a
  homotopy fiber sequence of connective spectra
  \begin{equation*}
    K(\mathcal{C}^G_{hf}) \rightarrow
    K(\mathcal{C}^G_{hfd}) \rightarrow
    K_0(\mathcal{C}^G_{hfd}) / K_0( \mathcal{C}^G_{hf})
  \end{equation*}
  As $\pi_n(K_0(\mathcal{C}^G_{hfd}) / K_0( \mathcal{C}^G_{hf})) = 0$ for $n \neq 0$
  part (\ref{item:prop-k-theory:2})~of the proposition follows.
\end{proof}

\begin{rem}
  In view of the proposition, one can consider $\mathcal{C}^G_{hfd}$ as kind of
  ``idempotent completion'' of $\mathcal{C}^G_{hf}$.  (Recall that for
  algebraic K-theory of rings the idempotent completion
  \cite[2.B, p.61]{freyd;abelian-categories} of the category of finitely
  generated free $R$-modules gives the category of finitely generated
  projective $R$-modules, which has the correct $K_0$, cf.~also
  e.g.~\cite{Cardenas-Pedersen;Karoubi}.)
  
  Corollary~\ref{cor:hpty-idem-split} from the appendix shows that idempotents
  and certain homotopy idempotents
  split in $\mathcal{C}^G_{hfd}$.  The author does not know if every homotopy
  idempotent splits in $\mathcal{C}^G_{hfd}$.  Hence, it is not clear that
  $K_0(\mathcal{C}^G_{hfd})$ is the ``correct'' group from this point of view.
  However, the difference and interplay between $K_0(\mathcal{C}^G_{hf})$ and
  $K_0(\mathcal{C}^G_{hfd})$ is crucial in later work to construct a
  nonconnective delooping of $K(\mathcal{C}^G_{?})$ for all $?= f,hf,
  hfd$---we will obtain equivalent nonconnective K-theory
  spectra for all three finiteness conditions.
\end{rem}

%
%

\subsection{Change of rings} 
\label{subsec:change-of-rings}
Let $f\colon R \rightarrow S$ be a map of simplicial rings.

If $M$ is a cellular $R$-module then $S \otimes_R M$ is a cellular $S$-module
and we obtain a natural bijection $\diamond_R M \cong \diamond_S (S\otimes_R M)$
which makes $S\otimes_R M$ into a controlled $S$-module. This construction
respects all finiteness conditions and cofibrations, so we obtain an exact
functor $S\otimes_R -\colon \mathcal{C}^G_{?}(X; R) \rightarrow
\mathcal{C}^G_{?}(X; S)$.  Next we prove:

\begin{repthm}{thm:change-of-rings}[Change of rings]
  Let $f\colon R \rightarrow S$ be map of simplicial rings which is a weak
  equivalence. Then $f$ induces a map $\mathcal{C}^G_{f}(X; R) \rightarrow
  \mathcal{C}^G_{f}(X; S)$ which is an equivalence on algebraic K-Theory.
\end{repthm}

For technical reasons we assume all modules to be finite-dimensional in this
section.  Therefore, we only have this theorem for the finiteness condition
$f$.  Note that the theorem for $hf$ and $hfd$, except the $K_0$-part of
$hfd$, is already implied by the Theorem~\ref{thm:change-of-rings} using
Theorem~\ref{thm:k-theory-comparison-f-hf-hfd}.

The proof occupies the rest of this section, we need some preparations first.  A
map of simplicial rings which is a weak equivalence of the underlying
simplicial sets is called a \emph{weak equivalence of simplicial rings} for
short.

\begin{lem}\label{lem:w-eq-of-rings-give-w-eq-of-modules}
  Let $R \rightarrow S$ be a weak equivalence of simplicial rings and $P$ a
  cellular (uncontrolled) $R$-module. Let $\eta\colon P \rightarrow \res_R
  S\otimes_R P$ be the unit of the adjunction between the induction
  $S\otimes_R  -$ and the restriction $\res_R$.  Then $\eta$ is a weak
  equivalence of simplicial $R$-modules and, in particular, a homotopy
  equivalence of simplicial sets.
\end{lem}

\begin{proof}
  This follows from the gluing lemma and induction over the dimension of $P$.
  We obtain a pushout-diagram
  \begin{equation*}
    \xymatrix{
    \coprod R[\Delta^n]  \ar[d]^{\simeq}  & 
    \coprod R[\partial\Delta^n] \ar[l] \ar[r]\ar[d]^{\simeq} &
    P_{n-1} \ar[d]^{\simeq} \\
    \coprod S[\Delta^n]  & 
    \coprod S[\partial\Delta^n] \ar[l] \ar[r]&
    S\otimes_R P_{n-1} }
  \end{equation*}
  where the vertical maps are weak equivalences of simplicial $R$-modules,
  hence by the gluing lemma for simplicial $R$-modules
  (cf.~\cite[II.8.12;III.2.14]{goerss-jardine;simplicial-homotopy-theory}) the
  pushout $P_n \rightarrow S\otimes_R P_n$ is a weak equivalence.   We
  have $P = \bigcup_n P_{n}$ and the $n$-skeleton of $P$ and $P_n$
  agree. Also the $n$-skeleton of $S\otimes_R P$ and $S \otimes_R P_n$
  agree and $S \otimes_R P = \bigcup_n S \otimes_R P_n$.  Therefore, $P
  \rightarrow S\otimes_R P$ is a weak equivalence.  As simplicial
  abelian groups are fibrant as simplicial sets the weak equivalence is
  a homotopy equivalence of simplicial sets.
\end{proof}

\begin{lem} \label{lem:proof-of-R-S-weq-lifts}
  Let $M, P\in \mathcal{C}^G_{a}(X;R)$.  Let $i\colon A \subseteq M$ be a
  cellular submodule.  Assume $M$ is finite-dimensional.  Let $g\colon A
  \rightarrowtail P$ and $\widehat{f}\colon S\otimes_R M \rightarrow S
  \otimes_R P$ be maps such that the diagram
  \begin{equation}
    \xymatrix{
    S\otimes_R A \ar[r]^{S\otimes g} \ar@{ >->}[d]_{S\otimes i}  & S\otimes_R
    P \\
    S\otimes_R M \ar[ur]_{\widehat{f}} }
  \end{equation}
  commutes.  Then there is a map $f\colon M \rightarrow P$ extending $g$ such
  that $\widehat{f}$ is homotopic to $S\otimes_R f$ relative to the cellular
  submodule $S\otimes_R A$ of $S \otimes_R M$.
\end{lem}

The lemma shows that up to homotopy a map $S \otimes_R M \rightarrow
S\otimes_R P$ comes from a map $M \rightarrow P$.  We need the relative
version.

\begin{proof}
  Assume $M$, $\widehat{f}$, $g$, $P$ are $E$-controlled.  We do induction
  over the dimension of cells of $M$ which are not in $A$.  As usual it
  suffices to consider only one cell.  Let $e\colon R[\Delta^n] \rightarrow M$
  be attached to $A$ via $\partial \colon R[\partial \Delta^n] \rightarrow A$.

  Looking at the smallest submodules containing $e$, $\partial e$ and
  $\widehat{f}(S\otimes_R e)$ we obtain the following commutative diagram.  (We
  denote by $S \otimes_R e$ the cell in $S\otimes_R M$ corresponding to $e$
  via the isomorphism $\diamond_R M \cong \diamond_S (S\otimes_R M)$.)
  \begin{equation*}
    \xymatrix{
    S \otimes_R \left< \partial e\right>_A \ar[r] \ar@{ >->}[d] &
      \left<\widehat{f}(S\otimes_R e) \right>_{S\otimes_R P} \\
    S\otimes_R \left<e\right>_M \ar[ur] }
  \end{equation*}
  Because everything is $E$-controlled the support of every module is
  contained in
  $\left\{\kappa(e)\right\}^{E}$.  Note that $\left< \widehat{f}(S\otimes_R e)
  \right>_{S\otimes_R P}$ is isomorphic to $S\otimes_R P'$ for $P'
  \subseteq P$ a cellular $R$-submodule.  Using the adjunction $S\otimes_R -$
  and restriction $\res_R $ we obtain the solid commutative diagram below.
  \begin{equation*}
    \xymatrix{
    R[\partial\Delta^n] \ar@{-->}[r]\ar@{ >-->}[d] &
      \left< \partial e\right>_A \ar[r] \ar@{ >->}[d] &
      P' \ar[d]^{\eta}_{\sim} \\
    R[\Delta^n] \ar@{-->}[r]&
      \left< e\right>_M \ar[r] &
      \res_R S \otimes_R P'
    }
  \end{equation*}
  Here $\eta$ is the unit of the adjunction.  We want to find a lift up to
  homotopy relative to $\left< \partial e\right>_A$ in the solid diagram.  We
  can extend the diagram to the left by the dashed square, which is a pushout square.
  Hence, using the adjunction $R[-]$ and forgetful functor, it suffices to
  construct a lift up to homotopy relative to $\partial \Delta^n$ in the
  diagram of simplicial sets
  \begin{equation*}
    \xymatrix{
    \partial \Delta^n \ar[r]^{g} \ar@{ >->}[d]^{i} & P'
    \ar[d]^{\eta}_{\sim} \\
    \Delta^n \ar[r]^-{\widehat{f}} & \res_R S \otimes_R P'} .
  \end{equation*}
  There is a lift up to homotopy because by
  Lemma~\ref{lem:w-eq-of-rings-give-w-eq-of-modules} $\eta$ is a homotopy
  equivalence of simplicial sets.  With more effort we can arrange that we
  lift to a map $f\colon \Delta^n \rightarrow P'$ with $f\circ i = g$ and
  $\eta \circ f$ is homotopic to $\widehat{f}$ with the homotopy being trivial
  when restricted via $i$.  (Use the mapping cylinder to factor $\eta$
  and then use that we can find strict lifts for surjective homotopy
  equivalences and deformation retractions.)
%
%

  We obtain a map 
  \begin{equation*}
    f\colon \left< e \right>_M \rightarrow P'
  \end{equation*}
  such that $S\otimes_R f$ is homotopic to $\widehat{f}\colon S \otimes_R
  \left< e \right>_M \rightarrow S \otimes_R P'$ relative to $S\otimes_R
  \left< \partial e \right>_A$ and $f_{|\left<\partial e\right>_A} =
  g_{|\left<\partial e\right>_A}$.
  As $S\otimes_R P'$ has support on $\left\{
  \kappa_R(e)\}^{E}\right\}$ the map and the homotopy are $E$-controlled.

  Assuming the first cells of $M$ which are not in $A$ are of dimension
  $n$, we can use this procedure and the homotopy extension property to
  produce a map $S\otimes_R M \rightarrow S\otimes_R P$ which satisfies the
  assumption of the lemma for an $A' = A \cup \sk_n M$.  Induction and the
  finite-dimensionality of $M$ finishes the proof.
\end{proof}

\begin{proof}[Proof of Theorem~\ref{thm:change-of-rings}]
  We use Waldhausen's Approximation
  Theorem~\ref{thm:approximation-theorem} for the
  functor $F := S\otimes_R -\colon \mathcal{C}^G_{f}(X;R) \rightarrow
  \mathcal{C}^G_{f}(X;S)$.
  We prove (App~1) first. Let $\alpha\colon M \rightarrow M'$ be a map in
  $\mathcal{C}^G_{f}(X; R)$ such that $S\otimes_R \alpha$ is a homotopy equivalence in
  $\mathcal{C}^G_{f}(X; S)$. By Lemma~\ref{lem:proof-of-R-S-weq-lifts}, there
  is a map $\beta'\colon M' \rightarrow M$ such that the
  homotopy inverse $\beta\colon S \otimes_R M' \rightarrow S \otimes_R M$ of
  $S\otimes_R \alpha$ in $\mathcal{C}^G_{?}(X; S)$ is
  homotopic to $S\otimes_R \beta'$.
  Hence, there is a homotopy $H \colon S \otimes_R
  M[\Delta^1] \rightarrow S\otimes_R M$ from $S\otimes_R \id_R$ to $S\otimes_R
  (\beta' \circ \alpha)$ in $\mathcal{C}^G_{f}(X;S)$ which is homotopic
  relative to $M[\partial \Delta^1]$ to a homotopy $S\otimes_R H'$ where
  $H'$ is a homotopy from $\id_R$ to $\beta' \circ \alpha$, using
  Lemma~\ref{lem:proof-of-R-S-weq-lifts} again. Vice versa for $\alpha \circ
  \beta'$, so $\alpha$ is also a homotopy equivalence in
  $\mathcal{C}^G_{f}(X;R)$.\smallskip

  For (App~2) consider $M\in \mathcal{C}^G_{f}(X;R)$ and $N \in
  \mathcal{C}^G_{f}(X;S)$ and a map $f\colon S \otimes_R M \rightarrow N$.  We
  can assume that it is a cellular inclusion by taking the mapping cylinder.
  We show that
  $N$ is homotopy equivalent to a module $S\otimes_R
  \overline{M}^2$, with $\overline{M}^2\in\mathcal{C}^G_{f}(X;R)$, $M
  \subseteq \overline{M}^2$ and the homotopies are relative to $S\otimes_R
  M$.

  Assume that the $n$-skeleton of $S\otimes_R M$ and $N$ agree.  Let
  $N^{n+1}$ be the $(n+1)$-skeleton of $N$ relative to $S\otimes_R M$, i.e.,
  $N^{n+1} = \sk_{n+1} N \cup S\otimes_R M$.  Then $N^{n+1}$ is the pushout
  \begin{equation*}
    \xymatrix{
    S[\coprod\Delta^{n+1}] & 
    S[\coprod\partial \Delta^{n+1}] \ar[r]^-{\varphi^{n+1}}\ar@{ >->}[l]&
      S\otimes_R M  }.
  \end{equation*}
  where $\varphi^{n+1}$ is the attaching map for the cells.  By
  Lemma~\ref{lem:proof-of-R-S-weq-lifts}, there is a map $\psi^{n+1}\colon
  R[\coprod\partial\Delta^{n+1}] \rightarrow {M}$ such that
  $S\otimes_R\psi^{n+1}$ is homotopic to a $\varphi^{n+1}$. Call the homotopy
  $H^{n+1}$. Applying the gluing lemma to the diagram (where all vertical maps
  are homotopy equivalences)
  \begin{equation*}
    \xymatrix{
    S[\coprod\Delta^{n+1}]\ar[d] & 
      S[\coprod\partial \Delta^{n+1}] \ar[r]^{\varphi^{n+1}}\ar[l]\ar[d]&
      S\otimes_R {M}  \ar[d] \\
    S[\coprod\Delta^{n+1}][\Delta^1]   & 
      S[\coprod\partial \Delta^{n+1}] \ar[r]^{H^{n+1}}\ar[l][\Delta^1]  & 
      S\otimes_R {M} [\Delta^1] \\
    S\otimes_R R[\coprod\Delta^{n+1}]\ar[u] & 
    S\otimes_R R[\coprod\partial \Delta^{n+1}] 
      \ar[r]^-{S\otimes \psi^{n+1}}\ar[l]\ar[u] & 
      S\otimes_R {M} \ar[u]
    }
  \end{equation*}
  shows that the pushout of the first row is homotopy equivalent to the
  pushout of the last row. (This is a simplicial version of the topological
  fact that homotopic attaching maps yield homotopy equivalent CW-complexes.)
  Choose such a homotopy equivalence $\xi$.
  In the last row $S\otimes_R -$ commutes with the pushout, define
  $\overline{M}$ as the pushout of
  \begin{equation*}
    \xymatrix{
    R[\coprod\Delta^{n+1}] & 
      R[\coprod\partial \Delta^{n+1}] \ar[r]^-{\psi^{n+1}}\ar[l]&
      {M}  }.
  \end{equation*}
  Then consider the pushout along $\xi\colon N^{n+1}\rightarrow S\otimes_R M$ and
  the inclusion $N^{n+1} \rightarrowtail N$ to obtain $\overline{N}$:
  \begin{equation*}
    \xymatrix{
      N^{n+1} \ar@{ >->}[r] \ar@{ >->}[d]^{\simeq}_{\xi} & N \ar[d]^\simeq \\
      S\otimes_R \overline{M} \ar@{ >->}[r]_-{\overline{f}} & \overline{N}
    }
  \end{equation*}
  Now the $(n+1)$-skeleton of $S\otimes_R \overline{M}$ is isomorphic to
  $\overline{N}$ via $\overline{f}$.  By induction, and because $N$ is
  finite-dimensional, we obtain a diagram 
  \begin{equation*}
    \xymatrix{
      S\otimes_R M \ar@{ >->}[d] \ar@{ >->}[r] & N \ar[d]^{\simeq} \\
      S\otimes_R \overline{M}^1 \ar[r]^-{\cong} & \overline{N}^{1}}
  \end{equation*}
  which we can make into the desired diagramm
  \begin{equation*}
    \xymatrix{
    S\otimes_R M \ar[r] \ar[d]_{S\otimes -}  & N \\
    S\otimes_R \overline{M}^2 \ar[ur]_{\sim} }.
  \end{equation*}
  using a homotopy inverse for the right map and defining $\overline{M}^2$ as the
  mapping cylinder of $M \rightarrow \overline{M}^1$ to make the diagram
  strictly commutative.  This proves (App~2). The theorem follows by
  the Approximation Theorem~\ref{thm:approximation-theorem}.
\end{proof}

\section{Applications}
\label{sec:miscellaneous}
We discuss several applications of our category
$\mathcal{C}^G(X;R)$.  We will not provide proofs because they require
considerably more technology.

\subsection{Controlled algebra for discrete rings}
\label{subsec:controlled-algebra-discrete}
Each discrete ring $R_{d}$ can be made into a simplicial ring taking the
constant functor $[n] \mapsto R_d$ where all structure maps are the identity.
Then a simplicial module over $R_d$ is, essentially, by the
Dold-Kan-Theorem (see~\cite[III.2.3]{goerss-jardine;simplicial-homotopy-theory}), a
non-negatively graded chain complex.  More precisely, there is an adjunction
between simplicial modules and non-negative chain complexes and this
adjunction is a equivalence on homotopy categories.  Furthermore, a cellular
simplicial $R_d$-module with cells only in dimension $0$ is just a free
$R_d$-module.

Consider the subcategory of $0$-dimensional controlled cellular simplicial
$R_d$-modules, which no longer has nice homotopic properties, but is an
additive category.
In fact, this is essentially the category of controlled $R_d$-modules of
e.g.~\cite{Bartels-Farrell-Jones-Reich;;;}
or~\cite{Pedersen-Weibel;nonconnective}.  (There is a small technical
difference between our work and that of~\cite{Bartels-Farrell-Jones-Reich;;;}
in the definition of morphisms, however, that does not affect the algebraic
K-theory.) 

Therefore, the category $\mathcal{C}^G(X;R)$ we present here can be viewed as a
homotopical generalization of the category of controlled $R_d$-modules.
Unfortunately, it comes with a price: the arguments to surrounding
$\mathcal{C}^G(X;R)$ become more involved.
There are simple arguments in the case of discrete modules and rings which
under this generalization would necessitate invoking Waldhausen's
approximation theorem.
Still, we believe that most arguments have an analogue for
$\mathcal{C}^G(X;R)$.

\subsection{The Farrell-Jones Conjecture}
\label{subsec:farrell-jones}
\subsubsection{Statement and Significance}
Let $R$ be a discrete ring or a simplicial ring and $G$ a group.  The
Farrell-Jones Conjecture provides a calculation of $K_n(R[G])$,
$n\in\mathbb{Z}$, the algebraic K-theory of the group ring $R[G]$,
in terms of the algebraic K-theory of $R$ and the geometry of the group
$G$.  More precisely, it claims that the \emph{assembly map}
\begin{equation*}
  H^G_n(E_{\mathcal{VC}} G; \mathbf{K}_R) \rightarrow K_n(R[G])
\end{equation*}
is an isomorphism for every $n \in \mathbb{Z}$.   Here the right-hand side is
the nonconnective algebraic K-theory of the group ring $R[G]$, while the
left-hand side is the $G$-equivariant homology theory with coefficients in the
$G$-equivariant nonconnective K-theory spectrum, evaluated at the
classifying space of $G$ for virtual cyclic subgroups.
We refrain from discussing details, as the Farrell-Jones Conjecture is
not our main focus in this article and refer to
\cite{arxiv:Bartels;on-proofs-of} or the slightly outdated
survey~\cite{Luck-Reich;conjectures}.  

The Farrell-Jones Conjecture implies a plethora of other usually long-standing
conjectures.  This includes the vanishing of the Whitehead Group for
torsion-free groups and the Borel Conjecture about the rigidity of aspherical
manifolds. We refer to
\cite{Luck-Reich;conjectures,Bartels-Reich;applications} for details.
Therefore, it is interesting to know the Farrell-Jones Conjecture for as many
rings $R$, called the \emph{coefficients}, and groups $G$ as possible.

\subsubsection{Status and Proofs}
There is recent and ongoing progress on the
Farrell-Jones Conjecture which has substantially enlarged the class of groups
for which it is known.  Recent approaches prove a more
general version, the ``Farrell-Jones Conjecture with wreath products'', see
Section~6 of \cite{arxiv:Bartels-Reich-Rueping;group-rings}.  Also, that
version allows any additive category $\mathcal{A}$ as coefficients. If
$\mathcal{A}$ is the category of finitely generated free $R$-modules, we
recover the version stated above.

Recent work by Bartels, Farrell, Lück, Reich, Rüping, Wegner, Wu and
others establishes the ``Farrell-Jones Conjecture with wreath products and
coefficients in an additive category'' for large classes of groups---most
recently for $GL_n(\mathbb{Z})$ and some related groups
in~\cite{arxiv:Bartels-Reich-Rueping;group-rings}, solvable Baumslag-Solitar
groups in~\cite{arxiv:Farrell-Wu;solvable-Baumslag-Solitar} and, more
generally, solvable groups
in~\cite{arxiv:Wegner;solvable}.

All recent proofs have in common that they start by translating the
Farrell-Jones
Conjecture to a problem in the algebraic K-theory of controlled algebra, a
strategy first formulated in this way in~\cite{Bartels-Reich;hyperbolic}.

\subsubsection{A reformulation in terms of controlled algebra}
Let $Z$ be a $G$-CW-complex.  We want to make $X \times G \times [1,\infty)$ into
a $G$-control space.  Recall the continuous control structure
$\mathcal{E}_{cc}$ on $X \times [1,\infty)$
from Example~\ref{example:continous-control}.  We can pull back this
morphism control conditions along the projection $p\colon X \times G \times
[1,\infty) \rightarrow X \times [1,\infty)$ by setting
\begin{equation*}
  p^{-1}\mathcal{E}_{cc} := \{ (p\times p)^{-1}(E) \mid E \in
  \mathcal{E}_{cc}\}.
\end{equation*}
Then $p^{-1}\mathcal{E}_{cc}$ is a morphism control structure on $X \times G
\times [1,\infty)$.  We obtain an object support structure
by setting
\begin{equation*}
  \mathcal{F}_{Gc}(X\times G) := \{G.K \times [1,\infty) \mid K \subseteq X
  \times G\text{ is compact}\}
\end{equation*}
where $G.K$ is the $G$-orbit of K.  If $X$ is $E_{\mathcal{VC}} G$, the
classifying space for $G$ (cf.~\cite[I.6]{dieck;transformation-groups})
and the family $\mathcal{VC}$ of virtually cyclic subgroups, we obtain a
category 
\begin{equation*}
  \mathcal{O}^G = 
    \mathcal{C}^G_f(E_{\mathcal{VC}} G \times G \times [1,\infty), p^{-1}\mathcal{E}_{cc},
    \mathcal{F}_{Gc}; R)
\end{equation*}
of controlled simplicial $R$-modules.  As explained above, for a discrete
ring $R_d$ there is a similar category of discrete controlled modules which we
call $\mathcal{O}^G_d$ for brevity.  As it is an additive category, its
algebraic K-theory is defined.
\begin{thm*}[{\cite[3.8]{Bartels-Reich;hyperbolic}}]
  $K_i(\mathcal{O}^G_d) = 0$ for all $i \in \mathbb{N}$ if and only if the
  Farrell-Jones Conjecture holds for $G$.
\end{thm*}
Hence, one can use controlled algebra and manipulation of the control space to
prove the Farrell-Jones Conjecture.  This is in fact the strategy carried out
by recent proofs.

For simplicial rings, an analogue of the theorem holds by unpublished work of
the author in~\cite{thesis-ullmann}.  Thus this article should be viewed as
first step to carry out the successful program of proving the Farrell-Jones
Conjecture for discrete rings in the settings of simplicial rings.

\subsection{Nonconnective algebraic K-theory}
\label{subsec:non-connective-algebraic-k-theory}
We provide a second direct application without proof.
Consider the control space $(\mathbb{R}^n, \mathcal{E}_d)$ arising from the
euclidean metric.  Then there is a map
\begin{equation*}
  K(R)\rightarrow \Omega^n K(\mathcal{C}_f(\mathbb{R}^n;R))
\end{equation*}
of connective K-theory spectra which is an isomorphism on $\pi_i$ for $i\geq
1$ and an injection on $\pi_0$.  This deloops $K(R)$, that means
the $K(\mathcal{C}_f(\mathbb{R}^n;R))$ for varying $n$ can be made into a
spectrum which may have interesting negative homotopy groups and where the positive
homotopy groups are those of $K(R)$.  This is the first construction of a
nonconnective K-theory spectrum for simplicial rings.  It generalizes the
delooping construction of $K(R_d)$ of~\cite{Pedersen-Weibel;nonconnective}.
However, is it known that $\pi_i K(R) = K_i(\pi_0 R)$ for $i = 0,1$.  Because
a Bass-Heller-Swan theorem is expected to hold for algebraic K-theory of
simplicial rings, this means that the negative algebraic K-groups of a
simplicial ring are just those of the discrete ring $\pi_0 R$.  But of
course, a spectrum contains more information than its homotopy groups.

The proofs of both applications need considerably more technology in
$\mathcal{C}^G(X;R)$, namely a notion of germs, which were developed
in~\cite{thesis-ullmann}.  We will come back to these in later work.

\subsection{Ring spectra}
\label{subsec:ring-spectra}
There are generalizations of rings for which the Farrell-Jones Conjecture
should give interesting results with implications to manifold theory.
\emph{Ring spectra} provide a natural generalization of
rings, and simplicial rings are an intermediate step between rings and ring
spectra.  Algebraic K-theory can be defined for ring spectra, see
in~\cite{Elmendorf-Kriz-Mandell;rings-modules}.  The statement of the
Farrell-Jones Conjecture makes sense with connective ring spectra as
coefficients.  In fact,
when Farrell and Jones
in~\cite{farrell-jones;isomorphism-conjectures,farrell-jones;k-theory-dynamics}
originally stated and proved a version of
their conjecture (for a certain class of groups), they also treated the case
of pseudoisotopies, which is more or less the case where the sphere spectrum
are the coefficients.  The sphere spectrum is the initial ring spectrum, like
the integers are the initial ring.

We hope that the theory presented here can be adapted to ring spectra.  We do
not want to go into details, but let us remark that in the 1990's a
variety of different models for ring spectra and categories of modules over ring
spectra where developed.  The main models are symmetric
spectra \cite{Hovey-Shipley-Smith;symmetric-spectra}, orthogonal spectra
\cite{May-Mandell-Schwede-Shipley;;;} and
$S$-modules~\cite{Elmendorf-Kriz-Mandell;rings-modules}.  The category of
$S$-modules is special among these as it has the nice property that it has a
model structure such that every object is fibrant and
\cite[III.2]{Elmendorf-Kriz-Mandell;rings-modules} provides a nice theory of
cellular objects.  As we build our category $\mathcal{C}^G(X;R)$ from cellular
modules (cf. Definition~\ref{def:controlled-module}), it looks like
cellular $S$-modules are a suitable candidate to carry out the program
presented here.  However, there will be a non-trivial amount of work involved,
as the category of $S$-modules is rather hard to define.

\appendix
\section{Appendix: a simplicial mapping telescope}
\label{sec:appendix-simplicial-mapping-telescope}
A map $\eta\colon K \rightarrow K$ in $\mathcal{C}^G$ is called a
\emph{homotopy idempotent} if $\eta^2$ is homotopic to~$\eta$.
Here we provide the necessary tools we need about homotopy idempotents in this
and later work.  This gives some insight into the category
$\mathcal{C}^G_{hfd}(X;R)$, for any control space $X$ and simplicial ring $R$.

We defined the Waldhausen category
 $\mathcal{C}^G =
\mathcal{C}^G(X,\mathcal{E},\mathcal{F};R)$ for a control space
$(X,\mathcal{E}, \mathcal{F})$ and a simplicial ring $R$ in
Section~\ref{subsubsec:G-equivariant-controlled-modules}.

\subsection{Coherent homotopy idempotents}

Some parts of Theorem~\ref{thm:mapping-telescopes} below need an extra
assumption on the idempotent, which we will define now.
We use the diagram language of~\ref{subsubsec:diagrams} in what follows. 
\begin{definition}\label{def:coherent-homotopy-idempotent}
  A homotopy idempotent $\eta\colon K \rightarrow K$ with homotopy $H$ from
  $\eta^2$ to $\eta$ is called \emph{coherent} if there is a map $G\colon
  K[\Delta^1 \times \Delta^1] \rightarrow K$ whose restrictions to the
  boundary are given by the following diagram
  \begin{equation*}
    \xymatrix{ 
    \bullet \ar[r]^{\eta\circ H} \ar[d]_{H\circ \eta[\Delta^1]} \ar[dr] &
     \bullet \ar[d]^H \\ \bullet \ar[r]_H & \bullet}.
  \end{equation*}
\end{definition}
If $\eta^2 = \eta$ then $\eta$ is coherent.
More generally, assume $L$ is homotopy dominated by $K$
(cf.~Definition~\ref{def:retract-homotopy-retract}).  This means that we
have maps $i\colon L \rightarrow K$, $p\colon K\rightarrow L$ and a homotopy
$H' \colon p\circ i \simeq \id_L$.  Then $\eta := i \circ p$ 
is a homotopy idempotent with homotopy $i\circ H' \circ p$ from $\eta^2$ to
$\eta$.
\begin{lem}\label{lem:coherent}
  If $\eta$ arises from a homotopy domination, it is coherent.
\end{lem}

\begin{proof}
  The coherence homotopy $G$ can be given by the composition
  \begin{equation*}
    K[\Delta^1\times \Delta^1] \cong K[\Delta^1][\Delta^1]
    \xrightarrow{p[\Delta^1][\Delta^1]} L[\Delta^1][\Delta^1]
    \xrightarrow{H'[\Delta^1]} L[\Delta^1]
    \xrightarrow{H'} L \xrightarrow{i} K.\qedhere
  \end{equation*}
\end{proof}

\subsection{A mapping telescope and split homotopy idempotents}
We will show that if $\eta$ is a coherent homotopy idempotent in $\mathcal{C}^G$, then it splits
up to homotopy.  For this we use a construction analogous to the mapping
telescope in topological spaces.  We summarize the results we need in
Theorem~\ref{thm:mapping-telescopes} and directly deduce the splitting as
Corollary~\ref{cor:hpty-idem-split}.  The proof of
Theorem~\ref{thm:mapping-telescopes} itself will occupy the rest of this
appendix.

\begin{rem}
  The author does not know if every homotopy idempotent in $\mathcal{C}^G$
  is coherent.  For the topological case it is known that there are unpointed
  homotopy idempotents of infinite-dimensional CW-complexes which do not
  split, however every pointed homotopy idempotent as well as every homotopy
  idempotent of finite-dimensional CW-complexes split,
  see~\cite{Hastings-Heller;idempotents-finite-dimensional}.
\end{rem}

\begin{thm}\label{thm:mapping-telescopes}
  Let $\eta \colon K \rightarrow K$ be map in
  $\mathcal{C}^G(X)$.  There is a construction $\Tel(-)$ which assigns to $\eta$
  an object $\Tel(\eta)$ in $\mathcal{C}^G(X)$.  It has the following
  properties:
  \begin{enumerate}
    \item \label{item:mapping-tel:inclusion}
      There is a cellular inclusion $\iota\colon K \rightarrowtail
      \Tel(\eta)$.
    \item \label{item:mapping-tel:map-of-idempotents}
      Let
      \begin{equation*}
	\xymatrix{ A \ar[r]^{\mu} \ar[d]^f & A \ar[d]^f \\
	   K\ar[r]^\eta & K}
      \end{equation*}
      be a strict commutative diagram. 
      Then $f$ induces a map $f_*\colon \Tel(\mu)
      \rightarrow \Tel(\eta)$.  This is functorial in $f$.  In particular, if
      $f$ is an isomorphism then $\Tel(f)$ is an isomorphism.
    \item \label{item:mapping-tel:homotopy-id-give-hpty-equiv-tels}
      If $\eta, \mu\colon K \rightarrow K$ are homotopic maps
      then there is a homotopy equivalence
      \begin{equation*}
	\Tel(\eta) \xrightarrow{\simeq} \Tel(\mu).
      \end{equation*}
    \item \label{item:mapping-tel:tel-of-id-hpty-equiv-to-obj}
      Consider the telescope $\Tel(\id_K)$ of the map $\id_K \colon K
      \rightarrow K$.  There is a map
      \begin{equation*}
	\Tel(\id_K) \rightarrow K
      \end{equation*}
      which is a homotopy equivalence.
    \item \label{item:mapping-tel:relative-to-space} 
      All maps in \eqref{item:mapping-tel:map-of-idempotents}
      to~\eqref{item:mapping-tel:tel-of-id-hpty-equiv-to-obj} are relative to
      $\iota\colon K \rightarrow \Tel(\eta)$, i.e.,~they commute with this
      cellular inclusion.
  \end{enumerate}
  Now assume additionally that $\eta\colon K \rightarrow K$ is a homotopy
  idempotent in $\mathcal{C}^G(X)$.
  \begin{enumerate}[resume]
    \item \label{item:mapping-tel:idempot-induces-id-on-tel}
      From \eqref{item:mapping-tel:map-of-idempotents} we obtain for $\mu = \eta
      = f$ an induced map $\eta_* \colon \Tel(\eta) \rightarrow \Tel(\eta)$.
      This map is a homotopy equivalence.  If $\eta$ is coherent, $\eta_*$ is
      homotopic to $\id$.
    \item \label{item:mapping-tel:hpty-retract}
      If $\eta$ is coherent then there is a map $c\colon \Tel(\eta)
      \rightarrow K$ such that $\iota \circ c$ is homotopic to
      $\id_{\Tel(\eta)}$.  Therefore, $\Tel(\eta)$ is a homotopy retract of
      $K$.  Furthermore, $c\circ \iota$ is homotopic to $\eta$ itself.
  \end{enumerate}
\end{thm}

\begin{cor}[Coherent homotopy idempotents split]\label{cor:hpty-idem-split}
  Let $\eta\colon K \rightarrow K$ be a coherent homotopy idempotent in
  $\mathcal{C}^G$.  Then there is a $B \in \mathcal{C}^G$ such that $K$ is
  homotopy equivalent to $\Tel(\eta)\vee B$.  Moreover, under this equivalence
  $\eta$ corresponds to the projection $\pr\colon \Tel(\eta) \vee B
  \rightarrow \Tel(\eta) \rightarrow \Tel(\eta) \vee B$, i.e.,~there is a
  homotopy commutative diagram
  \begin{equation*}
    \xymatrix{
    K \ar[r]^-f \ar[d]^{\eta}  &  \Tel(\eta) \vee B \ar[d]^{\pr}  \\
    K \ar[r]^-f                &  \Tel(\eta) \vee B }
  \end{equation*}
  where $f$ is the homotopy equivalence $K \xrightarrow{\simeq} \Tel(\eta)
  \vee B$.
\end{cor}

\begin{proof}
  We know by
  \ref{thm:mapping-telescopes}~\eqref{item:mapping-tel:hpty-retract} that
  $\Tel(\eta)$ is a homotopy retract of $K$.  We use the mapping cylinder and
  the  procedure from Section~\ref{subsubsec:rectifying-diagrams} to produce a
  \emph{strictly} commutative diagram (on the left) from the
  \emph{homotopy} commutative diagram (on the right):
  \begin{equation*}
    {\xymatrix{\Tel(\eta) \ar[r]^-c \ar[dr]_{\id} & K \ar[d]^\iota \\
    & \Tel(\eta) }} \qquad  \qquad
    {\xymatrix{
      \Tel(\eta) \ar@{ >->}[r]^{\inc} \ar[dr]_{\id} & T(c) \ar[d]\\
      &\Tel(\eta)}}
  \end{equation*}
  Here $T(c)$ is the mapping cylinder of $c$ and $\inc$ is a cellular
  inclusion.  Let $B$ denote the cofiber of $\inc$.  The sum of the retraction
  $T(c) \rightarrow \Tel(\eta)$ and the quotient map $T(c) \rightarrow B$
  produces a map $s\colon T(c)
  \rightarrow \Tel(\eta) \vee B$ (using that $\mathcal{C}^G$ is an additive
  category).  The map makes the following diagram of cofiber sequences
  commutative.
  \begin{equation*}
    \xymatrix{
      \Tel(\eta) \ar@{ >->}[r] \ar[d]  & T(c) \ar[d]^s\ar@{->>}[r] & B \ar[d] \\
      \Tel(\eta) \ar@{ >->}[r]  & \Tel(\eta)\vee B  \ar@{->>}[r] & B}
  \end{equation*}
  The extension axiom~\ref{lem:waldh-cat:extension-axiom}
  shows that $s$ is a homotopy equivalence.  This gives the homotopy
  equivalence $f\colon K \rightarrow T(c) \rightarrow \Tel(\eta) \vee
  B$.\smallskip

  By \ref{thm:mapping-telescopes}~\eqref{item:mapping-tel:hpty-retract}, the
  map $\eta\colon K \rightarrow K$ factorizes up to homotopy as $c\circ i
  \colon K \rightarrow \Tel(\eta) \rightarrow K$.  Hence, the upper triangle
  in the following diagram is homotopy commutative, whereas the lower one
  commutes strictly. 
  \begin{equation*}
    \xymatrix{
    K \ar[r]^\eta \ar[d]_{\iota} & K \ar[d] \\
    \Tel(\eta) \ar[ur]^c \ar@{ >->}[r]  & T(c) \ar@{->>}[r] & B}
  \end{equation*}
  It follows that $K\xrightarrow{\eta} K \rightarrow T(c) \rightarrow B$ is
  homotopic to the zero map.

  Furthermore, $K\rightarrow T(c) \rightarrow \Tel(\eta)$ equals $\iota$,
  hence by adding homotopies the map
  \begin{equation*}
    f\circ \eta\colon K \xrightarrow{} K \xrightarrow{\simeq} T(c) \xrightarrow{\simeq}
    \Tel(\eta) \vee B
  \end{equation*}
  is homotopic to $K\xrightarrow{\iota} \Tel(\eta)\rightarrowtail
  \Tel(\eta)\vee B$.  The following diagram is strictly commutative by
  \ref{thm:mapping-telescopes}~\eqref{item:mapping-tel:map-of-idempotents}
  and~\eqref{item:mapping-tel:relative-to-space}, where $0_B$ denotes the zero
  map on $B$.
  \begin{equation*}
    \xymatrix{
      K \ar[r]^-\iota \ar[d]_\eta  & \Tel(\eta) \vee B \ar[d]^{\eta_* \vee 0_B} \\
      K \ar[r]^-\iota              & \Tel(\eta) \vee B }
  \end{equation*}
  As $\eta_* \simeq \id$ by
  \ref{thm:mapping-telescopes}~\eqref{item:mapping-tel:idempot-induces-id-on-tel}
  it follows that 
  \begin{equation*}
    \xymatrix{
    K \ar[r]^-{\simeq}_-f \ar[d]^{\eta}  & \Tel(\eta) \vee B \ar[d]^{\id \vee
    0_B} \\
    K \ar[r]^-{\simeq}_-f & \Tel(\eta) \vee B}
  \end{equation*}
  is homotopy commutative.
\end{proof}

The proof of Theorem~\ref{thm:mapping-telescopes} is a bit involved.  We give
an outline of the arguments in
Section~\ref{subsec:appendix-structure-of-proof} before we turn to the
technical details.

\subsection{Structure of the proof}
\label{subsec:appendix-structure-of-proof}

Let us outline the proof of Theorem~\ref{thm:mapping-telescopes}.  
We list the sections in which we prove the claimed
statements at the end of this section.
In this section let $I$ be the simplicial interval $\Delta^1$. 

\subsubsection{The Definition.}
\label{subsubsec:appendix-definition-telescope}
From $\eta\colon K \rightarrow K$ we can build the mapping cylinder
$M^I(\eta)$, which is the (solid) pushout diagram below.
\begin{equation}
  \label{eq:appendix:structure:cyl}
  \vcenter{\xymatrix{
  & K[1] \ar[r]^{\eta} \ar@{ >->}[d] & K \ar[d]^{i_1} \\
  K[0] \ar@{ >-->}[r] \ar@/_10pt/@{-->}[rr]_{\iota_0} &
  K[I] \ar[r] & M^I(\eta) }}
\end{equation}
We obtain two inclusions: $\iota_1$ and the shown composition $\iota_0$.
Gluing infinitely many cylinders together we obtain the mapping telescope
$\Tel(\eta)$ as the following pushout.
\begin{equation}
  \label{diag:definition-telescope}
    \vcenter{\xymatrix{
    \coprod_{i=1}^{\infty} A[i] \amalg \coprod_{i=0}^{\infty} A[i+1]
    \ar@{ >->}[r]^-{\iota_0 \amalg \iota_1} \ar[d]^c &
    \coprod_{i=0}^{\infty} M^{I}(\eta) \ar[d] \\
    \coprod_{i=1}^\infty A[i] \ar[r]  & \Tel(\eta)}}
\end{equation}
  The map $\iota_0 \amalg \iota_1$ sends the summand $A[i]$ into the
  $i$th copy of $M^{I}(\eta)$ via $\iota_0$ and the summand $A[j+1]$
  into the $i$th copy of $M^{I}(\eta)$ via $\iota_1$. The map $c$ sends $A[l]$
  to $A[l]$.  This defines $\Tel(\eta)$.

\subsubsection{Theorem~\ref{thm:mapping-telescopes}.(\ref{item:mapping-tel:inclusion})
and (\ref{item:mapping-tel:map-of-idempotents}).}

The inclusion $\iota_0 \colon A[0] \rightarrow M^{I(0,1)}(f)$ induces a map
$\iota\colon K \rightarrow \Tel(\eta)$.  The commutative diagram from
\ref{thm:mapping-telescopes}.\eqref{item:mapping-tel:map-of-idempotents}
induces a map of diagrams~\eqref{eq:appendix:structure:cyl} and therefore a map
$M^I(\mu) \rightarrow M^I(\eta)$ which yields a map $f_* \colon \Tel(\mu)
\rightarrow \Tel(\eta)$. 

\subsubsection{Homotopy commutative diagrams and
Theorem~\ref{thm:mapping-telescopes}.(\ref{item:mapping-tel:homotopy-id-give-hpty-equiv-tels}).}
\label{subsubsec:appendix:outline:homotopy-commutative-diagrams}
For the next step we need to discuss homotopy commutative diagrams.  Assume that have
a (possibly noncommutative) solid diagram
\begin{equation*}
    \xymatrix{
    A \ar[r]^{f} \ar[d]^{a}   & A \ar[d]^{a}\\
    B \ar[r]^{g}  \ar@{==>}[ur]|(0.5){H}                & B }
  \end{equation*}
where $g\circ a$ is homotopic to $a \circ f$ via a homotopy $H\colon A[I]
\rightarrow B$, indicated in the interior of the diagram above.  We get
a map $M^I(f) \rightarrow B$ and the diagram
\begin{equation}
  \vcenter{\xymatrix{
   A[I] \ar[d] & A[1] \ar[r]\ar[l]^{\iota_0} \ar[d] & M^I(f) \ar[d] \\
   B[I]        & B[1] \ar[r]^{g} \ar[l]             & B }}
\label{eq:induce-map-long-cylinder}
\end{equation}
induces a map $(H,a)_*\colon M^{I'}(f) \rightarrow M^{I}(g)$ on the pushouts
of the rows, where $I'$ is now the \emph{concatenation} of two copies of $I$,
i.e., $I' = \Delta^1 \cup_{\Delta^0} \Delta^1$.

Note $M^{I'}(f)$ is also the
pushout of $A[I'] \leftarrow A[1] \xrightarrow{f} A$.  In the category of
topological spaces $M^{I'}(f)$ and $M^{I}(f)$ would be homeomorphic and we
would obtain a map $M^I(f) \rightarrow M^I(g)$.  For this outline, we pretend
we have this map in our simplicial context.  Hence, the map $(H,a)_*$ gives a
map $(H,a)_* \colon \Tel(f) \rightarrow \Tel(g)$, which depends on the
homotopy $H$.

If $\eta, \mu\colon K \rightarrow K$ are homotopic via a homotopy $H$ from
$\mu$ to $\eta$ we obtain a map $(H,\id)_* \colon \Tel(\eta) \rightarrow
\Tel(\mu)$.  The inverse homotopy $\overline{H}$ from $\mu$ to $\eta$
gives a map $(\overline{H}, \id)_* \colon \Tel(\mu) \rightarrow \Tel(\eta)$.
We can \emph{stack} homotopy commutative squares as shown below.
\begin{equation*}
    \xymatrix{
    K \ar[r]^{\eta} \ar[d]^{\id}   & K \ar[d]^{\id}\\
    K \ar[r]^{\mu} \ar[d]^{\id} \ar@{==>}[ur]|(0.5){H}    & K \ar[d]^{\id} \\
    K \ar[r]^{\eta}  \ar@{==>}[ur]|(0.5){\overline{H}}                & K }
\end{equation*}
We will describe the composition of $(H ,\id)_*$ and $(\overline{H}, \id)_*$
on $\Tel(\eta)$.  We then will develop a criterion which implies that this
composition is homotopic to $\id_{\Tel(\eta)}$.  Therefore, $(H, \id)_* \colon
\Tel(\eta) \rightarrow \Tel(\mu)$ is a homotopy equivalence.

\subsubsection{Theorem~\ref{thm:mapping-telescopes}.(\ref{item:mapping-tel:tel-of-id-hpty-equiv-to-obj})
and (\ref{item:mapping-tel:relative-to-space}).}
In the topological category $\Tel(\id_K)$ would just be $K \times [0,\infty)$,
therefore homotopy equivalent to $K$.  In the simplicial context, after
we define a simplicial ``interval $[0,\infty)$''
(in Section~\ref{subsubsec:interval-infinite}), we similarly have $\Tel(\id_K)
\cong K[ [0,\infty)]$ and the Convergent Homotopy
Lemma~\ref{lem:infinite-homotopies} provides a homotopy equivalence in that
case. 
This uses the Kan Extension Property in $\mathcal{C}^G$ in a crucial way.

We will show that all maps described so far respect the
inclusion $\iota$ from above.

\subsubsection{Theorem~\ref{thm:mapping-telescopes}.(\ref{item:mapping-tel:idempot-induces-id-on-tel}).}
Let $\eta\colon K \rightarrow K$ be a map in $\mathcal{C}^G$.  By ``sliding
down the cylinder'', the top inclusion $\iota_0 \colon K \rightarrow M^I(\eta)$ is
homotopic to the composition $K \xrightarrow{\eta} K \xrightarrow{\iota_1}
M^I(\eta)$.  Consider $M^I(\eta)
\cup_{\iota_1,\iota_0} M^I(\eta)$.  This is two mapping cylinders glued
together, like in the telescope construction.  We have two inclusions $A_0,A_1$
of
$M^I(\eta)$ into $M^I(\eta) \cup_{\iota_1,\iota_0} M^I(\eta)$.  The homotopy
from $\iota_0$ to $\eta\circ\iota_1$
extends to a homotopy from the first inclusion $A_0$ to $A_1\circ
\eta_*$.  This further extends to a
homotopy between $\id_{\Tel(\eta)}$ and a composition of maps $\sh \circ
\eta_{*}$. Here $\sh \colon \Tel(\eta) \rightarrow \Tel(\eta)$ is the
map which maps the component $M^I(f)$ at position $i$ to $M^I(f)$ at position
$i+1$, and $\eta_{*}$ is the map induced by $\eta$ on the telescope.

If $\eta_*$ is a coherent homotopy idempotent, we show that $\eta_* \circ
\eta_* \simeq \eta_*$ as maps $\Tel(\eta) \rightarrow \Tel(\eta)$ and
therefore, $\eta_* \simeq \sh \circ \eta_* \circ \eta_* \simeq \sh \circ \eta_*
\simeq \id_{\Tel(\eta)}$.

\subsubsection{Theorem~\ref{thm:mapping-telescopes}.(\ref{item:mapping-tel:hpty-retract}).}
Because $\eta$ is a homotopy idempotent we have a stacking of homotopy
commutative diagrams
\begin{equation*}
    \xymatrix{
    K \ar[r]^{\eta} \ar[d]^{\eta}   & K \ar[d]^{\eta}\\
    K \ar[r]^{\id} \ar[d]^{\eta} \ar@{==>}[ur]|(0.5){H}    & K \ar[d]^{\eta} \\
    K \ar[r]^{\eta}  \ar@{==>}[ur]|(0.5){\overline{H}}                & K }.
\end{equation*}
We obtain maps $(H,\eta)_* \colon \Tel(\eta) \rightarrow \Tel(\id_K)$
and $(\overline{H},\eta)_* \colon \Tel(\id_K) \rightarrow \Tel(\eta)$.  We
show that the compositions are homotopic to $\eta_*$ on $\Tel(\eta)$, which is
homotopic to $\id_{\Tel(\eta)}$ by
\eqref{item:mapping-tel:idempot-induces-id-on-tel}, and
to $\eta_*$ on $\Tel(\id_K)$.  By
\ref{thm:mapping-telescopes}.\eqref{item:mapping-tel:tel-of-id-hpty-equiv-to-obj},
we have $\Tel(\id_K) \simeq K$. By inspecting the resulting maps closely, we
obtain that $\eta_*$ on $\Tel(\id_K)$ indeed corresponds to $\eta \colon K
\rightarrow K$.

\medskip

We left out quite a few details in the outline above.  The most important
involves concatenation of homotopies, which necessitates 
an analogue of ``Moore homotopies''. These are homotopies where we allow the
``intervals'' to have different lengths and even different directions.
These occur when we investigate the maps induced by homotopy commutative
diagrams, cf.~the
diagram~\ref{subsubsec:appendix:outline:homotopy-commutative-diagrams}.(\ref{eq:induce-map-long-cylinder}).

The main difficulty is that 
for $I' = \Delta^1\cup_{\Delta^0} \Delta^1$ the mapping
cylinders $M^{I'}(f)$ and $M^{I}(f)$ are in general only homotopy equivalent
and not isomorphic.  We later discuss composition of maps and want certain
maps to be equal, not only homotopic.  The solution is to remember the
``interval'' $I'$ (and its larger relatives).  Because we need inverse
homotopies we must allow intervals in different directions.  (See also
Remark~\ref{rem:directions-of-homotopies}.)
\medskip

We will introduce intervals in the category of simplicial sets in the next
section, Section~\ref{subsec:appendix-simplicial-intervals}.  We discuss
the resulting ``long homotopies'' in Section~\ref{subsec:long-homotopies}
and compare intervals of different length in
Section~\ref{subsec:intervals-different-length}.  Then we discuss
homotopies of infinite length and prove the Convergent Homotopy Lemma
in~\ref{subsec:homotopies-infinite-length}.  We define the mapping telescope
for these long homotopies in~\ref{subsec:long-mapping-cylinder-telescope}.
This provides parts~(\ref{item:mapping-tel:inclusion}),
(\ref{item:mapping-tel:map-of-idempotents}) and
(\ref{item:mapping-tel:tel-of-id-hpty-equiv-to-obj}) of
Theorem~\ref{thm:mapping-telescopes}.
Homotopy commutative squares are discussed
in~\ref{subsec:homotopy-commutative-square-mapping-tel}.  We show that they
induce maps on telescopes and give in~\ref{subsec:homotopy-criterion-tel} a
criterion when those are homotopic.  The last three
sections~\ref{subsec:on-thm:id-give-htpy-equiv}
to~\ref{subsec:on-thm-hpty-retract} finally prove parts
(\ref{item:mapping-tel:homotopy-id-give-hpty-equiv-tels}),
(\ref{item:mapping-tel:idempot-induces-id-on-tel}) and
(\ref{item:mapping-tel:hpty-retract}) of Theorem~\ref{thm:mapping-telescopes}.
(\ref{item:mapping-tel:relative-to-space}) is obvious.

\subsection{Simplicial Intervals}
\label{subsec:appendix-simplicial-intervals}

We start by defining what we will mean by an \emph{interval} in the category of
simplicial sets.  (We have not seen our definition in the literature.)  The
name is chosen to stress the analogies to the topological setting.  Our
definition is essentially a nice formal description of simplicial set of the
form $\rightarrow \leftarrow$ etc.

\begin{definition}[Simplicial intervals]\label{def:interval}\  
  \begin{enumerate}
    \item Let $i\in \Nat$. A one-point simplicial set $I(i)$, $I(i)_k =
      \{i\}$, together with a bijection $l\colon I(i)_0 \rightarrow \{i\}$
      from its zero simplices is a called  a \emph{point at $i$} or
      \emph{interval of length $0$ from $i$ to $i$}.

    \item An \emph{interval of length $1$ from $i$ to $(i+1)$}, denoted
      $I(i,i+1)$, is a simplicial set isomorphic to $\Delta^1$ together with a
      bijection of its zero simplices to the set $\{i, i+1\}$, $l\colon
      I(i,i+1)_0 \rightarrow \{i,i+1\}$ .  The map $l$ is called the
      \emph{labeling}.

    \item Let $i,j\in \Nat$, $i + 2 \leq j$.  An \emph{interval of length
      $(i-j)$ from $i$ to $j$} is a simplicial set $I(i,j)$ together with a
      bijection $l\colon I(i,j)_0\rightarrow \{i, i+1,\dots, j\}$ such that
      there is a pushout diagram
      \begin{equation}\label{diag:pushout-of-intervals}
        \vcenter{\xymatrix{ I(j-1) \ar@{ >->}[r]\ar@{ >->}[d]  &
        I(j-1,j) \ar[d] \\ I(i,j-1) \ar[r] & I(i,j) }}
      \end{equation}
      where the maps are compatible with the labelings and $I(j-1)
      \rightarrowtail I(j-1, i)$, $I(j-1) \rightarrowtail I(i, j-1)$ are the
      obvious inclusions.

    \item The \emph{standard interval from $i$ to $(i+1)$} is the simplicial
      set $\Delta^1$ together with the labeling $l(0) = i$, $l(1)=i+1$.  The
      \emph{standard interval from $i$ to $j$} for $i + 2 \leq j$ is the
      simplicial set arising from the standard interval from $i$ to $j-1$ by
      the pushout~\eqref{diag:pushout-of-intervals} with $I(j-1,j)$ being the
      standard interval of length $1$.

    \item An interval $I(i,j)$ from $i$ to $j$ is called \emph{ordered} if it
      is isomorphic to the standard interval from $i$ to $j$ and the
      isomorphism respects the labeling.
  \end{enumerate}
\end{definition}

\begin{rem}\label{rem:intervals-notation}
  We sometimes draw pictures for intervals. The standard interval is $0
  \rightarrow 1$.  The four intervals for $I(0,2)$ are the following ones: 
  \begin{equation*}
    \begin{array}{c@{\qquad\qquad}c}
      0\rightarrow 1 \rightarrow 2 & 0 \rightarrow 1 \leftarrow 2 \\
      0\leftarrow  1 \rightarrow 2 & 0 \leftarrow  1 \leftarrow 2.
    \end{array}
  \end{equation*}
  The notion $I(i,j)$ is ambiguous, as we want it to cover all possible
  orderings.
  \end{rem}

\subsubsection{Concise notation}
We often just write $I(i,j)$ for an interval from $i$ to $j$ leaving all the
other data understood.  For $A\in \mathcal{C}^G$ we also often abbreviate
$A[I(i,j)]$ as $A[i,j]$ and $A[I(i)]$ as $A[i]$, slightly misusing
notation.

\subsubsection{The infinite interval}\label{subsubsec:interval-infinite}
  Define a simplicial set $I(i,\infty)$ to be an \emph{interval from $i$ to
  $\infty$} if it is the filtered colimit (or union) of intervals $I(i,j)$
  for $j \rightarrow \infty$.  It is called \emph{ordered} if each of the
  $I(i,j)$ is.

\subsection{Long Homotopies}
\label{subsec:long-homotopies}
Our notion of interval gives rise to a
notion of homotopy.
\begin{definition}[Long Homotopy]
  Let $I(0,j)$ be an interval from $0$ to $j$. Let $f_0, f_j \colon A
  \rightarrow B$ be two maps in $\mathcal{C}^G$.
  A \emph{long homotopy} from $f_0$ to $f_j$ is a map $H\colon A[I(0,j)]
  \rightarrow B$ such that the restriction to $A[0]$ is $f_0$ and the
  restriction to $A[j]$ is $f_j$.  We say that $H$ has \emph{length} $j$.
\end{definition}

\begin{example}\label{ex:trivial-long-homotopy}
  If $f\colon A \rightarrow B$ is a map in $\mathcal{C}^G$ and $I(0,i)$ any
  interval we always have the \emph{trivial homotopy}
  $\Tr\colon A[0,i] \rightarrow B$ induced by the map $A[0,i] \rightarrow A
  \rightarrow B$.  We also define it for $i=0$ and therefore call the map
  $\Tr\colon A[0,0] = A[0] = A \xrightarrow{f} B$ the \emph{trivial homotopy
  of length $0$}.
\end{example}

\subsubsection{Ordinary and long homotopies}
Every homotopy in the usual sense is a long homotopy of length $1$.  Every
long homotopy gives a homotopy in the usual sense by the Kan property.  This
is not functorial, which is the reason why we need to consider long
homotopies.  We will omit the ``long'' in the following.

\subsubsection{Concatenation of intervals}
If $I(0,i)$ and $I(0,j)$ are intervals we define the concatenation
$I(0,i)\mathop{\Box} I(0,j)$ to be the pushout
\begin{equation*}
  \xymatrix{
    I(i) \ar[r] \ar[d]  & I(i, i+j) \ar[d] \\
    I(0,i) \ar[r]  & I(0,i)\mathop{\Box} I(0,j)}
\end{equation*}
where $I(i, i+j)$ is defined as a relabeling of $I(0,j)$, replace
the labeling $l$ of $I(0,j)$ by $l(k) = i+k$.

\subsubsection{Concatenation of homotopies}\label{sec:concat-of-homotopies}
Homotopies which agree on the start resp.\ endpoint can be concatenated. For
$H_1 \colon A[0,i] \rightarrow B$, $H_2\colon A[0,j] \rightarrow B$ with
${H_1}_{|A[i]} = {H_2}_{|A[0]}$ define the \emph{concatenation}
\begin{equation*}
  H_1 \mathop{\Box} H_2 \colon A[0,i+j] \rightarrow B
\end{equation*}
as the map induced by the identification on the pushout $I(0,i) \mathop{\Box}
I(0,j)$. The concatenation of homotopies is strictly associative.

\subsubsection{Inverse homotopies}
If $I(0,j)$ is an interval, define the \emph{reversed interval}
$\overline{I(0,j)}$ as the same simplicial set with the labeling
$l$ replaced by $\overline{l}(k) := j - l(k)$.
If $H\colon A[I(0,j)] \rightarrow B$ is a homotopy the
\emph{inverse homotopy} $\overline{H}$ is the
obvious map $\overline{H}\colon A[\overline{I(0,j)}] \rightarrow B$.

If $j=1$ and $I(0,j)$ is an ordered interval we draw
the homotopy as $\xymatrix{ \ar[r]|H & }$ and the inverse homotopy as
$\xymatrix{&\ar[l]|H }$. 

\begin{lem}[Concating a homotopy and its inverse]\label{lem:concat-homotopy-and-inverse}
  Let $H\colon A[0,1]  \rightarrow B$ be a homotopy.
  The concatenation $H \mathop{\Box} \overline{H}$  is homotopic, relative
  boundary, to the trivial homotopy $\Tr\colon A[0,2]
  \rightarrow A \rightarrow B$.
\end{lem}

\begin{proof}
  Assume that $I(0,1)$ is the standard interval, the other case proceeds
  similarly.
  The homotopy $A[0,2][\Delta^1] \rightarrow B$ is given by the left diagram
  below.  It is constructed by gluing the $2$-simplices together which are
  shown on the right.  These arise from the $2$nd degeneracy map $\Delta^2
  \rightarrow \Delta^1$.
  
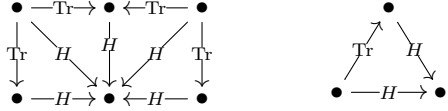
\begin{equation*}
    \xymatrix{
    \bullet \ar[d]|{\Tr} \ar[r]|{\Tr}  \ar[dr]|{H}  & 
    \bullet \ar[d]|(0.4){H}  &
    \bullet \ar[d]|{\Tr} \ar[l]|{\Tr} \ar[dl]|{H}  \\
    \bullet \ar[r]|{H}  & \bullet  & 
    \bullet \ar[l]|{H} 
    }\qquad \qquad
    \xymatrix@R=5ex@C=2ex{
      & \bullet \ar[dr]|H & \\
      \bullet \ar[ur]|{\Tr} \ar[rr]|H && \bullet }
  \end{equation*}\qedhere 
\end{proof}

\subsubsection{Variations of the previous lemma}
\label{subsubsec:variations-longer-homotopies}
In the proof of the lemma we gave a homotopy from the trivial homotopy to the
given one.  We can give one in the other direction by a similar proof where we
use a map $A[\Delta^2] \rightarrow B$ arising from the Kan extension property.

The other concatenation, $\overline{H}\mathop{\Box} H$, is also homotopic to
the trivial homotopy.  Furthermore, the lemma still holds if we allow an interval of length
$n$ instead of length $1$. To prove this, one does induction over $n$ and
starts building the homotopy from the middle, filling outer parts in each
induction step with some of four trivial homotopies of the remaining
homotopies $H'$ of length $1$, like
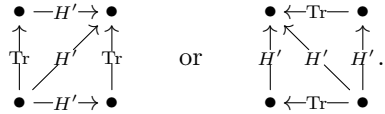
\begin{equation*}
  \vcenter{
  \xymatrix{
    \bullet \ar[r]|{H'}  &
    \bullet  \\
    \bullet \ar[u]|{\Tr} \ar[ur]|{H'} \ar[r]|{H'}  &
    \bullet \ar[u]|{\Tr}  }} \qquad \text{or} \qquad
  \vcenter{
  \xymatrix{
    \bullet  &
    \bullet \ar[l]|{\Tr} \\
    \bullet \ar[u]|{H'}  &
    \bullet \ar[u]|{H'} \ar[l]|{\Tr} \ar[ul]|{H'}  } }.
\end{equation*}
These techniques will work in the more complicated situations later, hence we
will only draw the diagrams for length $1$ homotopies in the
following.\smallskip

\begin{rem}\label{rem:directions-of-homotopies}
  Lemma~\ref{lem:concat-homotopy-and-inverse} illustrates why we need
  intervals in
  different directions.  It might be possible to avoid these in all the
  following proofs by choosing $\overline{H}$ to be the filling of the horn
  
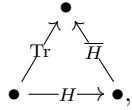
\begin{equation*}
    \xymatrix@R=5ex@C=2ex{
        & \bullet & \\
        \bullet \ar[ur]|{\Tr} \ar[rr]|H && \bullet
        \ar[ul]|{\overline{H}} ,}
  \end{equation*}
  but then we would at least have to remember not only the choice of
  $\overline{H}$, but the whole filling. Also, several of the later diagrams
  would lose their symmetry.
\end{rem}

\subsection{Comparing intervals of different length}
\label{subsec:intervals-different-length}

We now show that
for $A\in\mathcal{C}^G$ the modules $A[0,i]$, ($i \in \Nat$), are homotopy
equivalent to $A$ and the homotopies can be chosen to be relative to the
endpoints.  The result will follow from the next lemma.

\begin{lem}\label{lem:sides-are-homotopy-retracts}
  Let $\Lambda^2_i$ be the $i$th horn and $d_i$ the $i$th face of $\Delta^2$.
  Then $A[d_i]$ and $A[\Lambda^2_i]$ are homotopy equivalent relative the
  $0$-simplices of $d_i$.  The homotopy equivalence can be chosen to be one of
  the maps $A[\Lambda^2_i]\rightarrow A[d_i]$ which induced by collapsing one
  $1$-simplex.
\end{lem}

\begin{proof}
  Consider the composition
  \begin{equation*}
    A[\Lambda^2_i] \rightarrow A[\Delta^2] \rightarrow A[d_i]
  \end{equation*}
  where the first map is the inclusion.  The last map, and hence the
  composition, can be induced by any map collapsing a $1$-simplex not
  equal to $d_1$ or $d_i$.  It is not hard to see that the first map has a
  deformation retraction by horn-filling.  The second map is induced by a
  deformation retraction of simplicial sets.  Therefore, the composition is a
  homotopy equivalence relative to the $0$-simplices of $d_i$.
\end{proof}

\begin{cor} \label{cor:long-cylinders-short-equivalent}
  Let $I(0,1)$ be the standard interval and $I(0,i)$ any interval.  Then we
  have a homotopy equivalence relative endpoints
  \begin{equation*}
    A[0,1] \simeq A[0,i].
  \end{equation*}
\end{cor}

\begin{proof}
  Lemma \ref{lem:sides-are-homotopy-retracts} implies $A[0,2]\simeq
  A[0,1]$ relative endpoints if there is a projection $I(0,2) \rightarrow
  I(0,1)$. It also implies $A[\rightarrow] \simeq A[\leftarrow]$ relative
  endpoints by the chain $A[\rightarrow] \simeq A[\rightarrow \leftarrow] \simeq
  A[\leftarrow]$ of homotopy equivalences relative endpoints.  The corollary
  follows by induction.
\end{proof}

\subsection{Homotopies of infinite length and the Convergent Homotopy Lemma}
\label{subsec:homotopies-infinite-length}

In the following, we assume for simplicity that the infinite interval
$I(0,\infty)$ is ordered (cf.~\ref{subsubsec:interval-infinite}).  We
abbreviate $A[I(0,\infty)]$, $A \in \mathcal{C}^G$, as $A[0,\infty)$.  We
suggestively call a map $A[0,\infty) \rightarrow B$ a homotopy \emph{of
infinite length}.  From such a homotopy we want to obtain a homotopy of length
$1$.   In general, this is impossible. But if the homotopy is
\emph{convergent} in the sense we define below, this can be done.

\begin{lem}[Convergent Homotopy Lemma]\label{lem:infinite-homotopies}
  Let $H \colon A[0,\infty) \rightarrow B$ be a \emph{convergent homotopy}, this
  means we assume:
  \begin{enumerate}
    \item There is a filtration $A_0 \subseteq A_1 \subseteq \cdots
      \subseteq A_n\subseteq \cdots \subseteq A$ by cellular submodules such
      that $\bigcup_i A_i = A$.
    \item For each $A_i$ there is an $n_i$ such that $H_{|A_i[n_i,\infty)}$ is
        the trivial homotopy $\Tr$ (cf.~\ref{ex:trivial-long-homotopy}).
  \end{enumerate}
  Then there exists a homotopy $G\colon
  A[\Delta^1] \rightarrow B$ with $G_{|A[0]} = H_{|A[0]}$ and
  $G_{|A_i[1]} = H_{|A_i[n_i]}$.
\end{lem}

\begin{rem}
  Recall that $A_i[n_i]$ and $A_i[n_i,\infty)$ denote obvious cellular
  submodules of $A[0,\infty)$.  We can and will assume in the proof that
  $n_{i+1} \geq n_i$.

  This lemma is well-known in the topological case. $G$ may be
  called the limit of the homotopy $H$.
\end{rem}

\begin{proof}[Proof of Lemma~\ref{lem:infinite-homotopies}]
  We enlarge $I(0,\infty)$ to a new simplicial set
  $\widehat{I(0,\infty)}$ by filling some horns.

  The subsimplicial set
  $I(0,2)$ is isomorphic to the horn $\Lambda^2_1$.  We consider the pushout of
  $\Delta^2 \leftarrowtail \Lambda^2_1 \rightarrowtail I(0,\infty)$ and
  call it $\widehat{I(0,2)}$.  It has an extra $1$-simplex with boundaries
  $I(0)$ and $I(2)$ in $I(0,\infty)$, which we call $(0\rightarrow 2)$.

  In $\widehat{I(0,2)}$ the $1$-simplices $(0,\rightarrow 2)$ and $I(2,3)$
  constitute a horn $\Lambda^2_1$.  Like before we define
  $\widehat{I(0,3)}$ to be the following pushout.
  \begin{equation*}
    \xymatrix{ \Lambda^2_1 \ar@{ >->}[r] \ar@{ >->}[d]  & \ar[d] \Delta^2 \\
    \widehat{I(0,2)} \ar[r]   &  \widehat{I(0,3)} }
  \end{equation*}
  Again it has an extra $1$-simplex with boundaries
  $I(0)$ and $I(3)$ in $I(0,\infty)$, which we call $(0\rightarrow 3)$.  Now
  we proceed by induction and define $\widehat{I(0,\infty)}$ as the
  filtered colimit $\widehat{I(0,\infty)} := \bigcup_n \widehat{I(0,n)}$. Figure
  \ref{fig:thick-line} sketches a picture of $ \widehat{I(0,\infty)}$ with
  $I(0,\infty)$ being the bottom line.\medskip
  \begin{figure}[t]
    \centering
%
\begin{tikzpicture}[very thin,
  side/.style={fill opacity=0.8,fill=white,draw=black},
  middle side/.style={side,fill=white!85!black},
  arr/.style={->,shorten >=2pt,shorten <=2pt},
  ]

  \draw[arr] (0,0) to (1,0);
  \draw[arr] (1,0) to (2,0);
  \draw[arr] (2,0) to (3,0);
  \draw[arr] (3,0) to (4,0);
  \draw[arr] (4,0) to (5,0);
  \node at (5,0) [right] {\dots};
  \draw[arr] (0,0) to [out=45,in=135] (2,0);
  \draw[arr] (0,0) to [out=45,in=135] (3,0);
  \draw[arr] (0,0) to [out=45,in=135] (4,0);
  \draw[arr] (0,0) to [out=45,in=135] (5,0);
\end{tikzpicture}
    \caption{A sketch of $\widehat{I(0,\infty)}$.}
    \label{fig:thick-line}
  \end{figure}
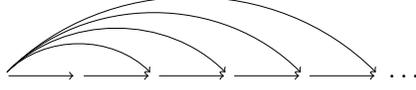
  
  We have $A\widehat{[0,\infty)} = \bigcup_n A\widehat{[0,n]}$, with
  $A\widehat{[0,n]}$ and $A\widehat{[0,\infty)}$ being abbreviations for
  $A[\widehat{I(0,n)}]$ and $A[\widehat{I(0,\infty)}]$, respectively.  Note
  that $A\widehat{[0,n]}$ arises from $A\widehat{[0,n-1]}$ by horn-filling. 
  We want to construct a certain map $\widehat{H} \colon A\widehat{[0,\infty)}
  \rightarrow B$ which extends $H$.

  We do induction over $i$.  Assume that we have constructed a homotopy $G_i
  \colon A_i\widehat{[0, \infty)} \rightarrow B $ which extends $H\colon
  A_i[0,\infty) \rightarrow B$ and has the property that
  ${G_i}_{|A_i[(0\rightarrow n)]} = {G_i}_{|A_i[(0\rightarrow n_i)]}$ for all $n
  \geq n_i$.  By iterating the relative horn-filling
  property~\ref{lem:horn-filling-relative}, we can extend $G_i$ to
  a map $A_{i+1}\widehat{[0,n_{i+1}]} \cup
  A_{i}\widehat{[0,\infty)}\rightarrow B$.
  
  For $n \geq n_{i+1}$, we do not want to apply the relative horn-filling
  property as we need special fillings.  By assumption, $H_{|A_{i+1}[n,n+1]}$
  is the trivial homotopy $\Tr$.  Hence, the relative horn spanned by
  $A_{i+1}[(0\rightarrow
  n)]$ and $A_{i+1}[n,n+1]$, given by $ A_{i+1}[\Lambda^2_1] \rightarrow
  A_{i+1}\widehat{[0,n]} \cup A_{i}\widehat{[0,\infty)}$, can be filled in the
  following way
  \begin{equation}\label{diag:fill-horn-constant}
   \vcenter{ \xymatrix { \ar@/^5ex/[rrr]|X  \ar@/^2ex/[rr]|X && \ar[r]|{\Tr} &
 }},
  \end{equation}
  where $X$ is the homotopy coming from the previous horn-fillings.  This
  defines a map $G_{i+1}\colon A_{i+1}\widehat{[0,\infty)}$ with
  ${G_{i+1}}_{|A[(0\rightarrow n)]} = {G_{i+1}}_{|A[(0 \rightarrow
  n_{i+1})]}$ for all $n \geq n_{i+1}$.  This shows the induction step.

  Taking the colimit over $G_i$ we obtain a map $\widehat{H}\colon
  A\widehat{[0,\infty)} \rightarrow B$.  We now define $G_{|A_j}\colon
  A_j[\Delta^1] \rightarrow B$ as the restriction of $\widehat{H}$ (or
  equivalently $G_j$) to $A_j[(0 \rightarrow n_j)]$, i.e.,~to the $1$-simplex
  from $0$ to $n_j$.  This is compatible with the inclusion $A_j \rightarrow
  A_{j+1}$ and thus the colimit over $j$ gives the desired homotopy $G\colon
  A[\Delta^1] \rightarrow B$.
\end{proof}

\begin{cor}\label{cor:half-interval-hpty-to-space}
  The map $i\colon A \rightarrow A[0,\infty)$ has a deformation retraction, so in
  particular $i$ is a homotopy equivalence.
\end{cor}

\begin{proof} 
  The map $[0,\infty) \rightarrow 0$ induces a retraction $r\colon A[0,\infty)
  \rightarrow A$ for $i$. We have to prove that the composition $i\circ r
  \colon A[0,\infty) \rightarrow A\rightarrow A[0,\infty)$ is homotopic to the
  identity. We use the Convergent Homotopy Lemma~\ref{lem:infinite-homotopies}.
  Define the convergent homotopy $H \colon A[0,\infty)[0,\infty) \rightarrow
  A[0,\infty)$ as the map induced by the map
  \begin{equation*}
    (i,j) \mapsto \min(i,j)
  \end{equation*}
  where we use that map $[0,\infty)\times [0,\infty) \rightarrow [0,\infty)$
  is determined on the $0$-simplices.  We regard $j$ as the homotopy
  direction. This map has the following properties:
  \begin{enumerate}
    \item For $j=0$ it is the projection to $0$, hence the map $i\circ r$.
    \item For any $j \geq i$ the map $A[0,i][j] \rightarrow A[0,i]$ is the
      identity.
    \item $\bigcup_i A[0,i]$ is a filtration of
      $A[0,\infty)$ by cellular modules.
  \end{enumerate}
  Now the Convergent Homotopy Lemma~\ref{lem:infinite-homotopies} applies and
  hence we obtain a homotopy $G\colon A[0,\infty)[\Delta^1] \rightarrow
  A[0,\infty)$ from $i\circ r$ to the identity.
\end{proof}

\subsection{The long mapping cylinder and the telescope}
\label{subsec:long-mapping-cylinder-telescope}
In the following, $I =I(0,i)$ is always an interval and $f\colon A \rightarrow
A$ a map in $\mathcal{C}^G$.
Define the mapping cylinder $M^I(f)$ for $I$ of $f$ like in
\ref{subsubsec:appendix-definition-telescope}.\eqref{eq:appendix:structure:cyl}
We obtain the
front and back inclusion $\iota_0, \iota_1 \colon A \rightarrow M^I(f)$.  We
have $I$ in the notation because we need to keep track of it in the following.
The definition is slightly different from the definition of the cylinder
functor in~\ref{subsec:cylinder-functor} because the mapping cylinder will
play a slightly different role here.

\begin{definition}[Long Mapping Telescope]\label{def:map-tel:general}
  Let $I$ be an interval in the sense of Definition~\ref{def:interval}.
  Define $\Tel^I(f)$, the \emph{mapping telescope of $f\colon
  A \rightarrow A$ for the interval $I$} is the following pushout.
  \begin{equation*}
    \xymatrix{
    \coprod_{i=1}^{\infty} A[i] \amalg \coprod_{j=0}^{\infty} A[j+1]
    \ar@{ >->}[r]^-{\iota_0 \amalg \iota_1} \ar[d]^c &
    \coprod_{i=0}^{\infty} M^{I}(f) \ar[d] \\
    \coprod_{i=1}^\infty A[i] \ar[r]  & \Tel^I(f)}
  \end{equation*}
  The map $\iota_0 \amalg \iota_1$ sends the summand $A[i]$ into the
  $i$th copy of $M^{I}(\eta)$ via $\iota_0$ and the summand $A[j+1]$
  into the $i$th copy of $M^{I}(\eta)$ via $\iota_1$. The map $c$ sends $A[l]$
  to $A[l]$. 
\end{definition}

\subsubsection{Front inclusion}\label{subsubsec:telescope-front-inclusion}
The front inclusion into the first mapping cylinder $\iota_0\colon
A\rightarrow M^{I}(f)$
(which is not used in the diagram above) gives a map
\begin{equation*}
  \iota \colon A \rightarrow \Tel^I(f)	
\end{equation*}
which is a called the \emph{front inclusion} of the mapping cylinder.

\begin{rem}
  The telescope consists of infinitely many mapping cylinders
  glued together on the right.  Each mapping cylinder has the same
  interval structure.  To construct $\Tel(f)$ we used that countable
  coproducts exist in~$\mathcal{C}^G$.
\end{rem}

\begin{lem}\label{lem:appendix:telescope-functorial}
  $f\mapsto \Tel^I(f)$ is functorial, that is a commutative diagram
  \begin{equation*}
    \xymatrix{ A \ar[r]^{f} \ar[d]^a & A \ar[d]^a \\
      K\ar[r]^{g} & K}
  \end{equation*}
  induces a map $a_* \colon \Tel(f) \rightarrow \Tel(g)$.\qed
\end{lem}

\begin{rem}
  As $\Tel(f) = \Tel^{\Delta^1}(f)$
  (cf.~\ref{subsubsec:appendix-definition-telescope}),
  Theorem~\ref{thm:mapping-telescopes}.\eqref{item:mapping-tel:inclusion} and
  \eqref{item:mapping-tel:map-of-idempotents}  are now obvious.  Furthermore,
  $\Tel(\id_K) \cong K[0,\infty)$, so
  Theorem~\ref{thm:mapping-telescopes}.\eqref{item:mapping-tel:tel-of-id-hpty-equiv-to-obj}
  follows by Corollary~\ref{cor:half-interval-hpty-to-space}.
\end{rem}

Using Corollaries~\ref{cor:long-cylinders-short-equivalent} and
\ref{cor:half-interval-hpty-to-space} we obtain the following lemma.

\begin{lem}\label{lem:mapping-cylinders-different-length-equivalent}
  \label{lem:tel-homotopy-equiv-to-obj}Recall $\Delta^1$ is the standard
  interval. 
  \begin{enumerate}  
    \item The mapping cylinders for $I$ and $\Delta^1$ are homotopy
      equivalent relative to the front and the back inclusion:
      $M^I(f) \simeq M^{\Delta^1}(f) = T(f)$.
    \item The mapping telescopes for $I$ and $\Delta^1$ are homotopy
      equivalent: $\Tel^I(f) \simeq \Tel^{\Delta^1}(f)$.
  \end{enumerate}
  Each map $I \rightarrow \Delta^1$ respecting the endpoints can be chosen to
  induce the homotopy equivalences. \qed
\end{lem}

\begin{rem}
  While we can give the homotopy equivalences quite explicitly, the inverse is
  not canonical and not easy to write down since we used the Kan Extension
  property to construct it. 
\end{rem}

\subsection{Homotopy commutative squares and induced maps on telescopes}
\label{subsec:homotopy-commutative-square-mapping-tel}
To prove the rest of Theorem~\ref{thm:mapping-telescopes} we need to know more
about homotopy commutative diagrams.  They will induce a map of (long) mapping
telescopes, but it is only ``functorial'', in some sense we make precise, if
we allow to change the intervals.

\begin{definition}[Homotopy commutative square]\label{def:homotopy-comm-diag}
  A square in $\mathcal{C}^G$
  \begin{equation}\label{diag:hom-comm}
    \xymatrix{
    A \ar[r]^{f} \ar[d]^{a}   & A \ar[d]^{a} \\
    B \ar[r]^{g}                & B }
  \end{equation}
  is \emph{homotopy commutative} if there is an interval $I=I(0,i)$ and a
  \emph{specified homotopy} $H^a
  \colon A[0,i] \rightarrow B$ which goes from $g\circ a$ to $a\circ f$. This
  should mean ${H^a}_{|A[0]} = g \circ a$ and ${H^a}_{|A[i]} = a \circ f$.
\end{definition}

\begin{rem}
  The homotopy of a homotopy commutative square always goes from the lower
  left corner to the upper right, it is helpful to visualize this as the
  diagram below when thinking about the homotopies. 
  \begin{equation*}
    \xymatrix{
    A \ar[r]^{f} \ar[d]^{a}   & A \ar[d]^{a}\\
    B \ar[r]^{g}  \ar@{=>}[ur]|(0.5){H}                & B }
  \end{equation*}
  We chose the direction of the homotopy such that it will fit together with
  our definition of mapping cylinder.
\end{rem}

The next observation is central for the rest of the proof.

\begin{lemdef}[stacking squares]\label{lemdef:stacking-htpy-comm-squares}
  We can \emph{stack} homotopy commutative squares. Given two homotopy
  commutative squares
  \begin{equation}\label{diag:two-homotopy-comm}
    \xymatrix{
    A \ar[r]^{f} \ar[d]^{a}   & A \ar[d]^{a} \\
    B \ar[r]^{g}                & B }, \qquad
    \xymatrix{
    B \ar[r]^{g} \ar[d]^{b}   & B \ar[d]^{b} \\
    C \ar[r]^{h}                & C }
  \end{equation}
  with homotopies $H^a, H^b$ using intervals $I^a$, $I^b$. Then the 
  \emph{stacked} square
  \begin{equation*}
    \xymatrix{
    A \ar[r]^{f} \ar[d]^{b\circ a}   & A \ar[d]^{b \circ a} \\
    C \ar[r]^{h}                & C }
  \end{equation*}
  is homotopy commutative with homotopy
  \begin{equation*}\label{eq:comp-hom-comm-diag}
    (H^b \circ a[I^b]) \mathop{\Box} (b \circ H^a) \colon A[I^b \mathop{\Box} I^a]
    \rightarrow C. 
  \end{equation*}
\end{lemdef}

\begin{proof}
  The given homotopy is a homotopy from $h\circ b\circ
  a$ to $b\circ a\circ f$.
\end{proof}

\subsubsection{Stacking is associative}
Stacking of homotopy commutative squares is strictly
associative, because concatenation of homotopies is. The length of the
homotopies add.  We will consider a strictly commutative square as a homotopy
commutative square where the homotopy has length $0$.

We only stack in one direction of the two possible directions
in which the squares could be stacked because we do not need the other case.

\begin{lemdef}[Homotopy commutative squares and mapping cylinders]\label{lemdef:homotopy-commutative-to-mapping-cylinders}
  Let $I$ be an interval. A homotopy commutative square
  \begin{equation}\label{diag:hpty-comm-square-1}
    \vcenter{\xymatrix{
    A \ar[r]^{f} \ar[d]^{a}   & A \ar[d]^{a} \\
    B \ar[r]^{g}                & B }}
  \end{equation}
  with homotopy $H\colon A[I] \rightarrow B$ induces a map 
  $(H,a)_*\colon M^I(f) \rightarrow B$ such that the following diagram
  commutes strictly.  
  \begin{equation}\label{diag:hpty-comm-map-cyl}
    \vcenter{\xymatrix{
    A \ar[r]^-{\iota_0} \ar[d]^a  & M^I(f) \ar[d]  & A \ar[l]_-{\iota_1} \ar[d]^a \\
    B \ar[r]^g                    & B              & B \ar[l]_{\id_B}}}
  \end{equation}
  Here $\iota_0$ is the front and $\iota_1$ the back
  inclusion.  Each such diagram determines uniquely the homotopy
  of~\eqref{diag:hpty-comm-square-1}.
\end{lemdef}

\begin{proof}
  The pushout of the strictly commutative diagram
  \begin{equation*}
    \xymatrix{
    A[I] \ar[d]^H & A[i] \ar[l]_{\iota_1} \ar[d]^{a\circ f} \ar[r]^f & A \ar[d]^a \\
    B             & B \ar[r]^{\id_B} \ar[l]_{\id_B} & B}
  \end{equation*}
  gives the map $(H,a)_*\colon M^I(f) \rightarrow B$ and then
  diagram~\eqref{diag:hpty-comm-map-cyl} commutes.  Conversely taking the map
  $A[I] \rightarrow M^I(f) \rightarrow B$ gets back the homotopy $H$ and the
  commutativity of~\eqref{diag:hpty-comm-map-cyl} shows that $H$ makes the
  square~\eqref{diag:hpty-comm-square-1} homotopy commutative.
\end{proof}

\subsubsection{Induced map on longer mapping cylinders}
Let $J$ be another interval. 
Consider the mapping cylinders of $\iota_0$ and $g$
from~\eqref{diag:hpty-comm-map-cyl}.
We obtain a map $a[J]
\mathop{\Box} (H,a)_* \colon M^{J\mathop{\Box} I}(f) \rightarrow M^{J}(g)$
such that the following diagram commutes.
\begin{equation}\label{diag:hpty-comm-long-map-cyl}
  \vcenter{\xymatrix{
  A \ar[r]^-{\iota_0} \ar[d]^a  & M^{J\mathop{\Box}I}(f) \ar[d]  & A \ar[l]_-{\iota_1} \ar[d]^a \\
  B \ar[r]^-{\iota_0}           & M^{J}(g)                       & B
  \ar[l]_-{\iota_1}   }}
\end{equation}
We call the $a[J] \mathop{\Box} (H,a)_*$ the \emph{cylinder map with respect
to $J$} of the homotopy commutative diagram~\eqref{diag:hpty-comm-square-1}.

\begin{definition}[Composition]\label{def:composition-map-of-cylinders}
  Given maps $f\colon A \rightarrow A$, $g\colon B \rightarrow B$ and $h\colon
  C \rightarrow C$ as well as $a\colon A \rightarrow B$ and $b\colon B
  \rightarrow C$.
  Assume we have cylinder maps $(H^a,a)_*\colon M^{I^a}(f) \rightarrow B$ and
  $(H^b,b)_*\colon M^{I^b}(g) \rightarrow C$ like in
  Definition~\ref{lemdef:homotopy-commutative-to-mapping-cylinders} satisfying
  diagrams like~\eqref{diag:hpty-comm-map-cyl}.
  Define the \emph{composition} $(H^a,a)_* \boxcircle (H^b,b)_*$ as
  \begin{equation*}
    M^{I^b\mathop{\Box} I^a}(f) \xrightarrow{a[I^b] \mathop{\Box} (H^a,a)_*}
    M^{I^b} (g) \xrightarrow{(H^b,b)_*}
    C
  \end{equation*}
  More generally, let $J$ be another interval.  Assume we have cylinder maps
  with respect to $J$:
  \begin{equation*}
    a[J] \mathop{\Box} (H^a,a)_*\colon M^{J\mathop{\Box}I^a}(f) \rightarrow M^{J}(g)
    \ \text{ and } \ 
    b[J]\mathop{\Box}(H^b,b)_*\colon M^{J\mathop{\Box} I^b}(g) \rightarrow
    M^J(h)
  \end{equation*}
  Define the \emph{composition with respect to $J$} as
  \begin{multline*} 
    \Bigl(a[J]\mathop{\Box} (H^a,a)_*\Bigr)
     \boxcircle 
    \Bigl(b[J]\mathop{\Box} (H^b,b)_*\Bigr) \colon
    \\
    M^{J\mathop{\Box} I^b\mathop{\Box} I^a}(f) 
      \xrightarrow{a[J]\mathop{\Box} a[I^b] \mathop{\Box} (H^a,a)_*}
    M^{J\mathop{\Box} I^b} (g) \xrightarrow{b[J]\mathop{\Box}(H^b,b)_*}
    M^J(h)
  \end{multline*}
\end{definition}

\begin{lem}\label{lem:map-cyl:composition}
  Given two homotopy commutative squares
  \begin{equation}
    \label{eq:homotopy-commutative-square}
    \xymatrix{
    A \ar[r]^{f} \ar[d]^{a}   & A \ar[d]^{a} \\
    B \ar[r]^{g}                & B }, \qquad
    \xymatrix{
    B \ar[r]^{g} \ar[d]^{b}   & B \ar[d]^{b} \\
    C \ar[r]^{h}                & C }
  \end{equation}
  with homotopies $H^a\colon A[I^a] \rightarrow B$, $H^b\colon B[I^b]
  \rightarrow C$.  Then the cylinder map of the stacked homotopy
  commutative square (cf.~\ref{lemdef:stacking-htpy-comm-squares}) 
  \begin{equation*}
    \xymatrix{
    A \ar[r]^{f} \ar[d]^{b\circ a}   & A \ar[d]^{b \circ a} \\
    C \ar[r]^{h}                & C }
  \end{equation*}
  is equal to the composition of the cylinder maps of the individual
  squares, i.e.,
  \begin{equation*}
    \left((H^b \circ a[I^b]) \mathop{\Box} (b\circ H^a), b\circ a\right)_* =
     (H^b, b)_* \circ \Bigl(a[I^b] \mathop{\Box} (H^a,a)_*\Bigr)
  \end{equation*}
  The same is true for cylinder maps with respect to $J$.
\end{lem}

\begin{proof}
  We have to check the equality of two maps $M^{J \mathop{\Box} I^b
  \mathop{\Box} I^a}(f) \rightarrow M^J(h)$.  Figure~\ref{fig:map-cylc} shows
  the situation.  With its help for the bookkeeping the equality can be
  checked directly.
\end{proof}

\begin{figure}[t]
  \centering
  \input{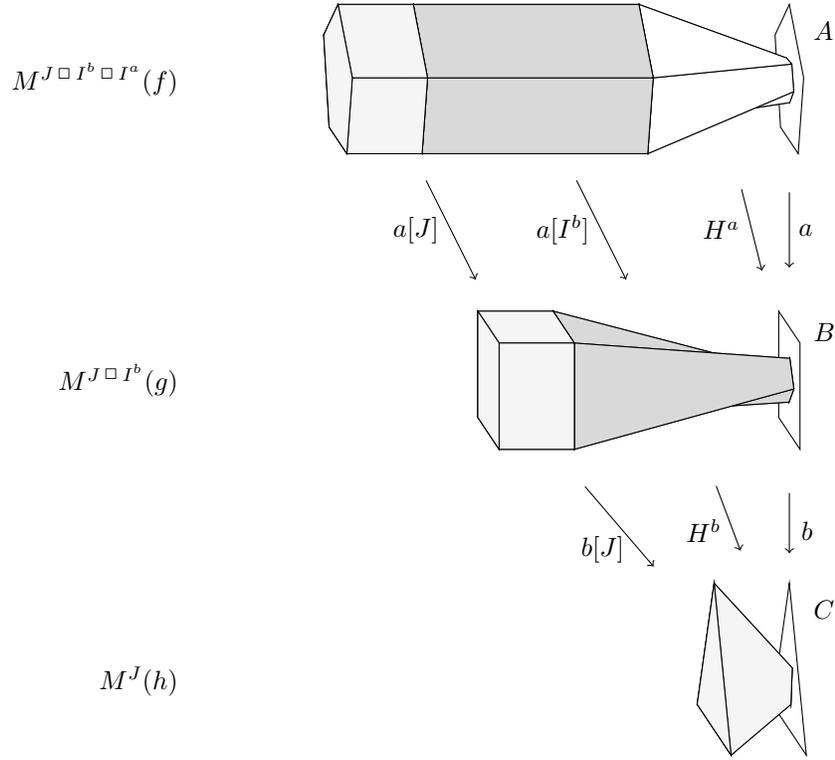}
  \caption{Composition of maps of long mapping cylinders. Shows the mapping
  cylinder construction is ``functorial'' if performed with long cylinders.}
  \label{fig:map-cylc}
\end{figure}

\subsubsection{Induced maps on telescopes}
\label{subsubsec:telescopes-stacking-composition}
Everything from~\ref{lemdef:homotopy-commutative-to-mapping-cylinders} on
transfers immediately to mapping telescopes, by gluing the parts together.  In
particular, if we have homotopy commutative squares like
in~\ref{lem:map-cyl:composition}.\eqref{eq:homotopy-commutative-square} we
obtain maps
\begin{align*}
  (H^a,a)_* &\colon \Tel^{J\mathop{\Box} I^a}(f) \rightarrow \Tel^{J}(g),\\
  (H^b,b)_* &\colon \Tel^{J\mathop{\Box} I^b}(g) \rightarrow
  \Tel^{J}(h).
\end{align*}
and their composition
\begin{equation*}
   (H^b,b)_* \mathop{\boxcircle} (H^a,a)_* \colon
   \Tel^{J\mathop{\Box}I^b \mathop{\Box} I^a} (f) \rightarrow
   \Tel^{J}(h) 
\end{equation*}
which is the same as the induced map of the stacking of the homotopy
commutative squares.
The map $(H,a)_*$ commutes with the front inclusion $\iota$
from~\ref{subsubsec:telescope-front-inclusion}.

We can specialize to $H$ being the trivial homotopy of length $0$ and $J :=
\Delta^1$.  Then we obtain the strict functoriality of $\Tel^{\Delta^1}(-)$ from
Theorem~\ref{thm:mapping-telescopes}
\eqref{item:mapping-tel:map-of-idempotents}.  

One the other hand  we can consider the square given by $f=g$ and $a=\id_A$
being homotopy commutative with trivial homotopy $\Tr\colon A[I] \rightarrow
A$, for $I$ any interval.  The resulting map $(\Tr,\id)_* \colon
\Tel^{\Delta^1 \mathop{\Box} I}(f) \rightarrow \Tel^{\Delta^1}(f)$ is induced
by the projection $\Delta^1 \mathop{\Box} I \rightarrow \Delta^1$ mapping $I$
to $I(1)\subseteq \Delta^1$.

\subsection{A homotopy criterion for maps on $\Tel(-)$}
\label{subsec:homotopy-criterion-tel}
We need a criterion when two homotopy commutative squares
\begin{equation*}
  \vcenter{\xymatrix{
    A \ar[r]^{f} \ar[d]^{a}   & A \ar[d]^{a} \\
    B \ar[r]^{g}                & B }} \qquad \text{and} \qquad
  \vcenter{\xymatrix{
    A \ar[r]^{f} \ar[d]^{\widetilde{a}}   & A \ar[d]^{\widetilde{a}} \\
    B \ar[r]^{g}                & B } }
\end{equation*}
with homotopies $H^a$ and $H^{\widetilde{a}}$ induce homotopic maps on mapping
telescopes.
Assume:

\subsubsection{Homotopy conditions}
\label{subsubsec:assumption-for-induced-map}
\begin{enumerate}
  \item $H^a$ and $H^{\widetilde{a}}$ have the same length and are indexed
    over the same interval $I = I(0,i)$.  We can arrange this by extending
    with trivial homotopies.
  \item There is a homotopy $H\colon A[J] \rightarrow B$ from $a$ to
    $\widetilde{a}$.
  \item There is ``\/$2$-homotopy'' $G\colon A[I][J] \rightarrow B$ from $H^a$
    to $H^{\widetilde{a}}$ which restricts on $I(0)\times J$ to
    $g\circ H$ and on $I(i)\times J$ to $H \circ f[J]$.
\end{enumerate}
The last condition can be visualized for $I=J= \Delta^1$ by writing $G\colon
A[\Delta^1 \times \Delta^1] \rightarrow B$ in our diagram language as
\begin{equation}\label{diag:lem:two-homotopy}
  \vcenter{\xymatrix{
  g\circ a \ar[r]^{H^{a}} \ar[d]_{g\circ H} \ar[dr] & 
    a\circ f \ar[d]^{H\circ f[J]} \\
    g\circ\widetilde{a} \ar[r]_{H^{\widetilde{a}}}  & \widetilde{a}\circ f }}.
\end{equation}

\begin{lem}[Homotopy criterion]\label{lem:2-homotopy-induce-homotopy}
  If the conditions from~\ref{subsubsec:assumption-for-induced-map} are
  satisfied, the two induced maps $(H^a, a)_*,
  (H^{\widetilde{a}}, \widetilde{a})_* \colon 
  \Tel^{\Delta^1 \mathop{\Box} I}(f) \rightarrow \Tel^{\Delta^1}(g)$ are
  homotopic. The homotopy is $(G, H)_*\colon 
  \Tel^{\Delta^1 \mathop{\Box} I}(f)[J] \rightarrow \Tel^{\Delta^1}(g)$.
\end{lem}

\begin{proof}
  Interpreting $G$ as homotopy from $g\circ H$ to $H\circ f[J]$ gives a
  homotopy commutative square
  \begin{equation*}
    \xymatrix{
      A[J] \ar[r]^{f[J]} \ar[d]^{H}   & A[J] \ar[d]^{H} \\
      B \ar[r]^{g}                & B } 
  \end{equation*}
  with homotopy $G\colon (A[J]) [I] \rightarrow B$.
  We obtain a map
  $(G,H)_*\colon \Tel^{\Delta^1\mathop{\Box} I}(f[J]) \rightarrow
  \Tel^{\Delta^1}(g)$. As the telescope is a colimit it commutes with
  adjoining an interval, hence we can write the domain
  of the induced map as $\Tel^{\Delta^1\mathop{\Box} I}(f)[J]$.  Therefore,
  $(G,H)_*$ is a homotopy.  We leave it to the reader to check that it is the
  desired one.
\end{proof}

\begin{rem}
  Thanks to
  Lemma~\ref{lem:2-homotopy-induce-homotopy} we only need to provide
  diagrams
  like~\ref{subsubsec:assumption-for-induced-map}.\eqref{diag:lem:two-homotopy}
  to prove that maps on
  telescopes are homotopic.  To simplify the diagram we will usually assume
  that all intervals have length $1$ and are ordered, cf.
  Remark~\ref{subsubsec:variations-longer-homotopies}.
\end{rem}

\subsection{On
Theorem~\ref{thm:mapping-telescopes}.(\ref{item:mapping-tel:homotopy-id-give-hpty-equiv-tels})
}
\label{subsec:on-thm:id-give-htpy-equiv}

\begin{lem}[Homotopic maps give equivalent telescopes]
  \label{lem:homotopic-maps-give-homotopic-telescopes}
  Let $f,g \colon A \rightarrow A$ be homotopic maps. Then
  $\Tel^{\Delta^1}(f)$ and $\Tel^{\Delta^1}(g)$ are homotopy equivalent.  The homotopy
  equivalences are relative to the front inclusions.
\end{lem}

\begin{proof}
  Let $H\colon A[I] \rightarrow A$ be the homotopy from $g$ to $f$.  We get
  two homotopy commutative squares
  \begin{equation*}
    \vcenter{\xymatrix{
    A \ar[r]^f \ar[d]^{\id}  & A \ar[d]^{\id} \\ A \ar[r]^g & A}}
    \qquad\text{ and }\qquad
    \vcenter{\xymatrix{
    A \ar[r]^g \ar[d]^{\id}  & A \ar[d]^{\id} \\ A \ar[r]^f & A}}
  \end{equation*}
  with homotopies $H\colon A[I]\rightarrow A$ and $\overline{H}\colon
  A[\overline{I}]\rightarrow A$ (where the latter is the inverse homotopy,
  cf.~Section~\ref{sec:concat-of-homotopies}).  This gives maps $(H,\id)_*
  \colon \Tel^{\Delta^1\mathop{\Box} I}(f) \rightarrow
  \Tel^{\Delta^1}(g)$ and 
  $(\overline{H},\id)_* \colon \Tel^{\Delta^1\mathop{\Box} \overline{I}}(g)
  \rightarrow \Tel^{\Delta^1}(f)$.  The
  composition
  \begin{equation*}
    (\overline{H},\id)_* \boxcircle (H,\id)_* \colon
    \Tel^{\Delta^1\mathop{\Box} \overline{I} \mathop{\Box} I}(f) \rightarrow 
    \Tel^{\Delta^1}(f)
  \end{equation*}
  (cf.~\ref{subsubsec:telescopes-stacking-composition})
  is induced by the homotopy commutative square
  \begin{equation*}
    \xymatrix{A \ar[r]^f \ar[d]^{\id} & A \ar[d]^{\id} \\
    A \ar[r]^f & A}
  \end{equation*}
  with homotopy $\overline{H} \mathop{\Box} H \colon A[\overline{I}
  \mathop{\Box} I] \rightarrow A$.  By
  Lemma~\ref{lem:concat-homotopy-and-inverse}, that homotopy is homotopic relative
  endpoints to the trivial homotopy, hence
  Lemma~\ref{lem:2-homotopy-induce-homotopy} shows that $(\overline{H},\id)_*
  \boxcircle (H,\id)_*$ is homotopic to $(\Tr,\id)_*$.  The same holds for the
  other composition.  

  As $(\Tr,\id)_*\colon \Tel^{\Delta^1 \mathop{\Box} I}(f) \rightarrow
  \Tel^{\Delta^1}(f)$ is a homotopy equivalence induced by the projection
  $\Delta^1 \mathop{\Box} I \rightarrow \Delta^1$ we
  obtain two homotopy commutative triangles
  \begin{equation*}
    \vcenter{\xymatrix{
    \Tel^{\Delta^1 \mathop{\Box} I}(f) \ar[r]^{\simeq} \ar[d]_{(H,\id)_*} &
    \Tel^{\Delta^1}(f)  \ar[dl]^{\varphi}  \\
    \Tel^{\Delta^1}(g)} }
    \quad \text{ and } \quad
    \vcenter{\xymatrix{
    \Tel^{\Delta^1 \mathop{\Box} \overline{I}}(g) \ar[r]^{\simeq}
    \ar[d]_{(\overline{H},\id)_*} &
    \Tel^{\Delta^1}(g)  \ar[dl]^{\psi}  \\
    \Tel^{\Delta^1}(f)} }
  \end{equation*}
  where $\varphi$ is defined using a chosen homotopy inverse of the
  horizontal map and $\psi$ similarly.  We claim both compositions of these
  maps are homotopic to the identity.  Together with these triangles we obtain a
  large diagram
  \begin{equation*}
    \xymatrix{
    \Tel^{\Delta^1 \mathop{\Box} \overline{I} \mathop{\Box} I}(f) 
            \ar[r]^{\simeq} \ar[d]_{(H,\id)_*} &
      \Tel^{\Delta^1 \mathop{\Box} I}(f) \ar[r]^{\simeq} \ar[d]_{(H,\id)_*} &
      \Tel^{\Delta^1}(f)  \ar[dl]^{\varphi}  \\
    \Tel^{\Delta^1 \mathop{\Box} \overline{I}}(g) \ar[r]^{\simeq}
    \ar[d]_{(\overline{H},\id)_*} &
    \Tel^{\Delta^1}(g)  \ar[dl]^{\psi}  \\
    \Tel^{\Delta^1}(f)}
  \end{equation*}
  where the square is strictly commutative.  The left vertical composition is
  the composition $(\overline{H},\id)_* \boxcircle (H,\id)_* $ which is
  homotopic to $(\Tr,\id)_*$, which is exactly the composition of the upper
  horizontal maps.  It follows that $\psi\circ \varphi\simeq \id$ and
  similarly for the other composition.

  As the homotopy inverses in the definition of $\psi$ and $\varphi$ can be
  chosen to respect the front inclusion by
  Lemma~\ref{lem:mapping-cylinders-different-length-equivalent} and all other
  maps and homotopies are relative to it $\varphi$ is a homotopy equivalence
  relative to the front inclusion.
\end{proof}


\subsection{On Theorem~\ref{thm:mapping-telescopes}.%
(\ref{item:mapping-tel:idempot-induces-id-on-tel})}

\subsubsection{The shift map}
Recall that $\Tel^I(f)$ is a quotient of $\coprod_{n\in \Nat} M^I(f)$. The map
taking the $n$th component to the $(n+1)$st component is compatible with the
quotient, hence induces a map $\Tel^I(f) \rightarrow \Tel^I(f)$, which we will
call the \emph{shift map} and denote it by~$\sh$.

\begin{lem}\label{lem:tel-shift}
  Let $I$ be an interval, $f\colon A \rightarrow A$ a self-map.  The maps
  $\sh$ and $(\Tr,f)_*$ from $\Tel^I(f)$ to $\Tel^I(f)$ are homotopy
  inverse:
  \begin{equation*}
    (\Tr, f)_* \circ \sh = \sh \circ (\Tr,f)_* \simeq \id \colon
      \Tel^I(f) \rightarrow \Tel^I(f)
  \end{equation*}
\end{lem}

\begin{proof}[Sketch of proof] 
  The first equality is clear.  For the homotopy one restricts the map of
  telescopes to a map $M^I(f) \rightarrow M^I(f) \cup_A M^I(f)$, which maps
  into the second summand.  Then one can construct a simplicial homotopy from
  this map to the map ``inclusion of the first summand''.  Namely, the two
  inclusions $I \rightarrow I \mathop{\Box} I$ give two maps $A[I] \rightarrow
  A[I \mathop{\Box}I]$ which are homotopic by ``sliding''.  The desired
  homotopy arises from this homotopy.  The details are left to the reader.
\end{proof}

\subsubsection{Telescopes of coherent homotopy idempotents}
In the following we consider coherent homotopy idempotents.

\begin{lem}\label{lem:map-tel-homotopy-idemp-induced-id-on-tel}
  Let $\eta\colon K \rightarrow K$ be a \emph{coherent} homotopy idempotent in
  $\mathcal{C}^G$.  Then the induced map $\eta_* = (\Tr,\eta)_* \colon
  \Tel^{\Delta^1}(\eta) \rightarrow \Tel^{\Delta^1}(\eta)$ is homotopic to the
  identity.
\end{lem}

\begin{proof} 
  We show $\eta_* \circ \eta_* \simeq \eta_*$, then by
  Lemma~\ref{lem:tel-shift} we have $\eta_*\circ \sh\simeq \id$ and
  therefore, $\eta_* \simeq \eta_* \circ \eta_* \circ \sh \simeq \eta_* \circ
  \sh \simeq \id$ and we are done. 
  
  Assume that $H\colon A[I] \rightarrow A$ is the homotopy from $\eta^2$ to
  $\eta$ and for simplicity assume $I =\Delta^1$.
  As $\eta$ is coherent we have a diagram
  \begin{equation*}
    \xymatrix{ 
    \bullet \ar[r]^{\eta\circ H} \ar[d]_{H\circ \eta[I]} \ar[dr]|{X} &
     \bullet \ar[d]^H \\ \bullet \ar[r]_H & \bullet}.
  \end{equation*}
  Using this diagram we construct the following diagram, which gives a
  $2$-homotopy $G$.
  \begin{equation*}
    \xymatrix@C=44pt{
    \eta^3 \ar[r]|{\Tr} \ar[d]|{\eta\circ H} \ar[dr]|{X} & \eta^3 \ar[d]|{X} &
    \eta^3 \ar[l]|{\Tr} \ar[d]|{H\circ \eta[I]} \ar[ld]|{X} \\
    \eta^2 \ar[r]|{H} & \eta & \eta^2 \ar[l]|{H} \\
    \eta^2 \ar[r]|{\Tr} \ar[u]|{\eta\circ\Tr} \ar[ur]|{H} & \eta^2 \ar[u]|{H} &
    \eta^2 \ar[l]|{\Tr} \ar[u]|{\Tr\circ\eta[I]} \ar[ul]|{H}}
  \end{equation*}
  Lemma~\ref{lem:2-homotopy-induce-homotopy} shows that $(G,H)_*$ is a homotopy
  from $(\Tr, \eta^2)_*$ to $(\Tr,\eta)_*$, which are maps $\Tel^{\Delta^1\mathop{\Box} I
  \mathop{\Box} \overline{I}}(\eta) \rightarrow  \Tel^{\Delta^1}(\eta)$.  But
  $ \Delta^1 \mathop{\Box} I \mathop{\Box} \overline{I} \rightarrow \Delta^1$
  induces a homotopy equivalence on telescopes such that the triangle for
  $\eta_*$ below, as well as one for $\eta^2_*$ commute strictly.
  \begin{equation*}
    \xymatrix{
    \Tel^{\Delta^1 \mathop{\Box} I \mathop{\Box} \overline{I}}(\eta) 
            \ar[r]^-{\simeq} \ar[d]_{\eta_*} &
      \Tel^{\Delta^1}(\eta)  \ar[dl]^{\eta_*}  \\
    \Tel^{\Delta^1}(\eta)}
  \end{equation*}
  Therefore, the maps on the cylinder of the same lengths are homotopic, too.
\end{proof}

\subsection{On Theorem~\ref{thm:mapping-telescopes}.%
(\ref{item:mapping-tel:hpty-retract})}
\label{subsec:on-thm-hpty-retract}

\begin{lem}\label{lem:map-cyl:homotopy-idempotent-retraction}
  Let $\eta\colon K \rightarrow K$ be a coherent homotopy idempotent in
  $\mathcal{C}^G$.  Then there is a map $c\colon \Tel^{\Delta^1}(\eta)
  \rightarrow K$ such that the composition $\iota\circ c$ with the inclusion
  $\iota\colon K \rightarrow \Tel^{\Delta^1}(\eta)$ is homotopic to the
  identity on $\Tel^{\Delta^1}(\eta)$, whereas the other composition $c \circ \iota$
  is homotopic to $\eta\colon K \rightarrow K$.
\end{lem}

\begin{proof}
  Let $H \colon A[I] \rightarrow A$ be the homotopy from $\eta^2$ to $\eta$.
  We obtain two homotopy commutative squares
  \begin{equation*}
    \vcenter{\xymatrix{
    K \ar[r]^{\id} \ar[d]^{\eta} & K \ar[d]^{\eta} \\ K \ar[r]^{\eta} & K } }
    \qquad\text{ and }\qquad
    \vcenter{\xymatrix{
    K \ar[r]^{\eta} \ar[d]^{\eta} & K \ar[d]^{\eta} \\ K \ar[r]^{\id} & K } }
  \end{equation*}
  with homotopies $H\colon A[I] \rightarrow A$ and $\overline{H}\colon
  A[\,\overline{I}\,] \rightarrow A$.  Hence, we obtain two induced maps
  \begin{align*}
    (H,\eta)_* &\colon \Tel^{\Delta^1 \mathop{\Box} I}(\id_K) \rightarrow
    \Tel^{\Delta^1}(\eta) \\
    (\overline{H},\eta)_* &\colon \Tel^{\Delta^1 \mathop{\Box}
    \overline{I}}(\eta) \rightarrow \Tel^{\Delta^1}(\id_K) .
  \end{align*}
  Consider the composition (\ref{def:composition-map-of-cylinders})
  \begin{equation*}
    (H,\eta)_* \boxcircle (\overline{H},\eta)_* \colon
    \Tel^{\Delta^1 \mathop{\Box} I \mathop{\Box} \overline{I}}(\eta)
    \rightarrow
    \Tel^{\Delta^1}(\eta)
  \end{equation*}
  which by Lemma~\ref{lem:map-cyl:composition} is equal to $(H\circ \eta[I]
  \mathop{\Box} \eta\circ\overline{H}, \eta^2)_*$.  As the $\eta$
  is coherent we have a map
  $A[I\times I] \rightarrow A$ which is on the boundary of $I^2$ as shown
  below.
  \begin{equation}\label{diag:map-tel-coherent}
    \begin{split}\xymatrix{  
    \bullet \ar[r]^{\eta\circ H} \ar[d]_{H\circ \eta[I]} \ar[dr] &
    \bullet \ar[d]^H \\ \bullet \ar[r]_H & \bullet}\end{split}
  \end{equation}
  Thus by putting two copies of the above square together as shown below we
  obtain a $2$-homotopy~$G$  from $H\circ \eta[I] \mathop{\Box}
  \eta\circ\overline{H}$ to $\Tr$ as shown below.
  \begin{equation*}
    \xymatrix@C=44pt{
    \eta^3 \ar[r]|{H\circ \eta[I]} 
    \ar[d]|{\eta \circ H} &
    \eta^2 \ar[d]|{H} &
    \eta^3 \ar[l]|{\eta\circ H} 
    \ar[d]|{H\circ \eta[I]} \\
    \eta^2 \ar[r]|{H} & \eta & \eta^2 \ar[l]|{H} \\
    \eta^2 \ar[u]|{\eta\circ\Tr} \ar[ur]|{H} \ar[r]|{\Tr} &
    \eta^2 \ar[u]|{H} &
    \eta^2 \ar[u]|{\Tr\circ\eta[I]} \ar[ul]|{H} \ar[l]|{\Tr} } 
  \end{equation*}
  Lemma~\ref{lem:2-homotopy-induce-homotopy} $(G,H)_*$ provides a homotopy from
  the composition $(H,\eta)_* \boxcircle (\overline{H},\eta)_*$ to the map
  $(\Tr,\eta)_*$.  Similarly, the other composition $(\overline{H},\eta)_*
  \boxcircle (H,\eta)_*$ is homotopic to $(\Tr,\eta)_*\colon
  \Tel^{\Delta^1\mathop{\Box} \overline{I} \mathop{\Box} I}(\id_K)\rightarrow
  \Tel^{\Delta^1}(\id_K)$ using the $2$-homotopy
  \begin{equation*}
    \xymatrix@C=44pt{
    \eta^2 \ar[d]|H & \eta^3 \ar[l]|{\eta\circ H} \ar[dr]|{X} \ar[d]|{X}
    \ar[dl]|{X} \ar[r]|{H\circ \eta[I]} & \eta^2 \ar[d]|H \\
    \eta & \eta \ar[r]|{\Tr} \ar[l]|{\Tr} & \eta}
  \end{equation*}
  where $X$ is the diagonal in diagram~\eqref{diag:map-tel-coherent} and the
  upper left and right triangles are also from~\eqref{diag:map-tel-coherent}.

  Now we make $(H,\eta)_*$ and $(\overline{H},\eta)_*$ into maps of telescopes
  of the same length as in the proof of
  Lemma~\ref{lem:homotopic-maps-give-homotopic-telescopes}.  Define
  $\widetilde{c}$ and $\widetilde{\iota}$ by choosing a homotopy inverse
  in the top row of the following diagrams.
  \begin{equation*}
    \vcenter{\xymatrix{
    \Tel^{\Delta^1 \mathop{\Box} \overline{I}}(\eta) \ar[r]^{\simeq}
    \ar[d]_{(\overline{H},\eta)_*} &
    \Tel^{\Delta^1}(\eta)  \ar[dl]^{\widetilde{c}}  \\
    \Tel^{\Delta^1}(\id_K)} }
    \quad \text{ and } \quad
    \vcenter{\xymatrix{
    \Tel^{\Delta^1 \mathop{\Box} I}(\id_K) \ar[r]^{\simeq} \ar[d]_{(H,\eta)_*} &
    \Tel^{\Delta^1}(\id_K)  \ar[dl]^{\widetilde{\iota}}  \\
    \Tel^{\Delta^1}(\eta)} }
  \end{equation*}
  A similar argument as at the end of the proof of
  Lemma~\ref{lem:homotopic-maps-give-homotopic-telescopes} shows that $\widetilde{c}\circ  \widetilde{\iota}$ is homotopic to
  $\eta_*\colon \Tel^{\Delta^1}(\id_K) \rightarrow \Tel^{\Delta^1}(\id_K)$ and
  $\widetilde{\iota}\circ \widetilde{c}$ is homotopic to $\eta_*\colon
  \Tel^{\Delta^1}(\eta) \rightarrow \Tel^{\Delta^1}(\eta)$.  
  Lemma~\ref{lem:map-tel-homotopy-idemp-induced-id-on-tel} shows that on
  $\Tel^{\Delta^1}(\eta)$  the map $\eta_*$ is homotopic to the identity.

  Now $\iota_{\id_K}\colon K\rightarrow \Tel^{\Delta^1}(\id_K)$ is a homotopy
  equivalence and even an inclusion for a deformation retraction $\pr$ by
  Lemma~\ref{lem:tel-homotopy-equiv-to-obj}. Define $c$ as the composition $ \pr \circ
  \widetilde{c}\colon \Tel^{\Delta^1}(\eta)\rightarrow K$, note
  $\widetilde{\iota}\circ \iota_{\id_K} =
  \iota_{\eta}$.  Then $\iota\circ c =\widetilde{\iota}\circ \iota_{\id_K} 
  \circ\pr \circ \widetilde{c}$ is homotopic to $\widetilde{\iota} \circ
  \widetilde{c}$ and hence to $\id_{\Tel^{\Delta^1}(\eta)}$
  and $c \circ \iota = \pr \circ \widetilde{c} \circ\widetilde{\iota}\circ
  \iota_{\id_K}$is homotopic to $\eta\colon K \rightarrow K$.  This shows the lemma.
\end{proof}

\bibliographystyle{halpha}
\bibliography{ullmann-controlled-simplicial-rings}
\end{document}